\documentclass[12pt,a4paper,reqno,oneside]{amsart}
\usepackage{t1enc}
\usepackage{times}
\usepackage{amssymb}
\usepackage[mathscr]{euscript}
\usepackage{graphicx}
\usepackage{hyperref}

\usepackage{comment}

\usepackage{float}
\usepackage{epstopdf}

\usepackage{tikz}

\usepackage{dsfont}

\numberwithin{equation}{section}

\newtheorem{claim}{Claim}[section]
\newtheorem{thm}{Theorem}[section]

\newtheorem{lem}[thm]{Lemma}
\newtheorem{cor}[thm]{Corollary}
\newtheorem{ass}[thm]{Assumption}
\theoremstyle{remark}
\newtheorem{rem}[thm]{Remark}
\theoremstyle{definition}
\newtheorem{defn}[thm]{Definition}

\newcommand{\eps}{\varepsilon}
\newcommand{\ind}{\mathds{1}} % complex unit
\newcommand{\Hfct}{\mathcal{H}} % H-function
\newcommand{\R}{\mathbb{R}} % real numbers
\newcommand{\PP}{\mathrm{P}} % probability
\newcommand{\BB}{\mathbb B} %M_2 valued fBM
\newcommand{\NN}{\mathbb N} %Natural Numbers
\newcommand{\F}{\mathcal F} % filtration
\newcommand{\Wi}{\mathcal W} % Wiener transform
\newcommand{\N}{\mathcal N} % class of ahead looking processes
\newcommand{\<}{\langle}
\renewcommand{\>}{\rangle}
\newcommand{\DM}{\mathcal{D}}

\DeclareMathOperator{\card}{card}
\DeclareMathOperator{\dom}{dom}
\DeclareMathOperator{\mo}{mod}
\DeclareMathOperator{\Var}{Var}

\begin{document}

\title{Strong Uniqueness of Singular Stochastic Delay Equations}

\date{\today}

\author[D. Ba\~{n}os]{David Ba\~{n}os}
\address{D. Ba\~{n}os: Department of Mathematics, University of Oslo, P.O. Box 1053, Blindern, N--0316 Oslo, Norway}
\email{davidru@math.uio.no}
\author[H. H. Haferkorn]{Hannes Hagen Haferkorn}
\address{H. H. Haferkorn: Department of Mathematics, University of Oslo, Moltke Moes vei 35, P.O. Box 1053 Blindern, 0316 Oslo, Norway.}
\email{hanneshh@math.uio.no}
\author[F. Proske]{Frank Proske}
\address{F. Proske: CMA, Department of Mathematics, University of Oslo, Moltke Moes vei 35, P.O. Box 1053 Blindern, 0316 Oslo, Norway.}
\email{proske@math.uio.no} 
 

%\thanks{Thank you very much} 

\begin{abstract}
In this article we introduce a new method for the construction of unique strong solutions of a larger class of stochastic delay equations driven by a discontinuous drift vector field and a Wiener process. The results obtained in this paper can be regarded as an infinite-dimensional generalization of those of A. Y. Veretennikov \cite{Ver79} in the case of certain stochastic delay equations with irregular drift coefficients. The approach proposed in this work rests on Malliavin calculus and arguments of a ``local time variational calculus'', which may also be used to study other types of stochastic equations as e.g. functional It\^{o}-stochastic differential equations in connection with path-dependent Kolmogorov equations \cite{FZ16}.
\end{abstract}

\maketitle

\section{Introduction}

Let us consider the ordinary differential equation ODE

\begin{align}\label{3MotivatingODE}
 \frac{d}{dt}x(t)=b(t,x(t)),\quad x(0)=x\in\R^d,\quad 0\leq t\leq T,
\end{align}

where $b:[0,T]\times\R^d\rightarrow\R^d$ is a Borel-measurable vector field.\\
Using Picard iteration, it is well-known that if, e.g. $b$ is Lipschitz continuous and satisfies a linear growth condition, the ODE \eqref{3MotivatingODE} has a unique solution $x$. However, if the vector field $b$ is not Lipschitzian, then uniqueness or even existence of solutions to \eqref{3MotivatingODE} may fail.\\
\par
On the other hand, the situation changes dramatically, if equation \eqref{3MotivatingODE} is superposed by, for instance, a small Brownian noise, that is if $x$ is supposed to solve the stochastic differential equation (SDE)

\begin{align}\label{3MotivatingSDE}
 x(t)=x+\int_0^tb(s,x(s))ds+\eps W(t),
\end{align}

where $W$ is a $d$-dimensional Brownian motion and $\eps>0$. In fact, a milestone result by Zvonkin \cite{Zvon74} shows (for $d=1$) that the SDE \eqref{3MotivatingSDE} admits the existence of a unique strong solution for merely bounded Borel measurable $b$, regardless of how small $\eps$ is chosen. Subsequently, the latter result was generalized by Veretennikov \cite{Ver79} to the multidimensional case.\\
The intuition with respect to the regularization effect of the Brownian noise in \eqref{3MotivatingSDE} may be explained by the roughness of the Brownian paths which avoid ``longer stays'' at singularities of the vector field $b$.\\
Other, more recent, important results in this direction in the finite dimensional case were obtained by Krylov, R\"{o}ckner \cite{KR05}, Krylov \cite{Kry80}, Gy\"{o}ngy, Martinez \cite{GyM01}. Let us also mention here that the results of Zvonkin \cite{Zvon74} and Veretennikov \cite{Ver79} for bounded measurable drift coefficients $b$ were recently generalized by Da Prato, Flandoli, Priola, R\"{o}ckner \cite{DPFPR13} to the case of mild solutions to SPDEs of the form

\begin{align}\label{3MotivatingSPDE}
 x(t)=x+\int_0^tAx(s)+b(s,x(s))ds+\int_0^tZ(t)d\mathbb{W}(t),
\end{align}

where $A$ is a (non-zero) densely defined linear operator, $Z$ is a Hilbert-Schmidt operator valued function on $[0,T]$ and $\mathbb{W}$ is a cylindrical Brownian motion. Here in this striking work, the authors employ solutions of infinite dimensional Kolmogorov equations for the construction of unique strong solutions to \eqref{3MotivatingSPDE}.\\
\par
Motivated by the above mentioned results, one may ask whether the regularization effect of a Brownian or related noise also applies to other types of equations, as for example the delay equation of the form

\begin{align}\label{3MotivatingODDE}
 x(t)=\eta(0)+\int_0^tb(x_s)ds,\quad x_0=\eta,
\end{align}

where, for $s\in[0,T]$, $x_s$ denotes the past \emph{segment} of length $r>0$ of the process $x$ in $s$, i.e. $$x_s=\{u\mapsto x(s+u),\,u\in[-r,0]\},$$
which is seen as an element of a suitably chosen space (e.g. $C([-r,0])$ or $L^p([-r,0])$, etc.) and the initial path $\eta$ is an element of that same space.\\
Inspired by Zvonkin \cite{Zvon74} and Veretennikov \cite{Ver79}, it would be natural here to regularize the deterministic delay equation  \eqref{3MotivatingODDE} by adding a Brownian noise, that is to consider the SDE

\begin{align}\label{3MotivatingSDDE1}
 x(t)=\eta(0)+\int_0^tb(x_s)ds+W(t),\quad x_0=\eta.
\end{align}

Unfortunately, if $b$ is irregular (non-Lipschitzian) then \eqref{3MotivatingSDDE1} may not allow for the existence of a strong solution. An indication for this -- however in the related case of functional SDEs -- may be the celebrated counter-example of Tsirelson in \cite{Ts75} where the author proved the non-existence of a strong solution of a certain functional SDE with bounded and continuous drift coefficient. A possible explanation for this phenomenon could be the ``mismatching'' of the dimensions of the driving noise (in $\R^d$) and the infinite dimensional space (i.e. path space) on which the vector field $b$ is defined. In order to overcome this mismatching, one could try to compensate for this ``dimension gap'' by distorting \eqref{3MotivatingODDE} in the spirit of Zvonkin \cite{Zvon74} and Veretennikov \cite{Ver79} by means of a noise with dimension corresponding to that of the domain of $b$. For this purpose, one could recast the delay equation \eqref{3MotivatingODDE} in segment form and superpose it by a noise with values in the space that was chosen for the segments. However, it is not obvious at all, how this could be done without changing the delay character of the equation. More precisely, the resulting equation may admit for solutions in the space where the segments are defined on, but may itself not be a segment process corresponding to a solution of a (stochastic) delay equation.\\
\par
In order to restore well-posedness of \eqref{3MotivatingODDE} in the sense of unique strong solutions, one could resort to an alternative approach to the previous one by enlarging the perturbation noise from the ``inside'' of the vector field \emph{instead} of from the ``outside''. In other words, one could consider the stochastic delay equation (SDDE) of the following form:

\begin{align}\label{3MotivatingSDDE2}
 x(t)=\eta(0)+\int_0^tb(x_s+\mathbb{B}(s))ds+W(t),\quad x_0=\eta,
\end{align}

where $\mathbb{B}$ is a rough stochastic process with values in the domain of $b$. This is the type of equation that we are studying in this paper. The noise $\BB$ involved in equation \eqref{3MotivatingSDDE2} could be principally correlated with $W$ and Markovian as e.g. a $Q$-Wiener process. However, in this paper we assume that $\BB$ is a stationary Gaussian process of the form

\begin{align}\label{3Bprocess}
  \BB(t):=\sum_{n=1}^\infty w_nB^{H_n}(t)e_n,
\end{align}

where $B^{H_n}$, $n\geq1$ are independent one-dimensional fractional Brwonian motions with Hurst parameters $H_n\in(0,\frac{1}{2})$, $(w_n)_{n\geq1}$ is a sequence in $\ell^1$ and $(e_n)_{n\geq1}$ is an orthonormal basis of the state space (Hilbert space) of the segments $x_s$.\\
\par
The objective of this paper is the construction of unique strong solutions to the SDDE \eqref{3MotivatingSDDE2} in the case of a larger class of merely measurable and genuinely infinite-dimensional vector-fields $b$. To the best of our knowledge, the results obtained in this article are the first in the literature dealing with strong solutions of singular SDDE's in the sense of SDDE's with non-Lipschitzian or discontinuous coefficients.\\
Our construction technique with respect to strong solutions, which does not rely on the Yamada-Watanabe principle and which therefore considerably differs from those of the above mentioned authors, is based on Malliavin calculus and ideas in connection with a ``local time variational calculus''. More precisely, we approximate the singular vector field $b$ in \eqref{3MotivatingSDDE2} by a sequence of Lipschitz continuous vector fields $b^n$ on the state space of the segments. Then, we show that the sequence of strong solutions $x^n(t)$ to the SDDE's

\begin{align}\label{3MotivatingSDDE3}
 x^n(t)=\eta(0)+\int_0^tb^n(x^n_s+\mathbb{B}(s))ds+W(t),\quad x^n_0=\eta,
\end{align}

converges strongly in $L^2(\Omega)$ to a solution to \eqref{3MotivatingSDDE2}. In proving this, we apply a compactness criterion for square integrable functionals from Malliavin calculus \cite{DPMN92} to the sequence $x^n(t)$, $n\geq1$ in combination with an argument based on ``local time variational calculus''. See \cite{BanNilPro15}, where the authors proved strong uniqueness of singular SDE's driven by a fractional Brownian motion with Hurst parameters $H\in(0,\frac{1}{2})$. See also \cite{BOPP17} in the case of generalized vector fields and \cite{MBP10} in the Markovian setting of a Wiener process. As for this approach, we also mention a series of other articles in the Wiener and L\'{e}vy process case and in the Hilbert space setting \cite{Ban17,Ban17.2,FNP.13,HaaPros.14,MMNPZ10,MNP14}.\\
Further, we also want to point out the following characteristic feature of our paper in view of the work \cite{FNP.13}: In contrast to \cite{FNP.13}, we do not employ an infinite dimensional compactness criterion in the direction of the driving noise $(W,\BB)$. In fact, we only use a finite-dimensional compactness criterion with respect to $W$ (see, e.g. \cite{MMNPZ10}) for the construction of a strong solution to \eqref{3MotivatingSDDE2}. The latter technique presented in this article also exhibits the advantage of the study of SDDE's driven by certain types of L\'{e}vy processes.\\
\par
Finally let us mention a possible alternative method regarding the construction of strong solutions to singular SDDE's: In \cite{FZ16} the authors analyze path-dependent Kolmogorov-equations associated with solutions to functional It\^{o} SDE's. In this context it is conceivable that the framework developed by these authors could be extended to the case of SDDE's and be employed in connection with the ``It\^{o}-Tanaka-Zvonkin trick'': The drift part in \eqref{3MotivatingSDDE2} could be expressed, as e.g. in \cite{DPFPR13}, by means of a ``more regular'' term by using solutions of path-dependent Kolmogorov equations. Using the latter could enable one to establish strong uniqueness of such equations. See \cite{DPFPR13} and \cite{FNP.13}. However, such a task seems very difficult, since there are (besides other complications) no known suitable a priory estimates of solutions of path-dependent Kolmogorov equations associated with singular SDDE's in the sense of \cite{DPFPR13}.\\
\par
Our article is organized as follows: In Section 2 we give the mathematical framework of the paper. Section 3 is devoted to the construction of unique strong solutions of singular SDDE's based on the compactness criterion in \cite{DPMN92}.

\section{Preliminaries}

In this section we introduce the mathematical framework of our paper, which we aim at using in Section 3 for the construction of strong solutions to singular SDDE's. See \cite{BanNilPro15}.\\
\par
Let $0\leq\theta<t\leq T$ and $m\in\NN$. We define $$\Delta^m_{\theta,t}:=\{s_1,\dots,s_m\in\R:\, \theta\leq s_m<\dots<s_1\leq t\}.$$
The following lemma gives a representation formula for products of integrals over such sets $\Delta^m_{\theta,t}$, $\Delta^n_{\theta,t}$.

\begin{lem}\label{3IntegralProductShuffle}
 Let $m,n\in\NN$, $S(m,n)$ denote the set of shuffle permutations $\sigma:\{1,\dots,m+n\}\rightarrow\{1,\dots,m+n\}$, s.t. $\sigma(1)<\dots<\sigma(m)$ and $\sigma(m+1)<\dots<\sigma(m+n)$. Moreover, let $f:\Delta^m_{\theta,t}\rightarrow\R$ and $g:\Delta^n_{\theta,t}\rightarrow\R$ be Lebesgue integrable. Then
 \begin{align}
 \begin{split}
  &\int_{\Delta^m_{\theta,t}}f(s_1,\dots,s_m)ds_m\dots ds_1\int_{\Delta^n_{\theta,t}}g(s_{m+1},\dots,s_{m+n})ds_{m+n}\dots ds_{m+1}\\
  &=\sum_{\sigma\in S(m,n)}\int_{\Delta^{m+n}_{\theta,t}}f(s_{\sigma(1)},\dots,s_{\sigma(m)})g(s_{\sigma(m+1)},\dots,s_{\sigma(m+n)})ds_m\dots ds_1.
 \end{split}
 \end{align}
\end{lem}

\begin{proof}
 It holds $$\Delta^m_{\theta,t}\times\Delta^n_{\theta,t}=\dot\bigcup_{\sigma\in S(m,n)}\{(w_1,\dots,w_{m+n}):\theta<w_{\sigma(m+n)}<\dots<w_{\sigma(1)}<t\}\cup\N$$ for some Lebesgue nullset $\N$. Therefore,
 \begin{align*}
  &\int_{\Delta^m_{\theta,t}}f(s_1,\dots,s_m)ds_m\dots ds_1\int_{\Delta^n_{\theta,t}}g(s_{m+1},\dots,s_{m+n})ds_{m+n}\dots ds_{m+1}\\
  &=\int_{\dot\bigcup_{\sigma\in S(m,n)}\{\dots\}\cup\N}f(w_1,\dots,w_m)g(w_{m+1},\dots,w_{m+n})dw_{m+n}\dots dw_1\\
  &=\sum_{\sigma\in S(m,n)}\int_{\Delta^{m+n}_{\theta,t}}f(w_1,\dots,w_m)g(w_{m+1},\dots,w_{m+n})dw_{\sigma(m+n)}\dots dw_{\sigma(1)},
 \end{align*}
where we reordered the integrals by application of Fubini's theorem in the last step. Defining $\tilde w_i:=w_{\sigma(i)}$ and noting that $\sigma\in S(m,n)\Leftrightarrow\sigma^{-1}\in S(m,n)$, we eventually get
\begin{align*}
  &\int_{\Delta^m_{\theta,t}}f(s_1,\dots,s_m)ds_m\dots ds_1\int_{\Delta^n_{\theta,t}}g(s_{m+1},\dots,s_{m+n})ds_{m+n}\dots ds_{m+1}\\
  &=\sum_{\sigma\in S(m,n)}\int_{\Delta^{m+n}_{\theta,t}}f(\tilde w_{\sigma^{-1}(1)},\dots,\tilde w_{\sigma^{-1}(m)})g(\tilde w_{\sigma^{-1}(m+1)},\dots,\tilde w_{\sigma^{-1}(m+n)})d\tilde w_{m+n}\dots d\tilde w_{1}\\
  &=\sum_{\sigma\in S(m,n)}\int_{\Delta^{m+n}_{\theta,t}}f(\tilde w_{\sigma(1)},\dots,\tilde w_{\sigma(m)})g(\tilde w_{\sigma(m+1)},\dots,\tilde w_{\sigma(m+n)})d\tilde w_{m+n}\dots d\tilde w_{1},
 \end{align*}
 which ends the proof.
\end{proof}

\begin{cor}\label{3DeltatimesDeltaShuffle}
 Let $m,n\in\NN$ and let $f:\Delta^m_{\theta',t}\times\Delta^n_{\theta,\theta'}\rightarrow\R$ be Lebesgue integrable. Then
 \begin{align}
 \begin{split}
  &\bigg(\int_{\Delta^m_{\theta',t}\times\Delta^n_{\theta,\theta'}}f(s_1,\dots,s_{m},s_{m+1},\dots,s_{m+n})ds_{m+n}\dots ds_1\bigg)^2\\
  %&=\sum_{\substack{(\sigma,\tau)\in\\ S(m,m)\times S(n,n)}}\int_{\Delta^{2m}_{\theta',t}\times\Delta^{2n}_{\theta,\theta'}}f(s_{(\sigma,\tau)(1)},\dots,s_{(\sigma,\tau)(m+n)})f(s_{(\sigma,\tau)(m+n+1)},\dots,s_{(\sigma,\tau)(2m+2n)})ds_{2m+2n}\dots ds_1,
  &=\sum_{\substack{(\sigma,\tau)\in\\ S(m,m)\times S(n,n)}}\int_{\Delta^{2m}_{\theta',t}\times\Delta^{2n}_{\theta,\theta'}}\prod_{k=0}^1f(s_{(\sigma,\tau)(k(m+n)+1)},\dots,s_{(\sigma,\tau)(k(m+n)+m+n)})ds_{2m+2n}\dots ds_1,
 \end{split}
 \end{align}
 where we define the notation
 \begin{align*}
  (\sigma,\tau)(i)=\begin{cases}
                    \sigma(i), & 1\leq i\leq m\\
                    2m+\tau(i-m), & m+1\leq i\leq m+n\\
                    \sigma(i-n), & m+n+1\leq i\leq 2m+n\\
                    2m+\tau(i-2m), & 2m+n+1\leq i\leq 2m+2n.
                   \end{cases}
 \end{align*}
\end{cor}

\begin{proof}
 Defining $\tilde f(s_1,\dots,s_m):=\int_{\Delta^n_{\theta,\theta'}}f(s_1,\dots,s_{m},s_{m+1},\dots,s_{m+n})ds_{m+n}\dots ds_{m+1}$, it holds, by Lemma \ref{3IntegralProductShuffle},
 \begin{align*}
  &\bigg(\int_{\Delta^m_{\theta',t}\times\Delta^n_{\theta,\theta'}}f(s_1,\dots,s_{m},s_{m+1},\dots,s_{m+n})ds_{m+n}\dots ds_1\bigg)^2\\
  &=\bigg(\int_{\Delta^m_{\theta',t}}\tilde f(s_1,\dots,s_{m})ds_{m}\dots ds_1\bigg)^2\\
  &=\sum_{\sigma\in S(m,m)}\int_{\Delta^{2m}_{\theta',t}}\tilde f(s_{\sigma(1)},\dots,s_{\sigma(m)})\tilde f(s_{\sigma(m+1)},\dots,s_{\sigma(2m)})ds_{2m}\dots ds_1\\
  &=\sum_{\sigma\in S(m,m)}\int_{\Delta^{2m}_{\theta',t}}\bigg\{\Big(\int_{\Delta^n_{\theta,\theta'}}f(s_{\sigma(1)},\dots,s_{\sigma(m)},s_{2m+1},\dots,s_{2m+n})ds_{2m+n}\dots ds_{2m+1}\Big)\\
  &\quad\cdot\Big(\int_{\Delta^n_{\theta,\theta'}}f(s_{\sigma(m+1)},\dots,s_{\sigma(2m)},s_{2m+n+1},\dots,s_{2m+2n})ds_{2m+2n}\dots ds_{2m+n+1}\Big)\bigg\}ds_{2m}\dots ds_1\\
  &=\sum_{\sigma\in S(m,m)}\int_{\Delta^{2m}_{\theta',t}}\bigg\{\sum_{\tau\in S(n,n)}\int_{\Delta^{2n}_{\theta,\theta'}}f(s_{\sigma(1)},\dots,s_{\sigma(m)},s_{2m+\tau(1)},\dots,s_{2m+\tau(n)})\\
  &\quad\cdot f(s_{\sigma(m+1)},\dots,s_{\sigma(2m)},s_{2m+\tau(n+1)},\dots,s_{2m+\tau(2n)})ds_{2m+2n}\dots ds_{2m+1}\bigg\}ds_{2m}\dots ds_1\\
  &=\sum_{\substack{(\sigma,\tau)\in\\ S(m,m)\times S(n,n)}}\int_{\Delta^{2m}_{\theta',t}\times\Delta^{2n}_{\theta,\theta'}}\prod_{k=0}^1f(s_{(\sigma,\tau)(k(m+n)+1)},\dots,s_{(\sigma,\tau)(k(m+n)+m+n)})ds_{2m+2n}\dots ds_1.
 \end{align*}

\end{proof}

\begin{cor}\label{3DoubleShuffleInverse}
 It holds $(\sigma,\tau)^{-1}(l)\in\{1,\dots,m\}\cup\{m+n+1,\dots,2m+n\}$ whenever $l\in\{1,\dots,2m\}$ and $(\sigma,\tau)^{-1}(l)\in\{m+1,\dots,m+n\}\cup\{2m+n+1,\dots,2m+2n\}$ whenever $l\in\{2m+1,\dots,2m+2n\}$.
\end{cor}

For the segments and the stochastic perturbation that appear in our study (see \eqref{3MotivatingSDDE2}), we use the Hilbert space $M_2:=\R\times L^2([-r,0],\R),$ equipped with the scalar product
\begin{align*}
 \<x,y\>=x(0)y(0)+\int_{-r}^0x(u)y(u)du,
\end{align*}
and the corresponding norm $\|\cdot\|_{M_2}$, as state space. This space is known as the Delfour-Mitter-space. Furthermore, we denote by $\ell^p$, $p=1,2,\dots$, the space of sequences $(a_i)_{i\geq1}$ of real numbers such that
\begin{align*}
 \sum_{i=1}^\infty|a_i|^p<\infty.
\end{align*}

\begin{lem}\label{3CylfBM}
 Let $(e_n)_{n=1}^\infty$ be an orthonormal basis (ONB) of the Hilbert space $M_2$. Furthermore, let $(B^{H_n})_{n=1}^\infty$ be a sequence of independent fractional brownian motions in $\R$ with Hurst parameters $H_n$, $n=1,2,\dots$. Let $(w_n)_{n=1}^\infty$ be a sequence in $\ell^1$ for which we, for simplicity, assume that $w_n\leq1$ for all $n$. Then, for each $t\in[0,T]$
 \begin{align}
  \BB(t):=\sum_{n=1}^\infty w_nB^{H_n}(t)e_n
 \end{align}
 is a well-defined object in $M_2$ a.s.
\end{lem}

\begin{proof}
 Denote $X_n:=w_nB^{H_n}(t)e_n$, $n\geq1$. The random variables $X_n$, $n\geq1$, take values in $M_2$ and are independent. Define the partial sums $S_m:=\sum_{n=1}^mX_n$, $m\geq1$. Suppose, we can prove that $\sum_{n=1}^\infty\|X_n\|_{M_2}<\infty$. Then, for any $M\geq1$,
 \begin{align*}
  \sup_{m>M}\|S_m-S_M\|_{M_2}\leq\sup_{m>M}\sum_{n=M+1}^m\|X_n\|_{M_2}=\sum_{n=M+1}^\infty\|X_n\|_{M_2}\stackrel{M\rightarrow\infty}{\longrightarrow}0,\quad P\text{-a.s.},
 \end{align*}
 which means that the sequence $(S_m)_{m\geq1}\subset M_2$ is a Cauchy sequence and thus converges and the limit is $\BB(t)$. In order to prove that $\sum_{n=1}^\infty\|X_n\|_{M_2}<\infty$, we apply Kolmogorov's three series theorem. This means, we have to prove that there exists an $A>0$ such that
 \begin{itemize}
  \item[(i)] $\sum_{n=1}^\infty P(\|X_n\|_{M_2}>A)<\infty,$
  \item[(ii)] $\sum_{n=1}^\infty E[\|X_n\|_{M_2}\ind_{\{\|X_n\|_{M_2}>A\}}]<\infty,$
  \item[(iii)] $\sum_{n=1}^\infty \Var(\|X_n\|_{M_2}\ind_{\{\|X_n\|_{M_2}>A\}})<\infty.$
 \end{itemize}
%  \begin{align*}
%   \sum_{n=1}^\infty P(\|X_n\|_{M_2}>A)&<\infty,\\
%   \sum_{n=1}^\infty E[\|X_n\|_{M_2}\ind_{\{\|X_n\|_{M_2}>A\}}]&<\infty,\quad\text{and}\\
%   \sum_{n=1}^\infty Var(\|X_n\|_{M_2}\ind_{\{\|X_n\|_{M_2}>A\}})&<\infty.
%  \end{align*}
Note that $\|X_n\|_{M_2}=|w_n||B^{H_n}(t)|$ and that $w_nB^{H_n}(t)\stackrel{\text{Law}}{=}|w_n|t^{H_n}N$ for some $N\sim\N(0,1)$. Therefore, we have, for any $A>0$,
\begin{align*}
 E[\|X_n\|_{M_2}]=|w_n|t^{H_n}E[|N|]\leq|w_n|t^{H_n}\sqrt{E[|N|^2]}=|w_n|t^{H_n}\leq|w_n|(1\vee t).
\end{align*}
Then, by Markov's inequality, it follows that
\begin{align*}
 \sum_{n=1}^\infty P(\|X_n\|_{M_2}>A)&\leq\sum_{n=1}^\infty \frac{E[\|X_n\|_{M_2}]}{A}\leq\frac{1\vee t}{A}\sum_{n=1}^\infty |w_n|<\infty,
\end{align*}
as $(w_n)_{n\geq1}\in\ell^1$, which proves $(i)$. Moreover, we have
\begin{align*}
 E[\|X_n\|_{M_2}\ind_{\{\|X_n\|_{M_2}>A\}}]\leq E[\|X_n\|_{M_2}]\leq|w_n|(1\vee t),
\end{align*}
and thus
\begin{align*}
  \sum_{n=1}^\infty E[\|X_n\|_{M_2}\ind_{\{\|X_n\|_{M_2}>A\}}]\leq(1\vee t)\sum_{n=1}^\infty|w_n|<\infty.
\end{align*}
This proves $(ii)$.
Finally, observe that
\begin{align*}
 \Var(\|X_n\|_{M_2}\ind_{\{\|X_n\|_{M_2}>A\}})&\leq E[\|X_n\|_{M_2}^2\ind_{\{\|X_n\|_{M_2}>A\}}]\leq E[\|X_n\|_{M_2}^2]\\
 &=|w_n|^2t^{2H_n}E[|N|^2]\leq|w_n|^2(1\vee t^2),
\end{align*}
which yields the last convergence needed:
\begin{align*}
  \sum_{n=1}^\infty \Var(\|X_n\|_{M_2}\ind_{\{\|X_n\|_{M_2}>A\}})\leq(1\vee t^2)\sum_{n=1}^\infty|w_n|^2<\infty.
 \end{align*}
\end{proof}

Finally we recall a crucial property of the fractional Brownian motion, called the \textit{strong local non-determinism} (see \cite[p. 6]{GuoHuXiao17} or \cite[p. 10]{BanNilPro15}; for more details on the matter, see \cite{pitt.78} or \cite{xiao.11}).

\begin{rem}\label{3RemarkStrongLocalNondeterminism}
 Let $B^H$ denote a one-dimensional fractional Brownian motion with Hurst parameter $H$, $m\in\NN$, $0=:t_0<t_1<\dots<t_m\leq T$ and $\xi_1,\dots,\xi_m\in\R$. Then there exists a constant $C=C(H)>0$ such that
 \begin{align*}
  \Var\Big(\sum_{l=1}^m\xi_l\big(B^H(t_l)-B^H(t_{l-1})\big)\Big)&\geq C\sum_{l=1}^m|\xi_l|^2\Var\big(B^H(t_l)-B^H(t_{l-1})\big)\\
  &=C\sum_{l=1}^m|\xi_l|^2|t_l-t_{l-1}|^{2H}
 \end{align*}
\end{rem}

\section{Existence and uniqueness of global strong solutions}

As outlined in the introduction the object of study is a time-homogeneous Stochastic Functional Differential Equation of Delay type as given in \eqref{3MotivatingSDDE2}. We restrict ourselves to the one-dimensional case, i.e. $d=1$. Moreover, we assume the segments to take values in $M_2$ (which is reasonable under the assumptions on $b$ that we are going to impose later) and, for a given ONB $(e_i)_{i\geq1}$ of $M_2$, we denote by $\< x,e_i\>$ the $i^{\text{th}}$ Fourier coefficient, i.e. for $x\in M_2$,
$$x=\sum_{i\geq1}\<x,e_i\>e_i.$$
%Since each $x\in M_2$ is uniquely determined by its Fourier coefficients, the drift function $b:M_2\rightarrow\R$ can be seen as a function $b:\R^\infty\rightarrow\R$ acting on the Fourier coefficients.
In this study, we only consider drift functions of the particular form

\begin{align}\label{3driftRestriction}
 b(x)=\sum_{i\geq1}b_i(\<x,e_i\>),
\end{align}

where $b_i\in L^1(\R)\cap L^\infty(\R)$ are bounded, integrable functions and $(\|b_i\|_\infty)_{i\geq1}\in\ell^1$. The latter condition makes sure that the sum in \eqref{3driftRestriction} is well-defined for every $x\in M_2$. Now, for $0<\eps\leq1$, consider the SDDE

\begin{align}\label{3SFDE}
\begin{split}
 dx(t) &= b(x_t+\eps\BB(t))dt + dW(t)\\
 &=\sum_{i\geq1}b_i(\<x_t,e_i\>+\eps w_iB^{H_i}(t))dt+dW(t), \quad t\in [0,T],\\
 x_0 &= \eta\in M_2.
\end{split}
\end{align}

Throughout the paper, we assume that the Brownian motion $W$ is defined on a probability space $(\Omega,\mathfrak{A},\PP)$ where it generates a filtration $(\F^W_t)_{t\in[0,T}$, i.e. $\F^W_t:=\sigma(W(s),s\leq t)\subseteq\mathfrak{A}$. The object $\BB$ on the other hand, is defined on a second probability space $(\Omega',\mathfrak{A}',\PP')$ where it generates the filtration $(\F^\BB_t)_{t\in[0,T}$. The probability space we work on is the product probability space $(\Omega\times\Omega',\mathfrak{A}\otimes\mathfrak{A}',\PP\otimes\PP')$ with the filtration $(\F_t)_{t\in[0,T]}$, $\F_t=\F^W_t\otimes\F^\BB_t$, although we will later show that our results can be ``lifted'' to any stochastic basis $(\hat\Omega,\hat{\mathfrak{A}},\hat\PP,\widehat W,\hat\BB)$. Furthermore, we let $E$ denote the expectation operator under $(\Omega,\mathfrak{A},\PP)$, $E'$ denote the expectation operator under $(\Omega',\mathfrak{A}',\PP')$ and $\bar E$ the expectation operator under on the product space. In the same manner, we sometimes write $\bar\PP$ for the product probability measure $\PP\otimes\PP'$. To be more specific, let $X:\Omega\times\Omega'\rightarrow\R$ be a random variable on the product space. Then

\begin{align*}
 \bar E[X]=\int_{\Omega\times\Omega'}X(\omega,\omega')\PP(d\omega)\PP'(d\omega')
\end{align*}
denotes the expectation of $X$, whereas,

\begin{align*}
 E[X(\cdot,\omega')]=\int_\Omega X(\omega,\omega')\PP(d\omega)
\end{align*}

denotes the expectation of the random variable $X(\cdot,\omega'):\Omega\rightarrow\R$ for every fixed $\omega'\in\Omega'$ and

\begin{align*}
 E'[X(\omega,\cdot)]=\int_{\Omega'} X(\omega,\omega')\PP'(d\omega')
\end{align*}

denotes the expectation of the random variable $X(\omega,\cdot):\Omega'\rightarrow\R$ for every fixed $\omega\in\Omega$.\\
\par

Hereunder, we state the main result of this paper.

\begin{claim}\label{3VI_mainthm}
Under some integrability conditions on $b_i$, $i\geq 1$, and some ``roughness conditions'' on $\BB$, equation \eqref{3SFDE} has a unique strong solution $x=\{x(t), t\in [0,T]\}$ on a small interval $[0,T]$. Here, the interval size depends on $\eps$ -- the smaller $\eps$ is, the smaller $T$ needs to be.%Moreover, for every $t\in [0,T]$, $X_t$ is Malliavin differentiable in the direction of the Brownian motion $W$ in \eqref{3VI_WBH}.
\end{claim}

The proof of Claim \ref{3VI_mainthm} is based on the following steps:

\begin{itemize}
 \item[(I.)] We first prove the claim for the case where
 \begin{align*}
  b(x)=\sum_{i=1}^db_i(\<x,e_i\>),
 \end{align*}
 i.e. the sum in \ref{3driftRestriction} is a finite sum. Moreover, we assume in this step that each $b_i$, $i=1,\dots,d$ has compact support. We follow these steps:
 \begin{enumerate}
  \item Approximate $b$ by a sequence of bounded, Lipschitz functions $b^n$. More precisely, we approximate the functions $b_i$, $i=1,\dots,d$ by mollification: $$b^{n}(z):=\sum_{i=1}^db_{i,n}(z_i),\text{ where }b_{i,n}(z):=b_i\ast\varphi_n(z)$$ for some mollifier $\varphi$ with support $[-1,1]$. Then each $b_{i,n}$, $i=1,\dots,d$ is a $C^\infty_c$ function and therefore Lipschitz, which then also makes $b^n:\R^d\rightarrow\R$ Lipschitz. For SFDEs with Lipschitz coefficients which fulfill a linear growth condition (satisfied for every bounded function) we know by \cite[Theorem 2.1]{MohammedBook} that the corresponding SFDE has a unique solution. Therefore, the SFDE
  \begin{align*}
   \begin{cases}
    dx^{n}(t)&= \sum_{i=1}^db_{i,n}(\<x^{n}_t, e_i\>+\eps w_iB^{H_i}(t)) dt+ dW(t), \,\,0\leq t\leq T\\
    x^{n}_0&=\eta  \in M_2\,.
   \end{cases}
  \end{align*}
  has a unique strong solution $x^n$. In particular, for each $n\in\NN$, there exists a family of progressively measurable functionals $(\psi_n(t,\cdot,\cdot))_{t\in[0,T]}$ such that
  \begin{align*}
   x^n_t=\psi_n(t,W_\cdot,\BB_\cdot),\quad\text{for all }t\in[0,T].
  \end{align*}
  It is well known, see e.g. \cite{Nualart}, that for each $t\in[0,T]$ the strong solution $x^{n}(t)$, $n\ge 1$, is Malliavin differentiable, and that the Malliavin derivative $\DM_s x^{n}(t)$, $0\leq s\leq t$, with respect to $W$ satisfies
  \begin{align}
   \DM_s x^{n}(t) = 1 + \int_s^t \sum_{i=1}^db_{i,n}(\<x^{n}_u, e_i\>+\eps w_iB^{H_i}(u))\DM_s \< x^{n}_u,e_i\>du. 
  \end{align}
  \item We prove a compactness criterion for the sequence of solutions $x^n$ when we fix $\omega'\in\Omega'$. More precisely, we prove that for every fixed $t\in[0,T]$ and almost every $\omega'\in\Omega'$, there exists a subsequence $(n_k(\omega'))_{k\geq1}$ such that $x^{n_k(\omega')}(t,\cdot,\omega')$ is relatively compact in $L^2(\Omega,\F^W_t,\PP)$.
  %\item Assuming for the moment that the probability space has the form $(\Omega\times\Omega',\mathfrak{A}\otimes\mathfrak{A}',\PP\otimes\PP')$ where $W$ depends only in $\omega\in\Omega$ and $\BB$ only on $\omega'\in\Omega'$, we prove a compactness criterion for the sequence of solutions $x^n$ when we fix $\omega'\in\Omega'$. More precisely, we prove that for every fixed $t\in[0,T]$ and almost every $\omega'\in\Omega'$, there exists a subsequence $(n_k(\omega'))_{k\geq1}$ such that $x^{n_k(\omega')}(t,\cdot,\omega')$ is relatively compact in $L^2(\Omega,\F^W_t,\PP)$, \textcolor{red}{where $\F^W_t:=\sigma(W(s),s\leq t)\subseteq\mathfrak{A}$}.
  \item Applying the compactness criterion that we have proven in the previous step, we show that for all $t\in[0,T]$ there exists an $x(t)$ such that
  \begin{align*}
   x^n(t)\stackrel{n\rightarrow\infty}{\longrightarrow}x(t)\quad\text{strongly in }L^2(\Omega\times\Omega',\F_t,\PP\otimes\PP'),
  \end{align*}
  where $\F_t=\sigma((W(s),\BB(s)),s\leq t)$.
  \item Via Girsanov's theorem, we construct a weak solution $\tilde x$ of SFDE \ref{3SFDE} for the stochastic basis $(\Omega\times\Omega',\mathfrak{A}\otimes\mathfrak{A}',\widetilde\PP,\widetilde W,\BB)$. For the same family of progressively measurable functionals $(\psi_n(t,\cdot,\cdot))_{t\in[0,T]}$ as before, we define
  \begin{align*}
   \tilde x^n_t=\psi_n(t,\widetilde W_\cdot,\BB_\cdot),\quad\text{for all }t\in[0,T].
  \end{align*}
  Note that this definition makes $\tilde x^n$ adapted to the filtration $(\widetilde\F_t)_{t\in[0,T]}$, where $\F_t=\sigma((\widetilde W(s),\BB(s)),s\leq t)$. Applying the results from the previous steps, we show that
  \begin{align*}
   \tilde x^n(t)\stackrel{n\rightarrow\infty}{\longrightarrow}\tilde x(t)\quad\text{strongly in }L^2(\Omega\times\Omega',\widetilde \F_t,\widetilde\PP),
  \end{align*}
  which implies that the weak solution that we constructed is adapted to $(\widetilde\F_t)_{t\in[0,T]}$, the filtration generated by the two stochastic basis elements $\widetilde W$ and $\BB$. This result implies that for any stochastic basis $(\hat\Omega,\hat{\mathfrak{A}},\hat\PP,\hat{W},\hat{\BB})$ there is a solution to the SFDE which is $(\sigma((\hat W(s),\hat\BB(s)),s\leq t))_{t\in[0,T]}$-adapted, i.e. we found a \textit{strong solution}.
 \end{enumerate}
 \item[(II.)] In the second step we apply the results we found in $(I.)$ to prove the existence of a strong solution to the original SFDE \ref{3SFDE}.
\end{itemize}

\subsection{Approximation by finite dimensional Lipschitz SFDEs}

As outlined before, we first consider the case where $b_i$, $i=1,\dots,d$ have compact support and

\begin{align}\label{3bfiniteCase}
  b(x)=\sum_{i=1}^db_i(\<x,e_i\>).
 \end{align}

Let $\varphi:\R\rightarrow\R$ be a non-negative mollifier with support $[-1,1]$ and let $\varphi_n(z):=n\varphi(nz)$ for every $n\in\NN$. We now approximate every $b_i:\R\rightarrow\R$, $i=1,\dots,d$ by a sequence of functions $(b_{i,n})_{n\geq1}$, which is given by

\begin{align}
 b_{i,n}(z)=b_i\ast\varphi_n(z).
\end{align}

Moreover, we define the sequence $(b^n)_{n\geq1}$ of approximations of $b$ by

\begin{align}
 b^n(x)=\sum_{i=1}^db_{i,n}(\<x,e_i\>),
\end{align}

and the corresponding sequence of SFDEs are given by

\begin{align}\label{3ApproxSFDE}
   \begin{cases}
    dx^{n}(t)&= b^n(x^{n}_t+\eps\BB(t))dt+ dW(t)\\
    &= \sum_{i=1}^db_{i,n}(\<x^{n}_t, e_i\>+\eps w_iB^{H_i}(t)) dt+ dW(t), \,\,0\leq t\leq T\\
    x^{n}_0&=\eta  \in M_2\,.
   \end{cases}
  \end{align}

It then holds $b_{i,n}\rightarrow b_i$ pointwise and in $L^1$, as $n\rightarrow\infty$. Therefore, we also have $b^n\rightarrow b$ pointwise, as $n\rightarrow\infty$. Furthermore, since

\begin{align*}
 \|b_{i,n}\|_\infty&=\sup_{z\in\R}\Big|\int_\R b_i(y)\varphi_n(z-y)dy\Big|\leq\sup_{z\in\R}\int_\R |b_i(y)|\varphi_n(z-y)dy\\
 &\leq\sup_{z\in\R}\|b_i\|_\infty\int_\R \varphi_n(z-y)dy=\sup_{z\in\R}\|b_i\|_\infty\cdot1=\|b_i\|_\infty,
\end{align*}

and since every $b_{i,n}$ is $C^\infty_c$ and therefore Lipschitz (the first derivative is continuous with compact support and thus bounded), $b^n$ is also bounded and Lipschitz for every $n$. The following Lemma shows that the approximative SFDEs have a unique soluion.

\begin{lem}\label{3ExistAndUniqApprox}
 For each $n\geq1$, the system \eqref{3ApproxSFDE} has a unique, $(\F_t)_{t\in[0,T]}$-adapted solution $x^{n}\in L^2(\Omega\times\Omega', M_2([-r,T]))$.
\end{lem}
  
\begin{proof}
 We can reformulate SFDE \ref{3ApproxSFDE} as the $(d+1)$-dimensional SFDE
 \begin{align*}
  \begin{pmatrix}
   x^n(t)\\
   y_1(t)\\
   \vdots\\
   y_d(t)
  \end{pmatrix}=\int_0^th^n(x^n_s,y_1(s),\dots,y_n(s))ds
  +\begin{pmatrix}
    W(t)\\
    B^{H_1}(t)\\
    \vdots\\
    B^{H_d}(t)
   \end{pmatrix},
 \end{align*}
 where
 \begin{align*}
  h^n(x^n_s,y_1(s),\dots,y_n(s))=\begin{pmatrix}
                                b^n(\<x^{n}_t,e_1\>+\eps w_1y_1(t),\dots,\<x^{n}_t,e_d\>+\eps w_dy_d(t))\\
                                0\\
                                \vdots\\
                                0
                               \end{pmatrix}.
 \end{align*}
 Since $b^n$ is bounded and Lipschitz for every $n\in\NN$, so is $h^n$. It follows by \cite[Theorem 2.1]{MohammedBook} that \eqref{3ApproxSFDE} has a unique solution $x^{n}$.
\end{proof}

%%%%%%%%%%%%%%%%%%%%%%%%%%%%%%%%%%%%%%%%%%%%%%%%%%%%%%%%%%%%%%%%%%%%%%%%%%%%%%%%%%%%%%%%%%%%%%%%%%%%%%%%%%%%%%%%%%%%%%%%%%%%%%%%%%%%%%%%%%%%

\subsection{The $\omega'$-wise compactness criterion}

Before we can prove our compactness result, we are going to derive an iteration formula for the Malliavin derivative of $x^{n}(t)$ w.r.t. $W$. Fix $t\in[0,T]$. By standard results on Malliavin differentiability (see \cite{Nualart} and the references therein), we know that for each $n\geq1$, $x^{n}(t)$ is Malliavin differentiable with Malliavin derivative 

\begin{align}\label{3VI_DXn1}
\DM_s x^{n}(t) = 1 + \int_s^t \sum_{i=1}^db_{i,n}(\<x^{n}_u, e_i\>+\eps w_iB^{H_i}(u))\DM_s \< x^{n}_u,e_i\>du.
\end{align}

The following representation result will be useful for developing our compactness result later on.

\begin{lem}
 Let $0\leq\theta<\theta'\leq t\leq T$ and define $\chi_j(z):=\<(1,\ind_{[z,0]}(\cdot)),e_j\>$ for $z\leq0$, $j=1,\dots,d$. Then,
 \begin{itemize}
  \item[(Ma1)] for the Malliavin derivative $\DM_\theta x^{n}(t)$, the following representation holds true:
  \begin{align}\label{3MalliavinIteration}
   \begin{split}
   &\DM_\theta x^{n}(t)\\
   &=1+ \sum_{m=1}^\infty\sum_{j_1,\dots,j_m=1}^d\int_{\Delta^m_{\theta,t}}\Hfct_{j_1,\dots,j_m}(s,\theta)\prod_{l=1}^mb_{j_l,n}'(\< x^{n}_{s_l}, e_{j_l}\>+\eps w_{j_l}B^{H_{j_l}}(s_l))ds,
   \end{split}
  \end{align}
  where we defined
  \begin{align}\label{3Hfunction}
   \Hfct_{j_1,\dots,j_m}(s,\theta):=\chi_{j_m}(\theta-s_m)\prod_{l=1}^{m-1}\chi_{j_l}(s_{l+1}-s_l).
  \end{align}
 \end{itemize}
 \item[(Ma2)] for the difference $\DM_\theta x^{n}(t)-\DM_{\theta'} x^{n}(t)$, the following representation holds true:
 \begin{align}\label{3MalliavinDiffIteration}
   \begin{split}
   &\DM_\theta x^{n}(t)-\DM_{\theta'} x^{n}(t)\\
   &=\Big(\DM_\theta x^{n}(\theta')-1\Big)\\
   &\quad+ \sum_{m=1}^\infty\sum_{j_1,\dots,j_m=1}^d\int_{\Delta^m_{\theta',t}}\tilde\Hfct_{j_1,\dots,j_m}(s,\theta,\theta')\prod_{l=1}^mb_{j_l,n}'(\< x^{n}_{s_l}, e_{j_l}\>+\eps w_{j_l}B^{H_{j_l}}(s_l))ds\\
   &\quad+\sum_{m_1, m_2=1}^\infty\sum_{j_1,\dots,j_{m_1+m_2}=1}^d\int_{\Delta^{m_1}_{\theta',t}\times\Delta^{m_2}_{\theta,\theta'}}\Big\{\Hfct_{j_1,\dots,j_{m_1+m_2} }(s,\theta)\\
   &\quad\cdot\prod_{l=1}^{m_1+m_2}b_{j_l,n}'(\< x^{n}_{s_l}, e_{j_l}\>+\eps w_{j_l}B^{H_{j_l}}(s_l))\Big\}ds,
   \end{split}
  \end{align}
  where we defined
  \begin{align}\label{3Htildefunction}
   \tilde\Hfct_{j_1,\dots,j_m}(s,\theta,\theta'):=\Big(\chi_{j_m}(\theta-s_m)-\chi_{j_m}(\theta'-s_m)\Big)\prod_{l=1}^{m-1}\chi_{j_l}(s_{l+1}-s_l).
  \end{align}
\end{lem}

\begin{proof}
We start with proving \eqref{3MalliavinIteration}. First, we define the process $\Xi$ by

\begin{align}\label{3Xiprocess}
 \Xi(s):=\begin{pmatrix}
   \< x^{n}_s,e_1\> \\                                                                                                                                                          
   \vdots \\
   \< x^{n}_s,e_d\>
  \end{pmatrix}.
\end{align}

Then, equation \eqref{3VI_DXn1} becomes

\begin{align*}
&\DM_\theta x^{n}(t)= 1 + \int_\theta^t \bigg\{\nabla b^{n}(\< x^{n}_s, e_1\>+\eps w_1B^{H_1}(s),\dots)\cdot \DM_\theta\Xi(s)\bigg\}ds.
\end{align*}

By definition of $\<\cdot,\cdot\>$ in $M_2$,

\begin{align*}
&\DM_\theta\<x^{n}_s,e_j\>\\
&= \DM_\theta x^{n}(s)e_j(0) + \int_{-r}^0 \DM_\theta x^{n}(s+u)e_j(u)du\\
&=1\cdot e_j(0)+\int_\theta^s \nabla b^{n}(\< x^{n}_{s_1}, e_1\>+\eps w_1B^{H_1}(s_1),\dots)\begin{pmatrix}
                 \DM_\theta \< x^{n}_{s_1},e_1\> \\                                                                                                                                                          
                \vdots \\
                 \DM_\theta \< x^{n}_{s_1},e_d\>
                \end{pmatrix}
 ds_1e_j(0)\\
 &\quad+\int_{-r}^0\ind_{\{s+u\geq\theta\}}\bigg\{1+\int_\theta^{s+u} \nabla b^{n}(\< x^{n}_{s_1}, e_1\>+\eps w_1B^{H_1}(s_1),\dots)\begin{pmatrix}
                 \DM_\theta \< x^{n}_{s_1},e_1\> \\                                                                                                                                                          
                \vdots \\
                 \DM_\theta \< x^{n}_{s_1},e_d\>
                \end{pmatrix}
 ds_1\bigg\}e_j(u)du\\
 &=\<(1,\ind_{[\theta-s,0]}(\cdot)),e_j\>\\
 &\quad+\int_\theta^s \nabla b^{n}(\< x^{n}_{s_1}, e_1\>+\eps w_1B^{H_1}(s_1),\dots)\begin{pmatrix}
                 \DM_\theta \< x^{n}_{s_1},e_1\> \\                                                                                                                                                          
                \vdots \\
                 \DM_\theta \< x^{n}_{s_1},e_d\>
                \end{pmatrix}
 \<(1,\ind_{[s_1-s,0]}(\cdot)),e_j\>ds_1\\
 &=\chi_j(\theta-s)+\int_\theta^s \nabla b^{n}(\< x^{n}_{s_1}, e_1\>+\eps w_1B^{H_1}(s_1),\dots)\DM_\theta\Xi(s_1)\chi_j(s_1-s)ds_1,
\end{align*}

where we applied Fubini's theorem and the definitions of the function $\chi_j$ and the process $\Xi$ in the last two steps. Therefore,

\begin{align}\label{3DXi}
\begin{split}
 &\DM_\theta\Xi(s)=\begin{pmatrix}
    \DM_\theta \< x^{n}_s,e_1\> \\                                                                                                                                                          
   \vdots \\
    \DM_\theta \< x^{n}_s,e_d\>
  \end{pmatrix}\\ 
  &=\begin{pmatrix}
      \chi_1(\theta-s) \\                                                                                                                                                          
     \vdots \\
      \chi_d(\theta-s)
    \end{pmatrix}
    +\int_\theta^s \bigg\{\nabla b^{n}(\< x^{n}_{s_1}, e_1\>+\eps w_1B^{H_1}(s_1),\dots)\cdot \DM_\theta\Xi(s_1)\bigg\}\begin{pmatrix}
       \chi_1(s_1-s) \\                                                                                                                                                          
      \vdots \\
       \chi_d(s_1-s)
     \end{pmatrix}ds_1.
\end{split}
\end{align}

By iteration in connection with the fixed point theorem of Weissinger in e.g. $L^p$-spaces, we get

\begin{align*}
  &\DM_\theta\Xi(s)\\
  &=\begin{pmatrix}
      \chi_1(\theta-s) \\                                                                                                                                                          
     \vdots \\
      \chi_d(\theta-s)
    \end{pmatrix}
    +\sum_{m=1}^\infty\int_{\Delta^m_{\theta,s}} \bigg\{\nabla b^{n}(\< x^{n}_{s_m}, e_1\>+\eps w_1B^{H_1}(s_m),\dots)\cdot\begin{pmatrix}
       \chi_1(\theta-s_m) \\                                                                                                                                                          
      \vdots \\
       \chi_d(\theta-s_m)
     \end{pmatrix}\\
     &\quad\cdot\prod_{l=1}^{m-1}\nabla b^{n}(\< x^{n}_{s_l}, e_1\>+\eps w_1B^{H_1}(s_l),\dots)\cdot
     \begin{pmatrix}
       \chi_1(s_{l+1}-s_l) \\                                                                                                                                                          
      \vdots \\
       \chi_d(s_{l+1}-s_l)
     \end{pmatrix}\bigg\}\begin{pmatrix}
       \chi_1(s_1-s) \\                                                                                                                                                          
      \vdots \\
       \chi_d(s_1-s)
     \end{pmatrix}ds_m\dots ds_1.
\end{align*}

When plugging this result into \eqref{3VI_DXn1}, we achieve

\begin{align*}
&\DM_\theta x^{n}(t)\\
&= 1 + \sum_{m=1}^\infty\int_{\Delta^m_{\theta,t}} \bigg\{\nabla b^{n}(\< x^{n}_{s_m}, e_1\>+\eps w_1B^{H_1}(s_m),\dots)\cdot\begin{pmatrix}
       \chi_1(\theta-s_m) \\                                                                                                                                                          
      \vdots \\
       \chi_d(\theta-s_m)
     \end{pmatrix}\\
     &\quad\cdot\prod_{l=1}^{m-1}\nabla b^{n}(\< x^{n}_{s_l}, e_1\>+\eps w_1B^{H_1}(s_l),\dots)\cdot
     \begin{pmatrix}
       \chi_1(s_{l+1}-s_l) \\                                                                                                                                                          
      \vdots \\
       \chi_d(s_{l+1}-s_l)
     \end{pmatrix}\bigg\}ds_m\dots ds_1\\
     &=1 + \sum_{m=1}^\infty\sum_{j_1,\dots,j_m=1}^d\int_{\Delta^m_{\theta,t}}b_{j_m,n}'(\< x^{n}_{s_m}, e_1\>+\eps w_1B^{H_1}(s_m),\dots)\chi_{j_m}(\theta-s_m)\\
     &\quad\cdot\prod_{l=1}^{m-1}b_{j_l,n}'(\< x^{n}_{s_l}, e_1\>+\eps w_1B^{H_1}(s_l),\dots)\chi_{j_l}(s_{l+1}-s_l)ds_m\dots ds_1.
\end{align*}

Finally, we define the function $\Hfct_{j_1,\dots,j_m}$ as in \eqref{3Hfunction} and achive representation \eqref{3MalliavinIteration}.\\
\par

Next, we prove \eqref{3MalliavinDiffIteration}. Note that, by \eqref{3VI_DXn1} and \eqref{3Xiprocess}, we have

\begin{align*}
&\DM_\theta x^{n}(t)-\DM_{\theta'} x^{n}(t)\\
&\quad= 1 + \int_\theta^t \bigg\{\nabla b^{n}(\< x^{n}_s, e_1\>+\eps w_1B^{H_1}(s),\dots)\cdot \DM_\theta\Xi(s)\bigg\}ds\\
&\qquad-\bigg(1 + \int_{\theta'}^t \bigg\{\nabla b^{n}(\< x^{n}_s, e_1\>+\eps w_1B^{H_1}(s),\dots)\cdot \DM_\theta\Xi(s)\bigg\}ds\bigg)\\
&\quad= \int_\theta^{\theta'} \bigg\{\nabla b^{n}(\< x^{n}_s, e_1\>+\eps w_1B^{H_1}(s),\dots)\cdot \DM_\theta\Xi(s)\bigg\}ds\\
&\qquad+ \int_{\theta'}^t \bigg\{\nabla b^{n}(\< x^{n}_s, e_1\>+\eps w_1B^{H_1}(s),\dots)\cdot\Big(\DM_\theta\Xi(s)-\DM_{\theta'}\Xi(s)\Big)\bigg\}ds\\
&\quad=\Big(\DM_\theta x^{n}(\theta')-1\Big)+ \int_{\theta'}^t \bigg\{\nabla b^{n}(\< x^{n}_s, e_1\>+\eps w_1B^{H_1}(s),\dots)\cdot\Big(\DM_\theta\Xi(s)-\DM_{\theta'}\Xi(s)\Big)\bigg\}ds.
\end{align*}

By \eqref{3DXi}, we have
\begin{align*}
 &\DM_\theta\Xi(s)-\DM_{\theta'}\Xi(s)\\
 &=\begin{pmatrix}
     \chi_1(\theta-s)-\chi_1(\theta'-s) \\                                                                                                                                                          
     \vdots \\
     \chi_d(\theta-s)-\chi_d(\theta'-s)
    \end{pmatrix}\\
 &\quad+\int_\theta^{\theta'} \bigg\{\nabla b^{n}(\< x^{n}_{s_1}, e_1\>+\eps w_1B^{H_1}(s_1),\dots)\cdot \DM_\theta\Xi(s_1)\bigg\}\begin{pmatrix}
       \chi_1(s_1-s) \\                                                                                                                                                          
      \vdots \\
       \chi_d(s_1-s)
     \end{pmatrix}ds_1\\
 &\quad +\int_{\theta'}^s \bigg\{\nabla b^{n}(\< x^{n}_{s_1}, e_1\>+\eps w_1B^{H_1}(s_1),\dots)\cdot\Big(\DM_\theta\Xi(s_1)-\DM_{\theta'}\Xi(s_1)\Big)\bigg\}\begin{pmatrix}
       \chi_1(s_1-s) \\                                                                                                                                                          
      \vdots \\
       \chi_d(s_1-s)
     \end{pmatrix}ds_1\\
 &=A(\theta,\theta',s)\\
 &\quad +\int_{\theta'}^s \bigg\{\nabla b^{n}(\< x^{n}_{s_1}, e_1\>+\eps w_1B^{H_1}(s_1),\dots)\cdot\Big(\DM_\theta\Xi(s_1)-\DM_{\theta'}\Xi(s_1)\Big)\bigg\}\begin{pmatrix}
       \chi_1(s_1-s) \\                                                                                                                                                          
      \vdots \\
       \chi_d(s_1-s)
     \end{pmatrix}ds_1.
\end{align*}

By iteration, we achieve

\begin{align*}
 &\DM_\theta\Xi(s)-\DM_{\theta'}\Xi(s)\\
 &=A(\theta,\theta',s)+\sum_{m=1}^\infty\int_{\Delta^m_{\theta',s}} \bigg\{\nabla b^{n}(\< x^{n}_{s_m}, e_1\>+\eps w_1B^{H_1}(s_m),\dots)\cdot A(\theta,\theta',s_m)\\
 &\quad\cdot\prod_{l=1}^{m-1}\nabla b^{n}(\< x^{n}_{s_l}, e_1\>+\eps w_1B^{H_1}(s_l),\dots)\cdot\begin{pmatrix}
       \chi_1(s_{l+1}-s_l) \\                                                                                                                                                          
      \vdots \\
       \chi_d(s_{l+1}-s_l)
     \end{pmatrix}\bigg\}\begin{pmatrix}
       \chi_1(s_1-s) \\                                                                                                                                                          
      \vdots \\
       \chi_d(s_1-s)
     \end{pmatrix}ds_m\dots ds_1.
\end{align*}

Plugging this result into the equation for $\DM_\theta x^{n}(t)-\DM_{\theta'} x^{n}(t)$ yields

\begin{align*}
&\DM_\theta x^{n}(t)-\DM_{\theta'} x^{n}(t)\\
&=\Big(\DM_\theta x^{n}(\theta')-1\Big)+ \int_{\theta'}^t \bigg\{\nabla b^{n}(\< x^{n}_s, e_1\>+\eps w_1B^{H_1}(s),\dots)\cdot A(\theta,\theta',s)\bigg\}ds\\
&\quad+\sum_{m=1}^\infty\int_{\theta'}^t\int_{\Delta^m_{\theta',s}}\Bigg\{\nabla b^{n}(\< x^{n}_s, e_1\>+\eps w_1B^{H_1}(s),\dots)\\
&\quad\cdot\Bigg[\bigg\{\nabla b^{n}(\< x^{n}_{s_m}, e_1\>+\eps w_1B^{H_1}(s_m),\dots)\cdot A(\theta,\theta',s_m)\\
 &\quad\cdot\prod_{l=1}^{m-1}\nabla b^{n}(\< x^{n}_{s_l}, e_1\>+\eps w_1B^{H_1}(s_l),\dots)\cdot\begin{pmatrix}
       \chi_1(s_{l+1}-s_l) \\                                                                                                                                                          
      \vdots \\
       \chi_d(s_{l+1}-s_l)
     \end{pmatrix}\bigg\}\begin{pmatrix}
       \chi_1(s_1-s) \\                                                                                                                                                          
      \vdots \\
       \chi_d(s_1-s)
     \end{pmatrix}\Bigg]\Bigg\}ds_1ds%\\
% &=\Big(\DM_\theta x^{n}(\theta')-1\Big)\\
% &\quad+\sum_{m=1}^\infty\int_{\Delta^m_{\theta',t}} \bigg\{\nabla b^{n}(\< x^{n}_{s_m}, e_1\>+\eps w_1B^{H_1}(s_m),\dots)\cdot A(\theta,\theta',s_m)\\
%  &\quad\cdot\prod_{l=1}^{m-1}\nabla b^{n}(\< x^{n}_{s_l}, e_1\>+\eps w_1B^{H_1}(s_l),\dots)\cdot\begin{pmatrix}
%        \chi_1(s_{l+1}-s_l) \\                                                                                                                                                          
%       \vdots \\
%        \chi_d(s_{l+1}-s_l)
%      \end{pmatrix}\bigg\}ds_m\dots ds_1.
\end{align*}

\begin{align*}
% &\DM_\theta x^{n}(t)-\DM_{\theta'} x^{n}(t)\\
% &=\Big(\DM_\theta x^{n}(\theta')-1\Big)+ \int_{\theta'}^t \bigg\{\nabla b^{n}(\< x^{n}_s, e_1\>+\eps w_1B^{H_1}(s),\dots)\cdot A(\theta,\theta',s)\bigg\}ds\\
% &\quad+\sum_{m=1}^\infty\int_{\theta'}^t\int_{\Delta^m_{\theta',s}}\Bigg\{\nabla b^{n}(\< x^{n}_s, e_1\>+\eps w_1B^{H_1}(s),\dots)\\
% &\quad\cdot\Bigg[\bigg\{\nabla b^{n}(\< x^{n}_{s_m}, e_1\>+\eps w_1B^{H_1}(s_m),\dots)\cdot A(\theta,\theta',s_m)\\
%  &\quad\cdot\prod_{l=1}^{m-1}\nabla b^{n}(\< x^{n}_{s_l}, e_1\>+\eps w_1B^{H_1}(s_l),\dots)\cdot\begin{pmatrix}
%        \chi_1(s_{l+1}-s_l) \\                                                                                                                                                          
%       \vdots \\
%        \chi_d(s_{l+1}-s_l)
%      \end{pmatrix}\bigg\}\begin{pmatrix}
%        \chi_1(s_1-s) \\                                                                                                                                                          
%       \vdots \\
%        \chi_d(s_1-s)
%      \end{pmatrix}\Bigg]\Bigg\}ds_1ds\\
&\hspace{-2cm}=\Big(\DM_\theta x^{n}(\theta')-1\Big)\\
&\hspace{-2cm}\quad+\sum_{m=1}^\infty\int_{\Delta^m_{\theta',t}} \bigg\{\nabla b^{n}(\< x^{n}_{s_m}, e_1\>+\eps w_1B^{H_1}(s_m),\dots)\cdot A(\theta,\theta',s_m)\\
 &\hspace{-2cm}\quad\cdot\prod_{l=1}^{m-1}\nabla b^{n}(\< x^{n}_{s_l}, e_1\>+\eps w_1B^{H_1}(s_l),\dots)\cdot\begin{pmatrix}
       \chi_1(s_{l+1}-s_l) \\                                                                                                                                                          
      \vdots \\
       \chi_d(s_{l+1}-s_l)
     \end{pmatrix}\bigg\}ds_m\dots ds_1.
\end{align*}

By the definition of $A(\theta,\theta',s_m)$, this can be written as

\begin{align*}
&\DM_\theta x^{n}(t)-\DM_{\theta'} x^{n}(t)\\
&=\Big(\DM_\theta x^{n}(\theta')-1\Big)\\
&\quad+\sum_{m=1}^\infty\int_{\Delta^m_{\theta',t}} \bigg\{\nabla b^{n}(\< x^{n}_{s_m}, e_1\>+\eps w_1B^{H_1}(s_m),\dots)\cdot \begin{pmatrix}
     \chi_1(\theta-s_m)-\chi_1(\theta'-s_m) \\                                                                                                                                                          
     \vdots \\
     \chi_d(\theta-s_m)-\chi_d(\theta'-s_m)
    \end{pmatrix}\\
 &\quad\cdot\prod_{l=1}^{m-1}\nabla b^{n}(\< x^{n}_{s_l}, e_1\>+\eps w_1B^{H_1}(s_l),\dots)\cdot\begin{pmatrix}
       \chi_1(s_{l+1}-s_l) \\                                                                                                                                                          
      \vdots \\
       \chi_d(s_{l+1}-s_l)
     \end{pmatrix}\bigg\}ds_m\dots ds_1\\
&\quad+\sum_{m=1}^\infty\int_{\Delta^m_{\theta',t}} \bigg\{\nabla b^{n}(\< x^{n}_{s_m}, e_1\>+\eps w_1B^{H_1}(s_m),\dots)\\
&\quad\cdot\bigg[\int_\theta^{\theta'} \bigg\{\nabla b^{n}(\< x^{n}_{s_{m+1}}, e_1\>+\eps w_1B^{H_1}(s_{m+1}),\dots)\cdot \DM_\theta\Xi(s_{m+1})\bigg\}\begin{pmatrix}
       \chi_1(s_{m+1}-s_m) \\                                                                                                                                                          
      \vdots \\
       \chi_d(s_{m+1}-s_m)
     \end{pmatrix}ds_{m+1}\bigg]\\
 &\quad\cdot\prod_{l=1}^{m-1}\nabla b^{n}(\< x^{n}_{s_l}, e_1\>+\eps w_1B^{H_1}(s_l),\dots)\cdot\begin{pmatrix}
       \chi_1(s_{l+1}-s_l) \\                                                                                                                                                          
      \vdots \\
       \chi_d(s_{l+1}-s_l)
     \end{pmatrix}\bigg\}ds_m\dots ds_1.
\end{align*}

Plugging in the iteration for $\DM_\theta\Xi(s_{m+1})$, we get

\begin{align*}
 &\int_\theta^{\theta'} \bigg\{\nabla b^{n}(\< x^{n}_{s_{m+1}}, e_1\>+\eps w_1B^{H_1}(s_{m+1}),\dots)\cdot \DM_\theta\Xi(s_{m+1})\bigg\}\begin{pmatrix}
       \chi_1(s_{m+1}-s_m) \\                                                                                                                                                          
      \vdots \\
       \chi_d(s_{m+1}-s_m)
     \end{pmatrix}ds_{m+1}\\
 &=\int_\theta^{\theta'} \left\{\nabla b^{n}(\< x^{n}_{s_{m+1}}, e_1\>+\eps w_1B^{H_1}(s_{m+1}),\dots)\cdot \left[\begin{pmatrix}
      \chi_1(\theta-s_{m+1}) \\                                                                                                                                                          
     \vdots \\
      \chi_d(\theta-s_{m+1})
    \end{pmatrix}\right.\right.\\
    &\quad+\sum_{m_2=1}^\infty\int_{\Delta^{m_2}_{\theta,s}} \bigg\{\nabla b^{n}(\< x^{n}_{s_{m+1+m_2}}, e_1\>+\eps w_1B^{H_1}(s_{m+1+m_2}),\dots)\cdot\begin{pmatrix}
       \chi_1(\theta-s_{m+1+m_2}) \\                                                                                                                                                          
      \vdots \\
       \chi_d(\theta-s_{m+1+m_2})
     \end{pmatrix}\\
     &\quad\cdot\prod_{l=1}^{m+1+m_2-1}\nabla b^{n}(\< x^{n}_{s_l}, e_1\>+\eps w_1B^{H_1}(s_l),\dots)\cdot
     \begin{pmatrix}
       \chi_1(s_{l+1}-s_l) \\                                                                                                                                                          
      \vdots \\
       \chi_d(s_{l+1}-s_l)
     \end{pmatrix}\bigg\}\\
    &\quad\left.\left.\cdot\begin{pmatrix}
       \chi_1(s_{m+1+1}-s_{m+1}) \\                                                                                                                                                          
      \vdots \\
       \chi_d(s_{m+1+1}-s_{m+1})
     \end{pmatrix}ds_{m+1+m_2}\dots ds_{m+1+1}\right]\right\}
     \begin{pmatrix}
       \chi_1(s_{m+1}-s_m) \\                                                                                                                                                          
      \vdots \\
       \chi_d(s_{m+1}-s_m)
     \end{pmatrix}ds_{m+1}\\
 &=\sum_{m_2=1}^\infty\int_{\Delta^{m_2}_{\theta,\theta'}}\bigg\{\nabla b^{n}(\< x^{n}_{s_{m+m_2}}, e_1\>+\eps w_1B^{H_1}(s_{m+m_2}),\dots)\cdot\begin{pmatrix}
      \chi_1(\theta-s_{m+m_2}) \\                                                                                                                                                          
     \vdots \\
      \chi_d(\theta-s_{m+m_2})
    \end{pmatrix}\\
    &\quad\cdot\prod_{l=1}^{m+m_2-1}\nabla b^{n}(\< x^{n}_{s_l}, e_1\>+\eps w_1B^{H_1}(s_l),\dots)\cdot
     \begin{pmatrix}
       \chi_1(s_{l+1}-s_l) \\                                                                                                                                                          
      \vdots \\
       \chi_d(s_{l+1}-s_l)
     \end{pmatrix}\bigg\}\\
     &\quad\cdot\begin{pmatrix}
       \chi_1(s_{m+1}-s_m) \\                                                                                                                                                          
      \vdots \\
       \chi_d(s_{m+1}-s_m)
     \end{pmatrix}ds_{m+m_2}\dots ds_{m+1},    
\end{align*}

where we applied the index shift $m_2:=m_2+1$ together with Fubini's theorem in the last line. Thus, defining $m_2:=m$ and plugging in, we finally achive

\begin{align*}
&\DM_\theta x^{n}(t)-\DM_{\theta'} x^{n}(t)\\
&=\Big(\DM_\theta x^{n}(\theta')-1\Big)\\
&\quad+\sum_{m=1}^\infty\int_{\Delta^m_{\theta',t}} \bigg\{\nabla b^{n}(\< x^{n}_{s_m}, e_1\>+\eps w_1B^{H_1}(s_m),\dots)\cdot \begin{pmatrix}
     \chi_1(\theta-s_m)-\chi_1(\theta'-s_m) \\                                                                                                                                                          
     \vdots \\
     \chi_d(\theta-s_m)-\chi_d(\theta'-s_m)
    \end{pmatrix}\\
 &\quad\cdot\prod_{l=1}^{m-1}\nabla b^{n}(\< x^{n}_{s_l}, e_1\>+\eps w_1B^{H_1}(s_l),\dots)\cdot\begin{pmatrix}
       \chi_1(s_{l+1}-s_l) \\                                                                                                                                                          
      \vdots \\
       \chi_d(s_{l+1}-s_l)
     \end{pmatrix}\bigg\}ds_m\dots ds_1\\
&\quad+\sum_{m_1,m_2=1}^\infty\int_{\Delta^{m_1+m_2}_{\theta',t}\times\Delta^{m_2}_{\theta,\theta'}}\nabla b^{n}(\< x^{n}_{s_{m_1+m_2}}, e_1\>+\eps w_1B^{H_1}(s_{m_1+m_2}),\dots)\cdot\Big[\begin{pmatrix}
      \chi_1(\theta-s_{m_1+m_2}) \\                                                                                                                                                          
     \vdots \\
      \chi_d(\theta-s_{m_1+m_2})
    \end{pmatrix}\Big]\\
    &\quad\cdot\prod_{l=1}^{m_1+m_2-1}\nabla b^{n}(\< x^{n}_{s_l}, e_1\>+\eps w_1B^{H_1}(s_l),\dots)\cdot\begin{pmatrix}
      \chi_1(s_{l+1}-s_l) \\                                                                                                                                                          
      \vdots \\
      \chi_d(s_{l+1}-s_l)
     \end{pmatrix}\bigg\}ds_{m_1+m_2}\dots ds_1.
\end{align*}

This can be easily reformulated into \eqref{3MalliavinDiffIteration}.

\end{proof}

These representations will be used to prove the following lemma. Note that, by definition of $\chi_j(z)$, we have that

\begin{align}
 |\Hfct_{j_1,\dots,j_m}(s,\theta)|&\leq (1+r)^m\label{3HfunctionEstimation}\\
 |\tilde\Hfct_{j_1,\dots,j_m}(s,\theta,\theta')|&\leq (1+r)^{m-1}|\theta-\theta'|^{\frac{1}{2}}.\label{3HtildefunctionEstimation}
\end{align}

The next Lemma is an application of Girsanov's Theorem.

\begin{lem}\label{3GirsanovWdk}
 The process $(\overline W^{n}(t))_{t\in[0,T]}$ defined by
 \begin{align}\label{3Wdk}
  \overline W^{n}(t)&:=W(t)+\int_0^tb^{n}\Big(\< x^{n}_s,e_1\>+\eps w_1B^{H_1}(s),\dots,\< x^{n}_s,e_d\>+\eps w_dB^{H_d}(s)\Big)ds
 \end{align}
 is a Brownian motion under the measure $\overline\PP^{n}$ given by
 \begin{align}\label{3dPdkDP}
 \begin{split}
  \frac{d\overline\PP^{n}}{d\PP\otimes\PP'}\Big|_{\F_t}&=\mathcal{E}\bigg( -\int_0^\cdot b^{n}\Big(\< x^{n}_s,e_1\>+\eps w_1 B^{H_1}(s),\dots,\< x^{n}_s,e_d\>\eps w_d B^{H_d}(s)\Big)dW(s)\bigg)_t.
 \end{split}
 \end{align}
\end{lem}

\begin{proof}
 By Lemma \ref{3ExistAndUniqApprox}, the process $(x^{n}(t))_{t\in[0,T]}$ is well defined and therefore, so is $(\overline W^{n}(t))_{t\in[0,T]}$. Note that Novikov's condition is satisfied, since we even have for all $\alpha\in\R$:
\begin{align}\label{3Novikov}
\begin{split}
 &\bar E\bigg[\exp\Big\{\alpha\int_0^T|b^n(\<x_s,e_1\>+\eps w_1 B^{H_1}(s),\dots,\<x_s,e_d\>+\eps w_d B^{H_d}(s))|^2ds\Big\}\bigg]\\
  &\leq e^{|\alpha|T\sum_{i=1}^d\lVert b_i\rVert_\infty^2}<\infty.
\end{split}
\end{align}
Therefore, the Dol\'eons-Dade exponential

\begin{align*}
&\mathcal{E}\bigg( -\int_0^\cdot b^{n}\big(\< x^{n}_s,e_1\>+\eps w_1 B^{H_1}(s),\dots,\< x^{n}_s,e_d\>\eps w_d B^{H_d}(s)\big)dW(s)\bigg)_t\\
&:= \exp \bigg\{ -\int_0^\cdot b^{n}\big(\< x^{n}_s,e_1\>+\eps w_1 B^{H_1}(s),\dots,\< x^{n}_s,e_d\>\eps w_d B^{H_d}(s)\big)dW(s)\\
&\qquad-\frac{1}{2} \int_0^t\Big|b^{n}\big(\< x^{n}_s,e_1\>+\eps w_1 B^{H_1}(s),\dots,\< x^{n}_s,e_d\>\eps w_d B^{H_d}(s)\big)\Big|^2ds  \bigg\}, \quad t\in [0,T],
\end{align*}

is an $(\F_t)_{t\in[0,T]}$-martingale and it is in $L^p$ for all $p>0$. The result now follows from Girsanov's theorem.
\end{proof}

\begin{cor}\label{3GirsanovWdkCor}
 If we define for $\varphi\in M_2$, $i\geq1$
 \begin{align}
 \begin{split}
  F_i(t,\varphi)&:=\eta(0)e_i(0)+\int_{-r}^0(\ind_{\{t+u<0\}}\eta(t+u)+\ind_{\{t+u\geq0\}}\eta(0))e_i(u)du\\
  &\quad+\varphi(0)e_i(0)+\int_{-r}^0\ind_{\{t+u\geq0\}}\varphi(u)e_i(u)du,
 \end{split}
 \end{align}
 then we can rewrite,
 \begin{align*}
  \<x^n_t,e_i\>=F_i(t,\overline W^{n}_t).
 \end{align*}
 This can be seen by noting that $x^n(t)=\eta(0)+W^{n}(t)$ for $t\in[0,T]$ and applying the definition of $\<\cdot,e_i\>$ in $M_2$.
\end{cor}

\begin{rem}\label{3GirsanovWdkRemark}
 Note that \eqref{3dPdkDP} can be equivalently written as equivalently,
 \begin{align*}
  \frac{d\PP\otimes\PP'}{d\overline\PP^{n}}\Big|_{\F_t}&=\mathcal{E}\bigg( \int_0^\cdot b^{n}\Big( F_1(s,\overline W^{n}_s)+\eps w_1B^{H_1}(s),\dots,F_d(s,\overline W^{n}_s)+\eps w_dB^{H_d}(s)\Big)d\overline W^{n}(s)\bigg)_t.
 \end{align*}
 This can be seen since \eqref{3Wdk} and \eqref{3dPdkDP} imply that
 \begin{align*}
  \frac{d\overline\PP^{n}}{d\PP\otimes\PP'}&=\exp\bigg\{-\int_0^T b^{n}\Big(F_1(s,\overline W^{n}_s)+\eps w_1B^{H_1}(s),\dots,F_d(s,\overline W^{n}_s)+\eps w_dB^{H_d}(s)\Big)dW(s)\\
 &\qquad\quad-\frac{1}{2}\int_0^T \Big|b^{n}\Big(F_1(s,\overline W^{n}_s)+\eps w_1B^{H_1}(s),\dots,F_d(s,\overline W^{n}_s)+\eps w_dB^{H_d}(s)\Big)\Big|^2ds\bigg\}\\
 &=\exp\bigg\{-\int_0^T b^{n}\Big(F_1(s,\overline W^{n}_s)+\eps w_1B^{H_1}(s),\dots,F_d(s,\overline W^{n}_s)+\eps w_dB^{H_d}(s)\Big)d\overline W^{n}(s)\\
 &\qquad\quad+\frac{1}{2}\int_0^T \Big|b^{n}\Big(F_1(s,\overline W^{n}_s)+\eps w_1B^{H_1}(s),\dots,F_d(s,\overline W^{n}_s)+\eps w_dB^{H_d}(s)\Big)\Big|^2ds\bigg\},
 \end{align*}
 which proves that
 \begin{align*}
  \frac{d\PP\otimes\PP'}{d\overline\PP^{n}}=\bigg(\frac{d\overline\PP^{n}}{d\PP\otimes\PP'}\bigg)^{-1}&=\mathcal{E}\bigg( \int_0^\cdot b^{n}\Big(F_1(s,\overline W^{n}_s)+\eps w_1B^{H_1}(s),\\
 &\quad\quad\dots,F_d(s,\overline W^{n}_s)+\eps w_dB^{H_d}(s)\Big)d\overline W^{n}(s)\bigg)_T.
 \end{align*}
\end{rem}

Before we can prove a theorem that gives the $\omega'$-wise compactness criterion, we state some assumptions that will be assumed to hold during the entire paper.

\begin{ass}\label{3StandingAssumptions}
 Suppose, it exists a $\delta_H\in(0,1)$ such that the following conditions are satisfied:
\begin{itemize}
 \item[(T)] The time interval considered is small enough, meaning that there exists a $\delta_T\in(0,\eps^{\frac{3}{\delta_H}})$ such that
 \begin{align}\label{3Tcondition}
  T=T(\eps)<\big(\eps^3-\delta_T\big)^{\frac{1}{\delta_H}}.
 \end{align}
 \item[(H)] The Hurst parameters $(H_k)_{k=1\dots,d}$ corresponding to the terms $B^{H_1},\dots,B^{H_d}$ that appear in the pertunbation (see Lemma \ref{3CylfBM}) satisfies
 \begin{align}\label{3Hcondition}
  H_k<\frac{1-\delta_H}{3}\quad\text{for }k=1,\dots,d.
 \end{align}
 \item[(A)] It holds
 \begin{align}\label{3Acondition}
  \sum_{j=1}^dA_j<1,
 \end{align}
 where $A_j$, $j=1,\dots,d$ are defined by
 \begin{align}\label{3ADefinition}
  A_j=\frac{48\sqrt{2}(1+r)\Gamma(\delta_H)}{\sqrt{\pi}}C_{j}^{-\frac{3}{2}}|w_{j}|^{-3}\|b_{j}\|_{L^1},
 \end{align}
\end{itemize}
\end{ass}

\begin{rem}
 Assumption $(A)$ is automatically satisfied if one requires for all $j=1,\dots,d$
 \begin{align*}
 \|b_{j}\|_{L^1}\leq \sqrt{\pi}2^{-j}\left(48\sqrt{2}(1+r)\Gamma(\delta_H)\right)^{-1}C_{j}^{\frac{3}{2}}|w_{j}|^{3}.
 \end{align*}
\end{rem}

\begin{rem}
 Assumptions $(A)$ and $(T)$ imply together that
 \begin{align*}
  \eps ^{-3}t^{\delta_H}\sum_{j=1}^dA_j<1.
 \end{align*}
 This can be seen since, as $t^{\delta_H}\leq T^{\delta_H}<\eps^3$.
 %On the other hand, for $\eps\geq 1$, $t\leq T<\eps$ and therefore $t\vee t^{\delta_H}<T\vee T^{\delta_H}<\eps\vee \eps^{\delta_H}=\eps$.
\end{rem}

\begin{lem}\label{3VI_relcomp}
Fix $t\in[0,T]$ and assume Assumptions \ref{3StandingAssumptions} to be satisfied. Then, for almost every $\omega'\in\Omega'$, there exists a subsequence $(n_k(\omega'))_{k\geq1}$ such that

\begin{enumerate}
 \item we have
 \begin{align}\label{3CompactnessMalliavin1}
 \sup_{k\geq 1} E\bigg[\int_0^t|\DM_{\theta} x^{n_k(\omega')}(t,\cdot,\omega')|^2d\theta\bigg] <\infty,
 \end{align}
 \item there exists a $\beta\in (0,1/2)$ such that
 \begin{align}\label{3CompactnessMalliavin2}
  \sup_{k\geq 1} \int_0^t \int_0^t \frac{E[|\DM_\theta x^{n_k(\omega')}(t,\cdot,\omega') - \DM_{\theta'} x^{n_k(\omega')}(t,\cdot,\omega')|^2]}{|\theta' - \theta|^{1+2\beta}} d\theta' d\theta <\infty
 \end{align}
\end{enumerate}
In particular, the sequence $x^{n_k(\omega')}(t,\cdot,\omega')$ is relatively compact in $L^2(\Omega,\F^W_t,P)$ for almost every $\omega'\in\Omega'$.
\end{lem}

\begin{proof}
 %For simplicity, we prove the result only for $\eps<1$. For $\eps\geq1$, the proof follows almost the exact same steps.
 In order to prove this result, we are going to show that
 \begin{itemize}
  \item[(1')] it holds
  \begin{align}\label{3CompactnessMalliavin3}
   \sup_{n\geq 1}E'\bigg[E\bigg[\int_0^t|\DM_{\theta} x^{n}(t)-1|^2d\theta\bigg]\bigg]<\infty,
  \end{align}
  \item[(2')] with $\beta\in (0,1/2)$ from \eqref{3CompactnessMalliavin2}, we have
  \begin{align}\label{3CompactnessMalliavin4}
   \sup_{n\geq 1}E'\bigg[\int_0^t \int_0^t \frac{E[|\DM_\theta x^{n}(t) - \DM_{\theta'} x^{n}(t)|^2]}{|\theta' - \theta|^{1+2\beta}} d\theta' d\theta\bigg]<\infty.
  \end{align}
 \end{itemize}
 Once we have shown \eqref{3CompactnessMalliavin3} and \eqref{3CompactnessMalliavin4}, it follows by simple application of Fatou's lemma that
 \begin{align*}
  &E'\bigg[\liminf_{n\geq 1}\Big(E\bigg[\int_0^t|\DM_{\theta} x^{n}(t)-1|^2d\theta\bigg]+\int_0^t \int_0^t \frac{E[|\DM_\theta x^{n}(t) - \DM_{\theta'} x^{n}(t)|^2]}{|\theta' - \theta|^{1+2\beta}} d\theta' d\theta\Big)\bigg]\\
  &\quad\leq \liminf_{n\geq 1}E'\bigg[E\bigg[\int_0^t|\DM_{\theta} x^{n}(t)-1|^2d\theta\bigg]+\int_0^t \int_0^t \frac{E[|\DM_\theta x^{n}(t) - \DM_{\theta'} x^{n}(t)|^2]}{|\theta' - \theta|^{1+2\beta}} d\theta' d\theta\bigg]\\
  &\quad\leq \sup_{n\geq 1}E'\bigg[E\bigg[\int_0^t|\DM_{\theta} x^{n}(t)-1|^2d\theta\bigg]\bigg]+\sup_{n\geq 1}E'\bigg[\int_0^t \int_0^t \frac{E[|\DM_\theta x^{n}(t) - \DM_{\theta'} x^{n}(t)|^2]}{|\theta' - \theta|^{1+2\beta}} d\theta' d\theta\bigg]\\
  &\quad<\infty.
 \end{align*}
 But since any random variable with finite expectation is finite almost everywhere, this means that for $P'$-a.e. $\omega'\in\Omega'$
 \begin{align*}
  &\liminf_{n\geq 1}\Big(E\bigg[\int_0^t|\DM_{\theta} x^{n}(t,\cdot,\omega')-1|^2d\theta\bigg]+\int_0^t \int_0^t \frac{E[|\DM_\theta x^{n}(t,\cdot,\omega') - \DM_{\theta'} x^{n}(t,\cdot,\omega')|^2]}{|\theta' - \theta|^{1+2\beta}} d\theta' d\theta\Big)\\
  &\quad<\infty.
 \end{align*}
 In other words, there exists a subset $\Omega'_0\subseteq\Omega'$ with $P'(\Omega'_0)=1$ such that for all $\omega'\in\Omega'_0$ the above holds true. In particular, for every $\omega'\in\Omega'_0$, there exists a subsequence $(n_k(\omega'))_{k\geq1}$ such that $n_k(\omega')\rightarrow\infty$, as $n\rightarrow\infty$, and
 
 \begin{align*}
  &\sup_{k\geq 1}\Big(E\bigg[\int_0^t|\DM_{\theta} x^{n_k(\omega')}(t,\cdot,\omega')-1|^2d\theta\bigg]\\
  &\qquad+\int_0^t \int_0^t \frac{E[|\DM_\theta x^{n_k(\omega')}(t,\cdot,\omega') - \DM_{\theta'} x^{n_k(\omega')}(t,\cdot,\omega')|^2]}{|\theta' - \theta|^{1+2\beta}} d\theta' d\theta\Big)\\
  &\quad<\infty.
 \end{align*}
 
 This proves \eqref{3CompactnessMalliavin1} and \eqref{3CompactnessMalliavin2}. The relative compactness follows from Lemma \ref{3VI_compactcrit}.\\
 \par

 \textbf{Proof of \eqref{3CompactnessMalliavin3}:} Representation \eqref{3MalliavinIteration} yields
 
\begin{align}\label{3IterativeRep}
\begin{split}
 E'[E[|\DM_{\theta} x^{n}(t)-1|^2]]&=E'\bigg[E\Big[\Big|\sum_{m=1}^\infty\sum_{j_1,\dots,j_m=1}^d\int_{\Delta^m_{\theta,t}}\Big\{\Hfct_{j_1,\dots,j_m}(s,\theta)\\
 &\quad\cdot \prod_{l=1}^mb_{j_l,n}'(\<x^{n}_{s_l}, e_{j_l}\>+\eps w_{j_l}B^{H_{j_l}}(s_l))\Big\}ds\Big|^2\Big]\bigg]
\end{split}
\end{align}

Now applying Girsanov's theorem %in the sense of Lemma \ref{3GirsanovWdk}, Corollary \ref{3GirsanovWdkCor} and Remark \ref{3GirsanovWdkRemark},
together with Corollary \ref{3GirsanovWdkCor}, H\"{o}lder's theorem, we get

\begin{align*}
 &E'\bigg[E\bigg[|\DM_{\theta} x^{n}(t)-1|^2\bigg]\bigg]\\
 &=E_{\overline\PP^{n}}\Big[\frac{d\PP\otimes\PP'}{d\overline\PP^{n}}\Big|\sum_{m=1}^\infty\sum_{j_1,\dots,j_m=1}^d\int_{\Delta^m_{\theta,t}}\Big\{\Hfct_{j_1,\dots,j_m}(s,\theta)\prod_{l=1}^mb_{j_l,n}'(\<x^{n}_{s_l}, e_{j_l}\>+\eps w_{j_l}B^{H_{j_l}}(s_l))\Big\}ds\Big|^2\Big]\\
 &\leq E_{\overline\PP^{n}}\Big[\Big|\frac{d\PP\otimes\PP'}{d\overline\PP^{n}}\Big|^2\Big]^{\frac{1}{2}}E_{\overline\PP^{n}}\Big[\Big|\sum_{m=1}^\infty\sum_{j_1,\dots,j_m=1}^d\int_{\Delta^m_{\theta,t}}\Big\{\Hfct_{j_1,\dots,j_m}(s,\theta)\\
 &\quad\cdot \prod_{l=1}^mb_{j_l,n}'(F_{j_l}(s_l,\overline W^{n}_{s_l})+\eps w_{j_l}B^{H_{j_l}}(s_l))\Big\}ds\Big|^4\Big]^{\frac{1}{2}}
\end{align*}

Note again that, for each $\omega'\in\Omega'$, the law of the process $\overline W^{n}$ under $\overline\PP^{n}$ is the same as the law of the process $W$ under $\PP$. Moreover, we have the following estimate:

\begin{align*}
 E_{\overline\PP^{n}}\Big[\Big|\frac{d\PP\otimes\PP'}{d\overline\PP^{n}}\Big|^2\Big]^{\frac{1}{2}}&=E_{\overline\PP^{n}}\Big[\Big|\mathcal{E}\bigg( \int_0^\cdot b^{n}\Big(F_1(r,\overline W^{n}_r)+\eps w_1B^{H_1}(r),\\
 &\quad\dots,F_d(r,\overline W^{n}_r)+\eps w_1B^{H_d}(r)\Big)d\overline W^{n}(r)\bigg)_T\Big|^2\Big]^{\frac{1}{2}}\\
 &=E\Big[\mathcal{E}\bigg( \int_0^\cdot 2b^{n}\Big(F_1(r,W_r)+\eps w_1B^{H_1}(r),\dots\Big)dW(r)\bigg)_T\\
 &\quad\cdot\exp\Big(\int_0^T \Big|b^{n}\Big(F_1(r,W_{r})+\eps w_1B^{H_1}(r),\dots,\Big)\Big|^2dr\Big)\Big]^{\frac{1}{2}}\\
 &\leq\exp\Big(\frac{1}{2}M^2T\Big),
\end{align*}

since $0<\exp\Big(\int_0^T|b^{n}(\cdot)|^2ds\Big)\leq\exp\Big(M^2T\Big)$, $\mathcal{E}(\cdot)\geq0$ and $E[\mathcal{E}(\cdot)_T]=1$. This yields

\begin{align*}
 E'\bigg[E\bigg[|\DM_{\theta} x^{n}(t)-1|^2\bigg]\bigg]&\leq e^{\frac{1}{2}M^2T}\bar E\Big[\Big|\sum_{m=1}^\infty\sum_{j_1,\dots,j_m=1}^d\int_{\Delta^m_{\theta,t}}\Big\{\Hfct_{j_1,\dots,j_m}(s,\theta)\\
 &\quad\cdot \prod_{l=1}^mb_{j_l,n}'(F_{j_l}(s_l,W_{s_l})+\eps w_{j_l}B^{H_{j_l}}(s_l))\Big\}ds\Big|^4\Big]^{\frac{1}{2}}.
\end{align*}

Now we rewrite the above in terms of the norm $\|\cdot\|_{L^4(\bar\Omega)}$ and apply that $\|\sum\dots\|\leq\sum\|\dots\|$:

\begin{align*}
 E'\bigg[E\bigg[|\DM_{\theta} x^{n}(t)-1|^2\bigg]\bigg]&\leq e^{\frac{1}{2}M^2T}\bar E\Big[\Big|\sum_{m=1}^\infty\sum_{j_1,\dots,j_m=1}^d\int_{\Delta^m_{\theta,t}}\Big\{\Hfct_{j_1,\dots,j_m}(s,\theta)\\
 &\quad\cdot \prod_{l=1}^mb_{j_l,n}'(F_{j_l}(s_l,W_{s_l})+\eps w_{j_l}B^{H_{j_l}}(s_l))\Big\}ds\Big|^4\Big]^{\frac{1}{2}}\\
%  &\leq e^{\frac{1}{2}M^2T}E'\Bigg[E\bigg[\bigg(\sum_{m=1}^\infty\sum_{j_1,\dots,j_m=1}^d\Big|\int_{\Delta^m_{\theta,t}}\Big\{\Hfct_{j_1,\dots,j_m}(s,\theta)\\
%  &\quad\cdot \prod_{l=1}^mb_{j_l,n}'(F_{j_l}(s_l,W_{s_l})+\eps w_{j_l}B^{H_{j_l}}(s_l))\Big\}ds\Big|\bigg)^4\bigg]^{\frac{1}{2}}\Bigg]\\
%  &= e^{\frac{1}{2}M^2T}\bar E\Bigg[\bigg(\sum_{m=1}^\infty\sum_{j_1,\dots,j_m=1}^d\Big|\int_{\Delta^m_{\theta,t}}\Big\{\Hfct_{j_1,\dots,j_m}(s,\theta)\\
%  &\quad\cdot \prod_{l=1}^mb_{j_l,n}'(F_{j_l}(s_l,W_{s_l})+\eps w_{j_l}B^{H_{j_l}}(s_l))\Big\}ds\Big|\bigg)^4\Bigg]^{\frac{1}{2}}\\
 &=e^{\frac{1}{2}M^2T}\bigg\|\sum_{m=1}^\infty\sum_{j_1,\dots,j_m=1}^d\int_{\Delta^m_{\theta,t}}\Big\{\Hfct_{j_1,\dots,j_m}(s,\theta)\\
 &\quad\cdot \prod_{l=1}^mb_{j_l,n}'(F_{j_l}(s_l,W_{s_l})+\eps w_{j_l}B^{H_{j_l}}(s_l))\Big\}ds\bigg\|_{L^4(\bar\Omega)}^2\\
 &\leq e^{\frac{1}{2}M^2T}\bigg(\sum_{m=1}^\infty\sum_{j_1,\dots,j_m=1}^d\bigg\|\int_{\Delta^m_{\theta,t}}\Big\{\Hfct_{j_1,\dots,j_m}(s,\theta)\\
 &\quad\cdot \prod_{l=1}^mb_{j_l,n}'(F_{j_l}(s_l,W_{s_l})+\eps w_{j_l}B^{H_{j_l}}(s_l))\Big\}ds\bigg\|_{L^4(\bar\Omega)}\bigg)^2\\
 %&=e^{\frac{1}{2}M^2T}\bigg(\sum_{m=1}^\infty\sum_{j_1,\dots,j_m=1}^d\bar E\bigg[\Big(\Big|\int_{\Delta^m_{\theta,t}}\Big\{\Hfct_{j_1,\dots,j_m}(s,\theta)\\
 %&\quad\cdot \prod_{l=1}^mb_{j_l,n}'(F_{j_l}(s_l,W_{s_l})+\eps w_{j_l}B^{H_{j_l}}(s_l))\Big\}ds\Big|\Big)^4\bigg]^{\frac{1}{4}}\bigg)^2\\
 &=e^{\frac{1}{2}M^2T}\bigg(\sum_{m=1}^\infty\sum_{j_1,\dots,j_m=1}^d\bar E\bigg[\Big|\int_{\Delta^m_{\theta,t}}\Big\{\Hfct_{j_1,\dots,j_m}(s,\theta)\\
 &\quad\cdot \prod_{l=1}^mb_{j_l,n}'(F_{j_l}(s_l,W_{s_l})+\eps w_{j_l}B^{H_{j_l}}(s_l))\Big\}ds\Big|^4\bigg]^{\frac{1}{4}}\bigg)^2.
\end{align*}

For simpler notation, we define the process $Y:[0,T]\times\bar\Omega\rightarrow\R^d$ by
\begin{align}\label{3Yprocess}
\begin{split}
 Y(s)&:=(Y_1(s),\dots,Y_d(s))^\top\\
 Y_i(s)&:=F_i(s,W_s)+\eps w_iB^{H_i}(s).
\end{split}
\end{align}

Then the above estimate becomes

\begin{align}\label{3IntermediateEstimate1}
\begin{split}
 &E'\bigg[E\bigg[|\DM_{\theta} x^{n}(t)-1|^2\bigg]\bigg]\\
 &\leq e^{\frac{1}{2}M^2T}\bigg(\sum_{m=1}^\infty\sum_{j_1,\dots,j_m=1}^d\bar E\bigg[\Big|\int_{\Delta^m_{\theta,t}}\Hfct_{j_1,\dots,j_m}(s,\theta) \prod_{l=1}^mb_{j_l,n}'(Y_{j_l}(s_l))ds\Big|^4\bigg]^{\frac{1}{4}}\bigg)^2.
\end{split}
\end{align}

Let us, for a moment, only consider the term inside the expectation $\bar E$. Application of Lemma \ref{3IntegralProductShuffle} and rewriting $a^4=(a^2)^2$ yields

\begin{align*}
 &\Big|\int_{\Delta^m_{\theta,t}}\Hfct_{j_1,\dots,j_m}(s,\theta) \prod_{l=1}^mb_{j_l,n}'(Y_{j_l}(s_l))ds_m\dots ds_1\Big|^4\\
 &=\bigg(\sum_{\sigma\in S(m,m)}\int_{\Delta^{2m}_{\theta,t}}\Big\{\Hfct_{j_1,\dots,j_m}(s_{\sigma(1)},\dots,s_{\sigma(m)},\theta)\Hfct_{j_1,\dots,j_m}(s_{\sigma(m+1)},\dots,s_{\sigma(2m)},\theta)\\
 &\quad\cdot \prod_{l=1}^mb_{j_l,n}'(Y_{j_l}(s_{\sigma(l)}))\prod_{l=m+1}^{2m}b_{j_{l-m},n}'(Y_{j_{l-m}}(s_{\sigma(l)}))\Big\}ds\bigg)^2
\end{align*}

We define

\begin{align}\label{3H2tensor}
\begin{split}
 &\Hfct_{j_1,\dots,j_m}^{\otimes2}(\sigma;s_{1},\dots,s_{2m},\theta)\\
 &:=\Hfct_{j_1,\dots,j_m}(\tilde s_{1},\dots,\tilde s_{m},\theta)\Hfct_{j_1,\dots,j_m}(\tilde s_{m+1},\dots,\tilde s_{2m},\theta)\Big|_{\tilde s_i=s_{\sigma(i)},\,i=1,\dots,2m}.
\end{split}
\end{align}

Moreover, defining $\tilde l:=\sigma(l)$, we note that

\begin{align*}
 \prod_{l=1}^mb_{j_l,n}'(Y_{j_l}(s_{\sigma(l)}))\prod_{l=m+1}^{2m}b_{j_{l-m},n}'(Y_{j_{l-m}}(s_{\sigma(l)}))=\prod_{l=1}^{2m}b_{j_{[\sigma^{-1}(l)]_{\mo m}},n}'(Y_{j_{[\sigma^{-1}(l)]_{\mo m}}}(s_{l}))
\end{align*}

Plugging this in and applying Lemma \ref{3IntegralProductShuffle} again yields

\begin{align*}
 &\Big|\int_{\Delta^m_{\theta,t}}\Hfct_{j_1,\dots,j_m}(s,\theta) \Big(\prod_{l=1}^mb_{j_l,n}'(Y(s_l))\Big)ds_m\dots ds_1\Big|^4\\
 &=\bigg(\sum_{\sigma_1\in S(m,m)}\int_{\Delta^{2m}_{\theta,t}}\Big\{\Hfct_{j_1,\dots,j_m}^{\otimes2}(\sigma_1;s_{1},\dots,s_{2m},\theta)\\
 &\quad\cdot \prod_{l=1}^{2m}b_{j_{[\sigma_1^{-1}(l)]_{\mo m}},n}'(Y_{j_{[\sigma_1^{-1}(l)]_{\mo m}}}(s_{l}))\Big\}ds_{2m}\dots ds_1\bigg)\\
 &\quad\cdot\bigg(\sum_{\sigma_2\in S(m,m)}\int_{\Delta^{2m}_{\theta,t}}\Big\{\Hfct_{j_1,\dots,j_m}^{\otimes2}(\sigma_2;s_{2m+1},\dots,s_{4m},\theta)\\
 &\quad\cdot \prod_{l=2m+1}^{4m}b_{j_{[\sigma_2^{-1}(l-2m)]_{\mo m}},n}'(Y_{j_{[\sigma_2^{-1}(l-2m)]_{\mo m}}}(s_{l}))\Big\}ds_{4m}\dots ds_{2m+1}\bigg)%\\
%  &=\sum_{\sigma_1,\sigma_2\in S(m,m)}\sum_{\tau\in S(2m,2m)}\int_{\Delta^{4m}_{\theta,t}}\Big\{\Hfct_{j_1,\dots,j_m}^{\otimes2}(\sigma_1;s_{\tau(1)},\dots,s_{\tau(2m)},\theta)\\
%  &\quad\cdot \Hfct_{j_1,\dots,j_m}^{\otimes2}(\sigma_{2};s_{\tau(2m+1)},\dots,s_{\tau(4m)},\theta)\prod_{l=1}^{2m}b_{j_{[\sigma_1^{-1}(l)]_{\mo m}},n}'(Y_{j_{[\sigma_1^{-1}(l)]_{\mo m}}}(s_{\tau(l)}))\\
%  &\quad\cdot \prod_{l=2m+1}^{4m}b_{j_{[\sigma_2^{-1}(l-2m)]_{\mo m}},n}'(Y_{j_{[\sigma_2^{-1}(l-2m)]_{\mo m}}}(s_{\tau(l)}))\Big\}ds_{4m}\dots ds_1
\end{align*}

\begin{align*}
%  &\Big|\int_{\Delta^m_{\theta,t}}\Hfct_{j_1,\dots,j_m}(s,\theta) \Big(\prod_{l=1}^mb_{j_l,n}'(Y(s_l))\Big)ds_m\dots ds_1\Big|^4\\
%  &=\bigg(\sum_{\sigma_1\in S(m,m)}\int_{\Delta^{2m}_{\theta,t}}\Big\{\Hfct_{j_1,\dots,j_m}^{\otimes2}(\sigma_1;s_{1},\dots,s_{2m},\theta)\\
%  &\quad\cdot \prod_{l=1}^{2m}b_{j_{[\sigma_1^{-1}(l)]_{\mo m}},n}'(Y_{j_{[\sigma_1^{-1}(l)]_{\mo m}}}(s_{l}))\Big\}ds_{2m}\dots ds_1\bigg)\\
%  &\quad\cdot\bigg(\sum_{\sigma_2\in S(m,m)}\int_{\Delta^{2m}_{\theta,t}}\Big\{\Hfct_{j_1,\dots,j_m}^{\otimes2}(\sigma_2;s_{2m+1},\dots,s_{4m},\theta)\\
%  &\quad\cdot \prod_{l=2m+1}^{4m}b_{j_{[\sigma_2^{-1}(l-2m)]_{\mo m}},n}'(Y_{j_{[\sigma_2^{-1}(l-2m)]_{\mo m}}}(s_{l}))\Big\}ds_{4m}\dots ds_{2m+1}\bigg)\\
 \hspace{1.5cm}&=\sum_{\sigma_1,\sigma_2\in S(m,m)}\sum_{\tau\in S(2m,2m)}\int_{\Delta^{4m}_{\theta,t}}\Big\{\Hfct_{j_1,\dots,j_m}^{\otimes2}(\sigma_1;s_{\tau(1)},\dots,s_{\tau(2m)},\theta)\\
 &\quad\cdot \Hfct_{j_1,\dots,j_m}^{\otimes2}(\sigma_{2};s_{\tau(2m+1)},\dots,s_{\tau(4m)},\theta)\prod_{l=1}^{2m}b_{j_{[\sigma_1^{-1}(l)]_{\mo m}},n}'(Y_{j_{[\sigma_1^{-1}(l)]_{\mo m}}}(s_{\tau(l)}))\\
 &\quad\cdot \prod_{l=2m+1}^{4m}b_{j_{[\sigma_2^{-1}(l-2m)]_{\mo m}},n}'(Y_{j_{[\sigma_2^{-1}(l-2m)]_{\mo m}}}(s_{\tau(l)}))\Big\}ds_{4m}\dots ds_1
\end{align*}

Note that, by \eqref{3H2tensor},
\begin{align*}
 &\Hfct_{j_1,\dots,j_m}^{\otimes2}(\sigma_1;s_{\tau(1)},\dots,s_{\tau(2m)},\theta)\\
 &=\Hfct_{j_1,\dots,j_m}^{\otimes2}(\sigma_1;\hat s_1,\dots,\hat s_{2m},\theta)\Big|_{\hat s_i=s_{\tau(i)},\,i=1,\dots,2m}\\
 &=\Hfct_{j_1,\dots,j_m}(\tilde s_{1},\dots,\tilde s_{m},\theta)\Hfct_{j_1,\dots,j_m}(\tilde s_{m+1},\dots,\tilde s_{2m},\theta)\Big|_{\tilde s_i=\hat s_{\sigma_1(i)}=s_{\tau(\sigma_1(i))},\,i=1,\dots,2m}\\
 &=\Hfct_{j_1,\dots,j_m}(s_{\tau(\sigma_1(1))},\dots,s_{\tau(\sigma_1(m))},\theta)\Hfct_{j_1,\dots,j_m}(s_{\tau(\sigma_1(m+1))},\dots,s_{\tau(\sigma_1(2m))},\theta).
\end{align*}

A similar representation holds for $\Hfct_{j_1,\dots,j_m}^{\otimes2}(\sigma_2;s_{\tau(2m+1)},\dots,s_{\tau(4m)},\theta)$. So we can define

\begin{align}\label{3H4tensor}
\begin{split}
 &\Hfct_{j_1,\dots,j_m}^{\otimes4}(\tau,\sigma_1,\sigma_2;s_{1},\dots,s_{4m},\theta)\\
 &:=\Hfct_{j_1,\dots,j_m}^{\otimes2}(\sigma_1;s_{\tau(1)},\dots,s_{\tau(2m)},\theta)\Hfct_{j_1,\dots,j_m}^{\otimes2}(\sigma_2;s_{\tau(2m+1)},\dots,s_{\tau(4m)},\theta)\\
 &=\Hfct_{j_1,\dots,j_m}(s_{\tau(\sigma_1(1))},\dots,s_{\tau(\sigma_1(m))},\theta)\Hfct_{j_1,\dots,j_m}(s_{\tau(\sigma_1(m+1))},\dots,s_{\tau(\sigma_1(2m))},\theta)\\
 &\quad\cdot\Hfct_{j_1,\dots,j_m}(s_{\tau(\sigma_2(2m+1))},\dots,s_{\tau(\sigma_2(3m))},\theta)\Hfct_{j_1,\dots,j_m}(s_{\tau(\sigma_2(3m+1))},\dots,s_{\tau(\sigma_2(4m))},\theta)
\end{split}
\end{align}

In the same spirit as before, we can write

\begin{align*}
 &\prod_{l=1}^{2m}b_{j_{[\sigma_1^{-1}(l)]_{\mo m}},n}'(Y_{j_{[\sigma_1^{-1}(l)]_{\mo m}}}(s_{\tau(l)}))\prod_{l=2m+1}^{4m}b_{j_{[\sigma_2^{-1}(l-2m)]_{\mo m}},n}'(Y_{j_{[\sigma_2^{-1}(l-2m)]_{\mo m}}}(s_{\tau(l)}))\\
 &=\prod_{l=1}^{4m}b_{j_{\alpha_l},n}'(Y_{j_{\alpha_l}}(s_{l}))
\end{align*}

where

\begin{align}\label{3alphalDef}
 \alpha_l=\left[\sigma_{\left\lceil\frac{\tau^{-1}(l)-1}{2m}\right\rceil}^{-1}([\tau^{-1}(l)]_{\mo 2m})\right]_{\mo m}
\end{align}

where we used that
\begin{align*}
 \left\lceil\frac{\tau^{-1}(l)-1}{2m}\right\rceil=\begin{cases}
                                                          1,\quad\text{if }\tau^{-1}(l)\in\{1,\dots,2m\}\\
                                                          2,\quad\text{if }\tau^{-1}(l)\in\{2m+1,\dots,4m\}
                                                         \end{cases}.
\end{align*}

Plugging all in, we get

\begin{align*}
 &\Big|\int_{\Delta^m_{\theta,t}}\Hfct_{j_1,\dots,j_m}(s,\theta) \prod_{l=1}^mb_{j_l,n}'(Y_{j_l}(s_l))ds_m\dots ds_1\Big|^4\\
 &=\sum_{\substack{\sigma_1,\sigma_2\in S(m,m) \\ \tau\in S(2m,2m)}}\int_{\Delta^{4m}_{\theta,t}}\Hfct_{j_1,\dots,j_m}^{\otimes4}(\tau,\sigma_1,\sigma_2;s_{1},\dots,s_{4m},\theta) \prod_{l=1}^{4m}b_{j_{\alpha_l},n}'(Y_{j_{\alpha_l}}(s_l))ds_{4m}\dots ds_1.
\end{align*}

Note each $b_{i,n}$, $i=1,\dots,d$ has bounded support and is smooth. Therefore, $b_{i,n}$ is a Schwartz function, which means that it has an inverse Fourier transform, and thus

\begin{align*}
 b_{i,n}'(y)&=\frac{d}{dy}b_{i,n}(y)=\frac{d}{dy}\F\F^{-1}b_{i,n}(y)=\frac{d}{dy}\int_\R\F^{-1}b_{i,n}(u)e^{-iuy}du\\
 &=\frac{d}{dy}\int_\R\frac{1}{2\pi}\int_\R b_{i,n}(z)e^{izu}dze^{-iuy}du=(2\pi)^{-1}\int_\R b_{i,n}(z)\int_\R e^{-iu(y-z)}dudz.
\end{align*}

We therefore have

\begin{align*}
 &\Big|\int_{\Delta^m_{\theta,t}}\Hfct_{j_1,\dots,j_m}(s,\theta) \prod_{l=1}^mb_{j_l,n}'(Y_{j_l}(s_l))ds_m\dots ds_1\Big|^4\\
 &=\sum_{\substack{\sigma_1,\sigma_2\in S(m,m) \\ \tau\in S(2m,2m)}}\int_{\R^{4m}}\Big\{\Big(\prod_{l=1}^{4m}b_{j_{\alpha_l},n}(z_l)\Big)(2\pi)^{-4m}\int_{\R^{4m}}\int_{\Delta^{4m}_{\theta,t}}\big\{\Hfct_{j_1,\dots,j_m}^{\otimes4}(\tau,\sigma_1,\sigma_2;s_{1},\dots,s_{4m},\theta)\\
 &\quad\cdot\prod_{l=1}^{4m}e^{-iu_l(Y_{j_{\alpha_l}}(s_l)-z_l)}(-iu_l)\big\}dsdu\Big\}dz\\
 &\leq\sum_{\substack{\sigma_1,\sigma_2\in S(m,m) \\ \tau\in S(2m,2m)}}\int_{\R^{4m}}\Big\{\Big|\prod_{l=1}^{4m}b_{j_{\alpha_l},n}(z_l)\Big|\cdot\Big|(2\pi)^{-4m}\int_{\R^{4m}}\int_{\Delta^{4m}_{\theta,t}}\big\{\Hfct_{j_1,\dots,j_m}^{\otimes4}(\tau,\sigma_1,\sigma_2;s_{1},\dots,s_{4m},\theta)\\
 &\quad\cdot\prod_{l=1}^{4m}e^{-iu_l(Y_{j_{\alpha_l}}(s_l)-z_l)}(-iu_l)\big\}dsdu\Big|\Big\}dz.
\end{align*}

Using H\"{o}lder's inequality we achieve the following estimate

\begin{align}
 &\bar E\bigg[\Big|\int_{\Delta^m_{\theta,t}}\Hfct_{j_1,\dots,j_m}(s,\theta) \prod_{l=1}^mb_{j_l,n}'(Y_{j_l}(s_l))ds_m\dots ds_1\Big|^4\bigg]\nonumber\\
 &\leq \sum_{\substack{\sigma_1,\sigma_2\in S(m,m) \\ \tau\in S(2m,2m)}}\int_{\R^{4m}}\Bigg\{\prod_{l=1}^{4m}|b_{j_{\alpha_l},n}(z_l)|\nonumber\\
 &\quad\cdot\bar E\bigg[\bigg|(2\pi)^{-4m}\int_{\R^{4m}}\int_{\Delta^{4m}_{\theta,t}}\Hfct_{j_1,\dots,j_m}^{\otimes4}(\tau,\sigma_1,\sigma_2;s_{1},\dots,s_{4m},\theta)\nonumber\\
 &\quad\cdot\prod_{l=1}^{4m}e^{-iu_l(Y_{j_{\alpha_l}}(s_l)-z_l)}(-iu_l)dsdu\bigg|\bigg]\Bigg\}dz\nonumber\\
 &\leq \sum_{\substack{\sigma_1,\sigma_2\in S(m,m) \\ \tau\in S(2m,2m)}}\int_{\R^{4m}}\Bigg\{\prod_{l=1}^{4m}|b_{j_{\alpha_l},n}(z_l)|\nonumber\\
 &\quad\cdot\bar E\bigg[\bigg|(2\pi)^{-4m}\int_{\R^{4m}}\int_{\Delta^{4m}_{\theta,t}}\Hfct_{j_1,\dots,j_m}^{\otimes4}(\tau,\sigma_1,\sigma_2;s_{1},\dots,s_{4m},\theta)\label{3IntermediateEstimate2}\\
 &\quad\cdot\prod_{l=1}^{4m}e^{-iu_l(Y_{j_{\alpha_l}}(s_l)-z_l)}(-iu_l)dsdu\bigg|^2\bigg]^{\frac{1}{2}}\Bigg\}dz.\nonumber
\end{align}

Again, for a moment we just consider the term inside the expectation $\bar E$. Recall that for an integrable function $g: \R^q \rightarrow \mathbb{C}$ we can write
\begin{align*}
&\left|\int_{\R^{q}} g(u_1, \dots, u_q)   du_1 \dots du_q  \right|^2    \\
&\hspace{0.5cm} = \int_{\R^{q}} g(u_1, \dots, u_q) du_1 \dots du_q \overline{ \int_{\R^{q}} g(u_{q+1}, \dots, u_{2q}) du_{q+1} \dots du_{2q} } \\
&\hspace{0.5cm} = \int_{\R^{q}} g(u_1, \dots, u_q) du_1 \dots du_q  \int_{\R^{q}} \overline{g(u_{q+1}, \dots ,u_{2q})} du_{q+1} \dots du_{2q} \\
&\hspace{0.5cm} = \int_{\R^{q}} g(u_1, \dots, u_q) du_1 \dots du_q (-1)^{q} \int_{\R^{q}}  \overline{g(-u_{q+1}, \dots,- u_{2q})} du_{q+1} \dots du_{2q} \\
&\hspace{0.5cm} = (-1)^{q} \int_{\R^{2q}} g(u_1, \dots, u_q) \overline{g(-u_{q+1}, \dots,- u_{2q})} du_1 \dots du_{2q},
\end{align*}
where we have used the change of variables $(u_{q+1}, \dots, u_{2q}) \mapsto (-u_{q+1}, \dots, - u_{2q})$ in the third equality. Therefore, setting
\begin{align*}
 g(u_1,\dots,u_{4m})&:=\int_{\Delta^{4m}_{\theta,t}}\Big\{\Hfct_{j_1,\dots,j_m}^{\otimes4}(\tau,\sigma_1,\sigma_2;s_{1},\dots,s_{4m},\theta)\\
 &\quad\cdot\prod_{l=1}^{4m}e^{-iu_l(Y_{j_{\alpha_l}}(s_l)-z_l)}(-iu_l)\Big\}ds_{4m}\dots ds_1,
\end{align*}

and applying the rules for the complex conjugate we get
\begin{align*}
 &\overline{g(-u_{4m+1},\dots,-u_{8m})}\\
 &=\overline{\int_{\Delta^{4m}_{\theta,t}}\bigg\{\Hfct_{j_1,\dots,j_m}^{\otimes4}(\tau,\sigma_1,\sigma_2;s_{4m+1},\dots,s_{8m},\theta)}\\
 &\quad\overline{\cdot\prod_{l=4m+1}^{8m}e^{-i(-u_l)(Y_{j_{\alpha_{[l]\mo 4m}}}(s_l)-z_{[l]_{\mo 4m}})}(iu_l)\bigg\}ds_{8m}\dots ds_{4m+1}}\\
 &=\int_{\Delta^{4m}_{\theta,t}}\bigg\{\Hfct_{j_1,\dots,j_m}^{\otimes4}(\tau,\sigma_1,\sigma_2;s_{4m+1},\dots,s_{8m},\theta)\\
 &\quad\cdot\prod_{l=4m+1}^{8m}e^{\overline{iu_l(Y_{j_{\alpha_{[l]\mo 4m}}}(s_l)-z_{[l]_{\mo 4m}})}}(\overline{iu_l})\bigg\}ds_{8m}\dots ds_{4m+1}\\
 &=\quad\int_{\Delta^{4m}_{\theta,t}}\bigg\{\Hfct_{j_1,\dots,j_m}^{\otimes4}(\tau,\sigma_1,\sigma_2;s_{4m+1},\dots,s_{8m},\theta)\\
 &\quad\cdot\prod_{l=4m+1}^{8m}e^{-iu_l(Y_{j_{\alpha_{[l]\mo 4m}}}(s_l)-z_{[l]_{\mo 4m}})}(-iu_l)\bigg\}ds_{8m}\dots ds_{4m+1}.
\end{align*}
This yields

\begin{align*}
 &\bar E\bigg[\bigg|(2\pi)^{-4m}\int_{\R^{4m}}\int_{\Delta^{4m}_{\theta,t}}\Hfct_{j_1,\dots,j_m}^{\otimes4}(\tau,\sigma_1,\sigma_2;s_{1},\dots,s_{4m},\theta)\prod_{l=1}^{4m}e^{-iu_l(Y_{j_{\alpha_l}}(s_l)-z_l)}(-iu_l)dsdu\bigg|^2\bigg]\\
 &=\bar E\bigg[(2\pi)^{-8m}(-1)^{4m}\int_{\R^{8m}}\bigg(\int_{\Delta^{4m}_{\theta,t}}\Hfct_{j_1,\dots,j_m}^{\otimes4}(\tau,\sigma_1,\sigma_2;s_{1},\dots,s_{4m},\theta)\\
 &\quad\cdot\prod_{l=1}^{4m}e^{-iu_l(Y_{j_{\alpha_l}}(s_l)-z_l)}(-iu_l)ds_{4m}\dots ds_1\bigg)\cdot\bigg(\int_{\Delta^{4m}_{\theta,t}}\Hfct_{j_1,\dots,j_m}^{\otimes4}(\tau,\sigma_1,\sigma_2;s_{4m+1},\dots,s_{8m},\theta)\\
 &\quad\cdot\prod_{l=4m+1}^{8m}e^{-iu_l(Y_{j_{\alpha_{[l]\mo 4m}}}(s_l)-z_{[l]_{\mo 4m}})}(-iu_l)ds_{8m}\dots ds_{4m+1}\bigg)du\bigg]\\
 &=\bar E\bigg[(2\pi)^{-8m}\int_{\R^{8m}}\prod_{l=1}^{4m}e^{i(u_l+u_{l+4m})z_l}\Big((-iu_l)(-iu_{l+4m})\Big)\\
 &\quad\cdot\bigg(\int_{\Delta^{4m}_{\theta,t}}\Hfct_{j_1,\dots,j_m}^{\otimes4}(\tau,\sigma_1,\sigma_2;s_{1},\dots,s_{4m},\theta)\prod_{l=1}^{4m}e^{-iu_lY_{j_{\alpha_{[l]\mo 4m}}}(s_l)}ds_{4m}\dots ds_1\bigg)\\
 &\quad\cdot\bigg(\int_{\Delta^{4m}_{\theta,t}}\Hfct_{j_1,\dots,j_m}^{\otimes4}(\tau,\sigma_1,\sigma_2;s_{4m+1},\dots,s_{8m},\theta)\prod_{l=4m+1}^{8m}e^{-iu_lY_{j_{\alpha_{[l]\mo 4m}}}(s_l)}ds_{8m}\dots ds_{4m+1}\bigg)du\bigg].
\end{align*}

Now observe that, since $(-i)(-i)=-1$, we have that $\Big((-iu_l)(-iu_{l+4m})\Big)$ is of the form $(-1)u_lu_{l+4m}$ and thus
\begin{align*}
 \prod_{l=1}^{4m}e^{i(u_l+u_{l+4m})z_l}\Big((-iu_l)(-iu_{l+4m})\Big)&=(-1)^{4m}\prod_{l=1}^{4m}e^{i(u_l+u_{l+4m})z_l}\Big(u_lu_{l+4m}\Big)\\
 &=\prod_{l=1}^{4m}e^{i(u_l+u_{l+4m})z_l}u_lu_{l+4m}.
\end{align*}

Plugging this in and applying Lemma \ref{3IntegralProductShuffle}, we get

\begin{align*}
 &\bar E\Big[\Big|(2\pi)^{-4m}\int_{\R^{4m}}\int_{\Delta^{4m}_{\theta,t}}\Hfct_{j_1,\dots,j_m}^{\otimes4}(\tau,\sigma_1,\sigma_2;s_{1},\dots,s_{4m},\theta)\prod_{l=1}^{4m}e^{-iu_l(Y_{j_{\alpha_l}}(s_l)-z_l)}(-iu_l)dsdu\Big|^2\Big]\\
 &=\sum_{\rho\in S(4m,4m)}\bar E\Big[(2\pi)^{-8m}\int_{\R^{8m}}\Big(\prod_{l=1}^{4m}e^{i(u_l+u_{l+4m})z_l}u_lu_{l+4m}\\
 &\quad\hspace{-2mm}\cdot\int_{\Delta^{8m}_{\theta,t}}\Big\{\Hfct_{j_1,\dots,j_m}^{\otimes4}(\tau,\sigma_1,\sigma_2;s_{\rho(1)},\dots,s_{\rho(4m)},\theta)\Hfct_{j_1,\dots,j_m}^{\otimes4}(\tau,\sigma_1,\sigma_2;s_{\rho(4m+1)},\dots,s_{\rho(8m)},\theta)\\
 &\quad\hspace{-2mm}\cdot\prod_{l=1}^{8m}e^{-iu_lY_{j_{\alpha_{[l]\mo 4m}}}(s_{\rho(l)})}\Big\}ds_{8m}\dots ds_1\Big)du\Big]\\
 &=\sum_{\rho\in S(4m,4m)}\bar E\Big[(2\pi)^{-8m}\int_{\R^{8m}}\Big(e^{i\sum_{l=1}^{4m}(u_l+u_{l+4m})z_l}\prod_{l=1}^{8m}u_l\\
 &\quad\hspace{-2mm}\cdot\int_{\Delta^{8m}_{\theta,t}}\Big\{\Hfct_{j_1,\dots,j_m}^{\otimes8}(\rho,\tau,\sigma_1,\sigma_2;s_{1},\dots,s_{8m},\theta)e^{-i\sum_{l=1}^{8m}u_{\rho^{-1}(l)}Y_{j_{\alpha_{[\rho^{-1}(l)]_{\mo 4m}}}}(s_l)}\Big\}ds_{8m}\dots ds_1\Big)du\Big]\\
 &=\sum_{\rho\in S(4m,4m)}(2\pi)^{-8m}\int_{\R^{8m}}\Big(e^{i\sum_{l=1}^{4m}(u_l+u_{l+4m})z_l}\prod_{l=1}^{8m}u_l\\
 &\quad\hspace{-2mm}\cdot\int_{\Delta^{8m}_{\theta,t}}\Big\{\Hfct_{j_1,\dots,j_m}^{\otimes8}(\rho,\tau,\sigma_1,\sigma_2;s_{1},\dots,s_{8m},\theta)\bar E\Big[e^{-i\sum_{l=1}^{8m}u_{\rho^{-1}(l)}Y_{j_{\alpha^\rho_l}}(s_l)}\Big]\Big\}ds_{8m}\dots ds_1\Big)du,
\end{align*}

where $\Hfct_{j_1,\dots,j_m}^{\otimes8}(\dots)$ is defined the same way as $\Hfct_{j_1,\dots,j_m}^{\otimes4}(\dots)$ and $\Hfct_{j_1,\dots,j_m}^{\otimes2}(\dots)$ before and $\alpha^\rho_l:=\alpha_{[\rho^{-1}(l)]_{\mo 4m}}$. Now denote for any $\rho\in S(4m,4m)$, by $P_\rho\in\R^{8m\times8m}$ the permutation matrix such that, for $u\in\R^{8m}$, $P_\rho u=u_\rho$, where $u_\rho=(u_{\rho(1)},\dots,u_{\rho(8m)})$. Then
$$u_{\rho^{-1}}=P_{\rho^{-1}} u=P_\rho^{-1}u.$$
For every fixed $\rho\in S(4m,4m)$ we can define $\tilde u:=P_\rho^{-1}u$. Note that $\prod_{l=1}^{8m}u_l=\prod_{l=1}^{8m}\tilde u_l$. Since $|\det P_{\rho}^{-1}|=1$, we have by the transformation formula

\begin{align*}
 &\bar E\bigg[\bigg|(2\pi)^{-4m}\int_{\R^{4m}}\int_{\Delta^{4m}_{\theta,t}}\Hfct_{j_1,\dots,j_m}^{\otimes4}(\tau,\sigma_1,\sigma_2;s_{1},\dots,s_{4m},\theta)\prod_{l=1}^{4m}e^{-iu_l(Y_{j_{\alpha_l}}(s_l)-z_l)}(-iu_l)dsdu\bigg|^2\bigg]\\
 &=\sum_{\rho\in S(4m,4m)}(2\pi)^{-8m}\int_{\R^{8m}}\Big(e^{i\sum_{l=1}^{4m}((P_\rho\tilde u)_l+(P_\rho\tilde u)_{l+4m})z_l}\prod_{l=1}^{8m}\tilde u_l\\
 &\quad\cdot\int_{\Delta^{8m}_{\theta,t}}\Big\{\Hfct_{j_1,\dots,j_m}^{\otimes8}(\rho,\tau,\sigma_1,\sigma_2;s_{1},\dots,s_{8m},\theta)\bar E\Big[e^{-i\sum_{l=1}^{8m}\tilde u_lY_{j_{\alpha^\rho_l}}(s_l)}\Big]\Big\}ds_{8m}\dots ds_1\Big)d\tilde u.
\end{align*}

Note that $\sum_{l=1}^{8m}\tilde u_lY_{j_{\alpha^\rho_l}}(s_l)$ is a Gaussian random variable. To see this, rewrite

\begin{align}
 -\sum_{l=1}^{8m}\tilde u_lY_{j_{\alpha^\rho_l}}(s_l)&=\sum_{k\in\{j_1,\dots,j_m\}}\sum_{l: j_{\alpha^\rho_l}=k}(-\tilde u_l)Y_k(s_l)\label{3ReGrouping}\\
 &=\sum_{k\in\{j_1,\dots,j_m\}}\sum_{l: j_{\alpha^\rho_l}=k}(-\tilde u_l)(F_k(s_l,W_{s_l})+\eps w_kB^{H_k}(s_l))\nonumber\\
 &=\sum_{k\in\{j_1,\dots,j_m\}}\sum_{l: j_{\alpha^\rho_l}=k}(-\tilde u_l)F_k(s_l,W_{s_l})+\sum_{k\in\{j_1,\dots,j_m\}}G_k(s),\nonumber
\end{align}

where $G_k(s):=\sum_{l: j_{\alpha^\rho_l}=k}(-\tilde u_l)\eps w_kB^{H_k}(s_l)$, $s\in\Delta^{8m}_{\theta,t}$. Note that the first summand is Gaussian as a sum bounded linear functionals of the Gaussian process $W$ and that $(G_k(s))_{k}$ are independent Gaussian random variables, independent of the process $W$. We therefore have that

\begin{align*}
 \bar E\Big[e^{-i\sum_{l=1}^{8m}\tilde u_lY_{j_{\alpha^\rho_l}}(s_l)}\Big]&=\bar E\Big[e^{i(-\sum_{l=1}^{8m}\tilde u_lY_{j_{\alpha^\rho_l}}(s_l))}\Big]\\
 &=\exp\Big\{i\bar E\Big[-\sum_{l=1}^{8m}\tilde u_lY_{j_{\alpha^\rho_l}}(s_l)\Big]-\frac{1}{2}\Var_{\bar\PP}\Big(-\sum_{l=1}^{8m}\tilde u_lY_{j_{\alpha^\rho_l}}(s_l)\Big)\Big\}\\
 &=e^{i\mu^\rho(\tilde u,s)}\exp\Big\{-\frac{1}{2}\Var_{\bar\PP}\Big(-\sum_{l=1}^{8m}\tilde u_lY_{j_{\alpha^\rho_l}}(s_l)\Big)\Big\}
\end{align*}

Plugging in and using that $\int\dots\leq\int|\dots|$ as well as the fact that $|e^{i\varphi}|=1$ and that \newline$|\Hfct_{j_1,\dots,j_m}^{\otimes8}(\rho,\tau,\sigma_1,\sigma_2;s_{1},\dots,s_{8m},\theta)|\leq(1+r)^{8m}$, we have

\begin{align*}
 &\bar E\bigg[\bigg|(2\pi)^{-4m}\int_{\R^{4m}}\int_{\Delta^{4m}_{\theta,t}}\Hfct_{j_1,\dots,j_m}^{\otimes4}(\tau,\sigma_1,\sigma_2;s_{1},\dots,s_{4m},\theta)\prod_{l=1}^{4m}e^{-iu_l(Y_{j_{\alpha_l}}(s_l)-z_l)}(-iu_l)dsdu\bigg|^2\bigg]\\
 &\leq\sum_{\rho\in S(4m,4m)}(2\pi)^{-8m}\int_{\R^{8m}}\bigg(|e^{i\sum_{l=1}^{4m}((P_\rho\tilde u)_l+(P_\rho\tilde u)_{l+4m})z_l}|\prod_{l=1}^{8m}|\tilde u_l|\\
 &\quad\cdot\int_{\Delta^{8m}_{\theta,t}}\bigg\{|\Hfct_{j_1,\dots,j_m}^{\otimes8}(\rho,\tau,\sigma_1,\sigma_2;s_{1},\dots,s_{8m},\theta)||e^{i\mu^\rho(\tilde u,s)}|\\
 &\quad\cdot\Big|\exp\Big\{-\frac{1}{2}\Var_{\bar\PP}\Big(-\sum_{l=1}^{8m}\tilde u_lY_{j_{\alpha^\rho_l}}(s_l)\Big)\Big\}\Big|\bigg\}ds_{8m}\dots ds_1\bigg)d\tilde u\\
 &\leq\sum_{\rho\in S(4m,4m)}\left(\frac{1+r}{2\pi}\right)^{8m}\int_{\R^{8m}}\bigg(\prod_{l=1}^{8m}|\tilde u_l|\\
 &\quad\cdot\int_{\Delta^{8m}_{\theta,t}}\exp\Big\{-\frac{1}{2}\Var_{\bar\PP}\Big(-\sum_{l=1}^{8m}\tilde u_lY_{j_{\alpha^\rho_l}}(s_l)\Big)\Big\}ds_{8m}\dots ds_1\bigg)d\tilde u.
\end{align*}

Moreover, by the independence of $W$ and $(G_k)_k$ it holds

\begin{align*}
 \Var_{\bar\PP}\Big(-\sum_{l=1}^{8m}\tilde u_lY_{j_{\alpha^\rho_l}}(s_l)\Big)&=\Var_{\bar\PP}\Big(\sum_{k\in\{j_1,\dots,j_m\}}\sum_{l: j_{\alpha^\rho_l}=k}(-\tilde u_l)F_k(s_l,W_{s_l})+\sum_{k\in\{j_1,\dots,j_m\}}G_k(s)\Big)\\
 &=\Var_{\bar\PP}\Big(\sum_{k\in\{j_1,\dots,j_m\}}\sum_{l: j_{\alpha^\rho_l}=k}(-\tilde u_l)F_k(s_l,W_{s_l})\Big)\\
 &\quad+\sum_{k\in\{j_1,\dots,j_m\}}\Var_{\bar\PP}(G_k(s))\\
 &\geq\sum_{k\in\{j_1,\dots,j_m\}}\Var_{\bar\PP}(G_k(s))\\
 &=\sum_{k\in\{j_1,\dots,j_m\}}\Var_{\bar\PP}\Big(\sum_{l: j_{\alpha^\rho_l}=k}(-\tilde u_l)\eps w_kB^{H_k}(s_l)\Big).
\end{align*}

Now we define for each $k\in\{j_1,\dots,j_m\}$ and each $\rho\in S(4m,4m)$,

\begin{align}\label{3a_k}
 a_k:=\card\{l:\,j_{\alpha^\rho_l}=k\}=8\card\{l:\,j_l=k\},
\end{align}

(note that this implies that $\sum_{k\in\{j_1,\dots,j_m\}}a_k=8m$) and perform the transformations

\begin{align}\label{3ReGrouping2}
\begin{split}
 \R^{8m}\ni\tilde u&=\begin{pmatrix}
                      \tilde u_1\\
                      \vdots\\
                      \tilde u_{8m}
                     \end{pmatrix}\mapsto\tilde u^k=\begin{pmatrix}
                      \tilde u^k_1\\
                      \vdots\\
                      \tilde u^k_{a_k}
                     \end{pmatrix}\in\R^{a_k}\\
 \Delta^{8m}_{\theta,t}\ni s&=\begin{pmatrix}
                      s_1\\
                      \vdots\\
                      s_{8m}
                     \end{pmatrix}\mapsto s^k=\begin{pmatrix}
                      s^k_1\\
                      \vdots\\
                      s^k_{a_k}
                     \end{pmatrix}\in\Delta^{a_k}_{\theta,t}
\end{split}
\end{align}

where $\tilde u^k$ consists of those entries $\tilde u_l$ of $\tilde u$ such that $j_{\alpha^\rho_l}=k$ and $s^k$ is defined in the same manner, in particular, the entries of $s^k$ follow the same order as those of $\tilde u^k$. Note that the $\R^{8m}$ vector that one gets from putting all $\tilde u^k$, $k\in\{j_1,\dots,j_m\}$, together contains all the elements in $\tilde u$ just rearranged (i.e. permutated). We thus have

\begin{align*}
 \Var_{\bar\PP}\Big(-\sum_{l=1}^{8m}\tilde u_lY_{j_{\alpha^\rho_l}}(s_l)\Big)&\geq\sum_{k\in\{j_1,\dots,j_m\}}\Var_{\bar\PP}\Big(\sum_{l=1}^{a_k}(-\tilde u^k_l)\eps w_kB^{H_k}(s^k_l)\Big).
\end{align*}

Plugging this in and recalling that $\prod_{l=1}^{8m}|\tilde u_l|=\prod_{k\in\{j_1,\dots,j_m\}}\prod_{l=1}^{a_k}|\tilde u^k_l|$ and that
\begin{align*}
 \Delta^{8m}_{\theta,t}\subseteq\prod_{k\in\{j_1,\dots,j_m\}}\Delta^{a_k}_{\theta,t},
\end{align*}

we get the estimate

\begin{align*}
 &\bar E\bigg[\bigg|(2\pi)^{-4m}\int_{\R^{4m}}\int_{\Delta^{4m}_{\theta,t}}\Hfct_{j_1,\dots,j_m}^{\otimes4}(\tau,\sigma_1,\sigma_2;s_{1},\dots,s_{4m},\theta)\prod_{l=1}^{4m}e^{-iu_l(Y_{j_{\alpha_l}}(s_l)-z_l)}(-iu_l)dsdu\bigg|^2\bigg]\\
 &\leq\sum_{\rho\in S(4m,4m)}\left(\frac{1+r}{2\pi}\right)^{8m}\int_{\R^{8m}}\bigg(\Big(\prod_{k\in\{j_1,\dots,j_m\}}\prod_{l=1}^{a_k}|\tilde u^k_l|\Big)\\
 &\quad\cdot\int_{\prod_{k\in\{j_1,\dots,j_m\}}\Delta^{a_k}_{\theta,t}}\prod_{k\in\{j_1,\dots,j_m\}}\exp\Big\{-\frac{1}{2}\Var_{\bar\PP}\Big(\sum_{l=1}^{a_k}(-\tilde u^k_l)\eps w_kB^{H_k}(s^k_l)\Big)\Big\}d(s^k)_{k}\bigg)d(\tilde u^k)_k\\
 &=\sum_{\rho\in S(4m,4m)}\left(\frac{1+r}{2\pi}\right)^{8m}\prod_{k\in\{j_1,\dots,j_m\}}\int_{\R^{a_k}}\bigg(\prod_{l=1}^{a_k}|\tilde u^k_l|\\
 &\quad\cdot\int_{\Delta^{a_k}_{\theta,t}}\exp\Big\{-\frac{1}{2}\Var_{\bar\PP}\Big(\sum_{l=1}^{a_k}(-\tilde u^k_l)\eps w_kB^{H_k}(s^k_l)\Big)\Big\}ds^k\bigg)d\tilde u^k,
\end{align*}

where $d(s^k)_k:=\prod_{k\in\{j_1,\dots,j_m\}}ds^k$ and $d(\tilde u^k)_k$ is defined equivalently. For each $k\in\{j_1,\dots,j_m\}$, we define now the transformation $M^k:\R^{a_k}\rightarrow\R^{a_k}$ by

\begin{align*}
 \R^{a_k\times a_k}\ni M^k:=\begin{pmatrix}
                          1    &  0     &   0    & \cdots &   0    &   0    &   0   \\
                          -1   &  1   &   0    & \cdots &   0    &   0    &   0   \\
                          0      &  -1  &  1   & \cdots &   0    &   0    &   0   \\
                          \vdots & \vdots & \vdots & \ddots & \vdots & \vdots & \vdots\\
                          0      &  0     &   0    & \cdots &   1  &   0    &   0   \\
                          0      &  0     &   0    & \cdots &  -1  &  1   &   0   \\
                          0      &  0     &   0    & \cdots &    0   & -1   &  1
                         \end{pmatrix}
\end{align*}

and define the vector $\xi^k\in\R^{a_k}$ by

\begin{align*}
 \tilde u^k=:M^k\xi^k.
\end{align*}

Since $|\det M^k|=1$, we have by the transformation formula

\begin{align*}
 &\bar E\bigg[\bigg|(2\pi)^{-4m}\int_{\R^{4m}}\int_{\Delta^{4m}_{\theta,t}}\Hfct_{j_1,\dots,j_m}^{\otimes4}(\tau,\sigma_1,\sigma_2;s_{1},\dots,s_{4m},\theta)\prod_{l=1}^{4m}e^{-iu_l(Y_{j_{\alpha_l}}(s_l)-z_l)}(-iu_l)dsdu\bigg|^2\bigg]\\
 &\leq\sum_{\rho\in S(4m,4m)}\left(\frac{1+r}{2\pi}\right)^{8m}\prod_{k\in\{j_1,\dots,j_m\}}\int_{\R^{a_k}}\bigg(\prod_{l=1}^{a_k}|(M^k\xi^k)_l|\\
 &\quad\cdot\int_{\Delta^{a_k}_{\theta,t}}\exp\Big\{-\frac{1}{2}\Var_{\bar\PP}\Big(\sum_{l=1}^{a_k}(-(M^k\xi^k)_l)\eps w_kB^{H_k}(s^k_l)\Big)\Big\}ds^k\bigg)d\xi^k.
\end{align*}

Applying the strong local non-determinism (see Remark \ref{3RemarkStrongLocalNondeterminism}) of the fractional Brownian motion yields
\begin{align*}
 &\Var_{\bar\PP}\Big(\sum_{l=1}^{a_k}(-(M^k\xi^k)_l)\eps w_kB^{H_k}(s^k_l)\Big)\\
 &\quad=\Var_{\bar\PP}\Big(\xi^k_1\eps w_kB^{H_k}(s^k_1)+\sum_{l=2}^{a_k}(\xi^k_l-\xi^k_{l-1})\eps w_kB^{H_k}(s^k_l)\Big)\\
 &\quad=\Var_{\bar\PP}\Big(\xi^k_{a_k}\eps w_kB^{H_k}(s^k_{a_k})+\sum_{l=1}^{a_k-1}\xi^k_l\eps w_k(B^{H_k}(s^k_l)-B^{H_k}(s^k_{l+1}))\Big)\\
 &\quad\geq C_k\sum_{l=1}^{a_k}\big|\xi^k_l\big|^2\eps ^2|w_k|^2|s^k_l-s^k_{l+1}|^{2H_k}=\sum_{l=1}^{a_k}\frac{\big|\xi^k_l\big|^2}{\sigma_{l,k}^2(s^k)},
\end{align*}

for a constant $C_k=C(H_k)\in(0,1)$, where we defined $s_{a_k+1}:=\theta$ and

\begin{align}\label{3sigmalkDef}
 \sigma_{l,k}(s^k):=C_k^{-\frac{1}{2}}\eps ^{-1}|w_k|^{-1}|s^k_l-s^k_{l+1}|^{-H_k}.
\end{align}

Moreover, we have

\begin{align*}
 \prod_{l=1}^{a_k}|(M^k\xi^k)_l|&=|\xi^k_1|\prod_{l=2}^{a_k}|\xi^k_l-\xi^k_{l-1}|<(1+|\xi^k_1|)\prod_{l=2}^{a_k}(1+|\xi^k_l|)(1+|\xi^k_{l-1}|)\\
 &=\Big(\prod_{l=1}^{a_k}(1+|\xi^k_l|)\Big)\Big(\prod_{l=1}^{a_k-1}(1+|\xi^k_l|)\Big)\leq\prod_{l=1}^{a_k}(1+|\xi^k_l|)^2\\
 &<2^{a_k}\prod_{l=1}^{a_k}(1+|\xi^k_l|+|\xi^k_l|^2)=2^{a_k}\prod_{l=1}^{a_k}\sum_{\delta_l\in\{0,1,2\}}|\xi^k_l|^{\delta_l}\\
 &=2^{a_k}\sum_{\delta\in\{0,1,2\}^{a_k}}\prod_{l=1}^{a_k}|\xi^k_l|^{\delta_l}.
\end{align*}

Plugging this in and using that $\sum_{k\in\{j_1,\dots,j_m\}}a_k=8m$, we achieve the estimate

\begin{align*}
 &\bar E\bigg[\bigg|(2\pi)^{-4m}\int_{\R^{4m}}\int_{\Delta^{4m}_{\theta,t}}\Hfct_{j_1,\dots,j_m}^{\otimes4}(\tau,\sigma_1,\sigma_2;s_{1},\dots,s_{4m},\theta)\prod_{l=1}^{4m}e^{-iu_l(Y_{j_{\alpha_l}}(s_l)-z_l)}(-iu_l)dsdu\bigg|^2\bigg]\\
 &\leq\sum_{\rho\in S(4m,4m)}\prod_{k\in\{j_1,\dots,j_m\}}\left(\frac{1+r}{\pi}\right)^{a_k}\int_{\R^{a_k}}\bigg(\sum_{\delta\in\{0,1,2\}^{a_k}}\prod_{l=1}^{a_k}|\xi^k_l|^{\delta_l}\\
 &\quad\cdot\int_{\Delta^{a_k}_{\theta,t}}\exp\Big\{-\frac{1}{2}\sum_{l=1}^{a_k}\frac{\big|\xi^k_l\big|^2}{\sigma_{l,k}^2(s^k)}\Big\}ds^k\bigg)d\xi^k\\
 &=\sum_{\rho\in S(4m,4m)}\prod_{k\in\{j_1,\dots,j_m\}}\left(\frac{1+r}{\pi}\right)^{a_k}\sum_{\delta\in\{0,1,2\}^{a_k}}\int_{\Delta^{a_k}_{\theta,t}}\int_{\R^{a_k}}\prod_{l=1}^{a_k}|\xi^k_l|^{\delta_l}e^{-\frac{1}{2}\frac{\big|\xi^k_l\big|^2}{\sigma_{l,k}^2(s^k)}}d\xi^kds^k\\
 &=\sum_{\rho\in S(4m,4m)}\prod_{k\in\{j_1,\dots,j_m\}}\left(\frac{1+r}{\pi}\right)^{a_k}\sum_{\delta\in\{0,1,2\}^{a_k}}\int_{\Delta^{a_k}_{\theta,t}}\prod_{l=1}^{a_k}\int_{\R}|\xi^k_l|^{\delta_l}e^{-\frac{1}{2}\frac{\big|\xi^k_l\big|^2}{\sigma_{l,k}^2(s^k)}}d\xi^k_lds^k.
\end{align*}

Consider the inner integral

\begin{align*}
 \int_{\R}|\xi^k_l|^{\delta_l}e^{-\frac{1}{2}\frac{\big|\xi^k_l\big|^2}{\sigma_{l,k}^2(s^k)}}d\xi^k_l&=
 \begin{cases}
  \int_\R e^{-\frac{x^2}{2\sigma_{l,k}^2(s^k)}}dx=\sqrt{2\pi\sigma_{l,k}^2(s^k)},\quad \delta_{l}=0\\
  \int_\R |x|e^{-\frac{x^2}{2\sigma_{l,k}^2(s^k)}}dx=2\sigma_{l,k}^2(s),\quad \delta_{l}=1\\
  \int_\R |x|^2e^{-\frac{x^2}{2\sigma_{l,k}^2(s^k)}}dx=\sqrt{2\pi}\sigma_{l,k}^3(s),\quad \delta_{l}=2
 \end{cases}\\
 &\leq\sqrt{2\pi}(\sigma_{l,k}(s^k))^{1+\delta_{l}}.
\end{align*}

By \eqref{3sigmalkDef}, using that $1+\delta_{l}\in\{1,2,3\}$ and that $|s^k_l-s^k_{l+1}|\leq T(\eps)<\eps^{\frac{3}{\delta_H}} <1$, $C_k<1$ and $w_k\leq1$, we get

\begin{align*}
 \int_{\R}|\xi^k_l|^{\delta_l}e^{-\frac{1}{2}\frac{\big|\xi^k_l\big|^2}{\sigma_{l,k}^2(s^k)}}d\xi^k_l
 &\leq\sqrt{2\pi}(C_k^{-\frac{1}{2}}\eps ^{-1}|w_k|^{-1}|s^k_l-s^k_{l+1}|^{-H_k})^{1+\delta_{l}}\\
 &\leq\sqrt{2\pi} C_k^{-\frac{3}{2}}\eps ^{-3}|w_k|^{-3}|s^k_l-s^k_{l+1}|^{-3H_k}%\\
 %&\leq\sqrt{2\pi} C_k^{-\frac{3}{2}}\eps ^{-3}|w_k|^{-3}\sum_{\beta_l\in\{1,3\}}|s^k_l-s^k_{l+1}|^{-\beta_lH_k}.
\end{align*}

Note that the right-hand side of this inequality does not depend on $\delta$ any longer. Inserting this inequality and making use of the fact that $\card\{0,1,2\}^{a_k}=3^{a_k}$, the estimate becomes

\begin{align}
 &\bar E\bigg[\bigg|(2\pi)^{-4m}\int_{\R^{4m}}\int_{\Delta^{4m}_{\theta,t}}\Hfct_{j_1,\dots,j_m}^{\otimes4}(\tau,\sigma_1,\sigma_2;s_{1},\dots,s_{4m},\theta)\prod_{l=1}^{4m}e^{-iu_l(Y_{j_{\alpha_l}}(s_l)-z_l)}(-iu_l)dsdu\bigg|^2\bigg]\nonumber\\
 &\leq\sum_{\rho\in S(4m,4m)}\bigg\{\prod_{k\in\{j_1,\dots,j_m\}}\left(\frac{\sqrt{2}(1+r)}{\sqrt{\pi}}C_k^{-\frac{3}{2}}\eps ^{-3}|w_k|^{-3}\right)^{a_k}\nonumber\\
 &\quad\cdot\sum_{\delta\in\{0,1,2\}^{a_k}}\int_{\Delta^{a_k}_{\theta,t}}\prod_{l=1}^{a_k}|s^k_l-s^k_{l+1}|^{-3H_k}ds^k\bigg\}\label{3IntermediateEstimate3}\\
 &=\sum_{\rho\in S(4m,4m)}\bigg\{\prod_{k\in\{j_1,\dots,j_m\}}\left(\frac{3\sqrt{2}(1+r)}{\sqrt{\pi}}C_k^{-\frac{3}{2}}\eps ^{-3}|w_k|^{-3}\right)^{a_k}\int_{\Delta^{a_k}_{\theta,t}}\prod_{l=1}^{a_k}|s^k_l-s^k_{l+1}|^{-3H_k}ds^k\bigg\}\nonumber.
\end{align}

Now we study the term $\int_{\Delta^{a_k}_{\theta,t}}\prod_{l=1}^{a_k}|s^k_l-s^k_{l+1}|^{-3H_k}ds^k$. Observe that, by Assumption \ref{3StandingAssumptions} $(H)$, it holds $-3H_k>-1$ for all $l=1,\dots,a_k$. We can therefore apply Lemma \ref{3SimplexIntegralLemma} and get that

\begin{align*}
 \int_{\Delta^{a_k}_{\theta,t}}\prod_{l=1}^{a_k}|s^k_l-s^k_{l+1}|^{-3H_k}ds^k= \frac{\prod_{l=1}^{a_k}\Gamma(1-3H_k)}{\Gamma(1+a_k-3a_kH_k)}(t-\theta)^{a_k-3a_kH_k}.
\end{align*}

Now since, by Assumption \ref{3StandingAssumptions} $(H)$, $H_k<\frac{1-\delta_H}{3}$ for some $\delta_H\in(0,1)$, and since the Gamma function is monotone decreasing on the interval $[0,1]$ we have

\begin{align*}
 \Gamma(1-3H_k)&\leq\Gamma(\delta_H).
\end{align*}

Moreover, since $\Gamma(x)>\frac{1}{2}$ for all $x>0$ and since $1+a_k(1-3H_k)\geq1+a_k\delta_H>1$, we have% $\Gamma(a_k-\sum_{l=1}^{a_k}\beta_lH_k)>\frac{1}{2}$

\begin{align*}
 \frac{\prod_{l=1}^{a_k}\Gamma(1-3H_k)}{\Gamma(1+a_k(1-3H_k))}\leq 2\Gamma(\delta_H)^{a_k}\leq2^{a_k}\Gamma(\delta_H)^{a_k}\quad\text{for }k\in\{j_1,\dots,j_m\}.
\end{align*}

%\textcolor{red}{\textbf{Maybe here we can find some smarter estimation for the denominator than just estimating it against $2^{-a_k}$...}}

Note further that $(t-\theta)^{a_k(1-3H_k)}\leq(t-\theta)^{\delta_Ha_k}$. Plugging in, we achieve the estimate

% \begin{align*}
% \begin{cases}
%  (t-\theta)^{a_k-\left(\sum_{l=1}^{a_k}\beta_l\right)H_k}&\leq(t-\theta)^{a_k(1-H_k)}\leq(t-\theta)^{a_k}\quad\text{if }t-\theta\geq1\\
%  (t-\theta)^{a_k-\left(\sum_{l=1}^{a_k}\beta_l\right)H_k}&\leq(t-\theta)^{a_k(1-3H_k)}\leq(t-\theta)^{\delta_Ha_k}\quad\text{if }t-\theta<1.
% \end{cases}
% \end{align*}

%Plugging in, we achieve the estimate

\begin{align*}
 \int_{\Delta^{a_k}_{\theta,t}}\prod_{l=1}^{a_k}|s^k_l-s^k_{l+1}|^{-3H_k}ds^k\leq 2^{a_k}\Gamma(\delta_H)^{a_k}(t-\theta)^{\delta_Ha_k},
\end{align*}

which we insert in \eqref{3IntermediateEstimate3} to get

\begin{align*}
 &\bar E\bigg[\bigg|(2\pi)^{-4m}\int_{\R^{4m}}\int_{\Delta^{4m}_{\theta,t}}\Hfct_{j_1,\dots,j_m}^{\otimes4}(\tau,\sigma_1,\sigma_2;s_{1},\dots,s_{4m},\theta)\prod_{l=1}^{4m}e^{-iu_l(Y_{j_{\alpha_l}}(s_l)-z_l)}(-iu_l)dsdu\bigg|^2\bigg]\\
 &\leq\sum_{\rho\in S(4m,4m)}\bigg\{\prod_{k\in\{j_1,\dots,j_m\}}\left(\frac{3\sqrt{2}(1+r)}{\sqrt{\pi}}C_k^{-\frac{3}{2}}\eps ^{-3}|w_k|^{-3}\right)^{a_k}2^{a_k}\Gamma(\delta_H)^{a_k}(t-\theta)^{\delta_Ha_k}\bigg\}\\
 &\leq\sum_{\rho\in S(4m,4m)}\prod_{k\in\{j_1,\dots,j_m\}}\left(\frac{6\sqrt{2}(1+r)\Gamma(\delta_H)(t-\theta)^{\delta_H}}{\sqrt{\pi}}C_k^{-\frac{3}{2}}\eps ^{-3}|w_k|^{-3}\right)^{a_k}.
\end{align*}

Now recall that, by \eqref{3a_k},

\begin{align*}
 a_k=\card\{l:\,j_{\alpha^\rho_l}=k\}=8\card\{l:\,j_{l}=k\},
\end{align*}

which means that we can rewrite

\begin{align*}
 &\prod_{k\in\{j_1,\dots,j_m\}}\left(\frac{6\sqrt{2}(1+r)\Gamma(\delta_H)(t-\theta)^{\delta_H}}{\sqrt{\pi}}C_k^{-\frac{3}{2}}\eps ^{-3}|w_k|^{-3}\right)^{a_k}\\
 &=\prod_{l=1}^m\left(\frac{6\sqrt{2}(1+r)\Gamma(\delta_H)(t-\theta)^{\delta_H}}{\sqrt{\pi}}C_{j_l}^{-\frac{3}{2}}\eps ^{-3}|w_{j_l}|^{-3}\right)^8.
\end{align*}

%which does not depend on the permutation $\rho$ any more.
Since $\card(S(4m,4m))={{8m}\choose{4m}}\leq 2^{8m}$ we therefore obtain

\begin{align*}
 &\bar E\bigg[\bigg|(2\pi)^{-4m}\int_{\R^{4m}}\int_{\Delta^{4m}_{\theta,t}}\Hfct_{j_1,\dots,j_m}^{\otimes4}(\tau,\sigma_1,\sigma_2;s_{1},\dots,s_{4m},\theta)\prod_{l=1}^{4m}e^{-iu_l(Y_{j_{\alpha_l}}(s_l)-z_l)}(-iu_l)dsdu\bigg|^2\bigg]\\
 &\leq\prod_{l=1}^m\left(\frac{12\sqrt{2}(1+r)\Gamma(\delta_H)(t-\theta)^{\delta_H}}{\sqrt{\pi}}C_{j_l}^{-\frac{3}{2}}\eps ^{-3}|w_{j_l}|^{-3}\right)^8.
\end{align*}

Note that while the term on the left handside depends on $z$, the term on the right handside does not. We are now ready to go back to \eqref{3IntermediateEstimate2}. Plugging our results in, we achieve the following estimate

\begin{align*}
 &\bar E\bigg[\Big|\int_{\Delta^m_{\theta,t}}\Hfct_{j_1,\dots,j_m}(s,\theta) \prod_{l=1}^mb_{j_l,n}'(Y_{j_l}(s_l))ds_m\dots ds_1\Big|^4\bigg]\\
 &\leq \sum_{\substack{\sigma_1,\sigma_2\in S(m,m) \\ \tau\in S(2m,2m)}}\int_{\R^{4m}}\prod_{l=1}^{4m}|b_{j_{\alpha_l},n}(z_l)|\bar E\bigg[\bigg|(2\pi)^{-4m}\int_{\R^{4m}}\int_{\Delta^{4m}_{\theta,t}}\Hfct_{j_1,\dots,j_m}^{\otimes4}(\tau,\sigma_1,\sigma_2;s_{1},\dots,s_{4m},\theta)\\
 &\quad\cdot\prod_{l=1}^{4m}e^{-iu_l(Y_{j_{\alpha_l}}(s_l)-z_l)}(-iu_l)dsdu\bigg|^2\bigg]^{\frac{1}{2}}dz\\
 &\leq \sum_{\substack{\sigma_1,\sigma_2\in S(m,m) \\ \tau\in S(2m,2m)}}\left(\prod_{l=1}^{4m}\|b_{j_{\alpha_l},n}\|_{L^1}\right)\cdot\left(\prod_{l=1}^m\frac{12\sqrt{2}(1+r)\Gamma(\delta_H)(t-\theta)^{\delta_H}}{\sqrt{\pi}}C_{j_l}^{-\frac{3}{2}}\eps ^{-3}|w_{j_l}|^{-3}\right)^4\\
 &\leq\prod_{l=1}^m\left(\frac{48\sqrt{2}(1+r)\Gamma(\delta_H)(t-\theta)^{\delta_H}}{\sqrt{\pi}}C_{j_l}^{-\frac{3}{2}}\eps ^{-3}|w_{j_l}|^{-3}\|b_{j_{l}}\|_{L^1}\right)^4\\
 &=\prod_{l=1}^m\eps^{-12}(t-\theta)^{4\delta_H}A_{j_l}^4.
\end{align*}

In the last inequality, we used \eqref{3alphalDef} and the fact that $\sigma_1$, $\sigma_2$ and $\tau$ are permutations which implies that

\begin{align*}
 \prod_{l=1}^{4m}\|b_{j_{\alpha_l},n}\|_{L^1}&=\prod_{l=1}^{m}\|b_{j_{l},n}\|_{L^1}^4=\prod_{l=1}^{m}\Big(\int_\R|b_{j_{l},n}(z)|dz\Big)^4\\
 &=\prod_{l=1}^{m}\Big(\int_\R|b_{j_{l}}\ast\varphi_n(z)|dz\Big)^4\\
 &=\prod_{l=1}^{m}\Big(\int_\R|\int_\R b_{j_{l}}(y)\varphi_n(z-y)dy|dz\Big)^4\\
 &\leq\prod_{l=1}^{m}\Big(\int_\R\int_\R |b_{j_{l}}(y)|\varphi_n(z-y)dydz\Big)^4\\
 &=\prod_{l=1}^{m}\Big(\int_\R|b_{j_{l}}(y)|\int_\R \varphi_n(z-y)dzdy\Big)^4\\
 &=\prod_{l=1}^{m}\Big(\int_\R|b_{j_{l}}(y)|dy\Big)^4=\prod_{l=1}^{m}\|b_{j_{l}}\|_{L^1}^4.
\end{align*}

Moreover, we applied the fact that $\card(S(2m,2m))={{4m}\choose{2m}}\leq 2^{4m}$ and $\card(S(m,m))={{2m}\choose{m}}\leq 2^{2m}$.\\
\par
With this estimation at hands, we can finally return to \eqref{3IntermediateEstimate1} and obtain

\begin{align}
 &E'\bigg[E\bigg[|\DM_{\theta} x^{n}(t)-1|^2\bigg]\bigg]\nonumber\\
 &\leq e^{\frac{1}{2}M^2T}\bigg(\sum_{m=1}^\infty\sum_{j_1,\dots,j_m=1}^d\bar E\bigg[\Big|\int_{\Delta^m_{\theta,t}}\Hfct_{j_1,\dots,j_m}(s,\theta) \prod_{l=1}^mb_{j_l,n}'(Y_{j_l}(s_l))ds\Big|^4\bigg]^{\frac{1}{4}}\bigg)^2\nonumber\\
 &\leq e^{\frac{1}{2}M^2T}\bigg(\sum_{m=1}^\infty\sum_{j_1,\dots,j_m=1}^d\prod_{l=1}^m\eps^{-3}(t-\theta)^{\delta_H}A_{j_l}\bigg)^2\nonumber\\
 &=e^{\frac{1}{2}M^2T}\bigg(\sum_{m=1}^\infty\prod_{l=1}^m\sum_{j_l=1}^d\eps^{-3}(t-\theta)^{\delta_H}A_{j_l}\bigg)^2\nonumber\\
 &=e^{\frac{1}{2}M^2T}\bigg(\sum_{m=1}^\infty\Big(\eps^{-3}(t-\theta)^{\delta_H}\sum_{j=1}^dA_{j}\Big)^m\bigg)^2\label{3UsefulEstimationForLater1}.%\\
 %&\leq e^{\frac{1}{2}M^2T}\bigg(\sum_{m=1}^\infty\Big(\sum_{j=1}^\infty A_{j}(t-\theta)\Big)^m\bigg)^2.\nonumber
\end{align}

The right handside is finite by Assumptions \ref{3StandingAssumptions} $(T)$ and $(A)$. In fact,

\begin{align*}
 E'\bigg[E\bigg[|\DM_{\theta} x^{n}(t)-1|^2\bigg]\bigg]&\leq e^{\frac{1}{2}M^2T}\bigg(\sum_{m=1}^\infty\Big((t-\theta)^{\delta_H}\eps ^{-3}\Big)^m\bigg)^2\\
 &=e^{\frac{1}{2}M^2T}\bigg(\frac{1}{1-(t-\theta)^{\delta_H}\eps ^{-3}}-1\bigg)^2\\
 &=e^{\frac{1}{2}M^2T}\bigg(\frac{(t-\theta)^{\delta_H}\eps ^{-3}}{1-(t-\theta)^{\delta_H}\eps ^{-3}}\bigg)^2\\
 &= e^{\frac{1}{2}M^2T}\frac{|t-\theta|^{2\delta_H}}{\left(\eps ^{3}-|t-\theta|^{\delta_H}\right)^2}.
\end{align*}

Since we have $|t-\theta|^{\delta_H}\leq T(\eps)^{\delta_H}<\eps^3-\delta_T^{\delta_H}$, the denominator is larger than $\delta_T^{2\delta_H}$. The above term therefore simplifies to

\begin{align}\label{3FinalEstimationStep1}
%\begin{split}
 E'\bigg[E\bigg[|\DM_{\theta} x^{n}(t)-1|^2\bigg]\bigg]\leq e^{\frac{1}{2}M^2T}\delta_T^{-2\delta_H}|t-\theta|^{2\delta_H}.
%\end{split}
\end{align}

Integrating and taking the supremum over $n$ (note that we eliminated all dependence on $n$ before) now finally yields

\begin{align*}
 \sup_{n\geq1}\int_0^tE'\bigg[E\bigg[|\DM_{\theta} x^{n}(t)-1|^2\bigg]\bigg]d\theta&\leq e^{\frac{1}{2}M^2T}\delta_T^{-2\delta_H}\frac{t^{2\delta_H+1}}{2\delta_H+1}<\infty,
\end{align*}

which proves \eqref{3CompactnessMalliavin3}.

\textbf{Proof of \eqref{3CompactnessMalliavin4}:} Note that, by Fubini's theorem, we have

\begin{align*}
   E'\bigg[ \int_0^t \int_0^t \frac{E[|\DM_\theta x^{n}(t) - \DM_{\theta'} x^{n}(t)|^2]}{|\theta' - \theta|^{1+2\beta}} d\theta' d\theta\bigg]= \int_0^t \int_0^t \frac{E'\Big[ E[|\DM_\theta x^{n}(t) - \DM_{\theta'} x^{n}(t)|^2]\Big]}{|\theta' - \theta|^{1+2\beta}} d\theta' d\theta.
  \end{align*}

We therefore consider for the moment the term $E'\Big[ E[|\DM_\theta x^{n}(t) - \DM_{\theta'} x^{n}(t)|^2]\Big]$. The representation \eqref{3MalliavinDiffIteration} together with the fact that $(a+b+c)^2\leq 3a^2+3b^2+3c^2$ yield

\begin{align*}
 &E'\Big[ E[|\DM_\theta x^{n}(t) - \DM_{\theta'} x^{n}(t)|^2]\Big]\\
 &\leq 3E'\Big[ E[|\DM_\theta x^{n}(\theta')-1|^2]\Big]\\
 &\quad +3E'\Big[ E\Big[\Big|\sum_{m=1}^\infty\sum_{j_1,\dots,j_m=1}^d\int_{\Delta^m_{\theta',t}}\tilde\Hfct_{j_1,\dots,j_m}(s,\theta,\theta')\\
 &\quad \cdot \prod_{l=1}^mb_{j_l,n}'(\<x^{n}_{s_l}, e_{j_l}\>+\eps w_{j_l}B^{H_{j_l}}(s_l))ds\Big|^2\Big]\Big]\\
 &\quad +3E'\Big[ E\Big[\Big|\sum_{m_1, m_2=1}^\infty\sum_{j_1,\dots,j_{m_1+m_2}=1}^d\int_{\Delta^{m_1}_{\theta',t}\times\Delta^{m_2}_{\theta,\theta'}}\Big\{\Hfct_{j_1,\dots,j_{m_1+m_2} }(s,\theta)\\
   &\quad \cdot \prod_{l=1}^{m_1+m_2}b_{j_l,n}'(\<x^{n}_{s_l}, e_1\>+\eps w_{j_l}B^{H_{j_l}}(s_l),\dots)\Big\}ds\Big|^2\Big]\Big]\\
 &=3I_1(\theta,\theta')+3I_2(\theta,\theta',t)+3I_3(\theta,\theta',t).
\end{align*}

It follows directly from \eqref{3FinalEstimationStep1} that

\begin{align}\label{3I1Estimate}
 I_1(\theta,\theta')\leq e^{\frac{1}{2}M^2T}\delta_T^{-2\delta_H}|\theta'-\theta|^{2\delta_H}.
\end{align}

For the term $I_2(\theta,\theta',t)$ recall \eqref{3IterativeRep}, which we state again for convenience:

\begin{align*}
 E'[E[|\DM_{\theta} x^{n}(t)-1|^2]]&=E'\bigg[E\Big[\Big|\sum_{m=1}^\infty\sum_{j_1,\dots,j_m=1}^d\int_{\Delta^m_{\theta,t}}\Big\{\Hfct_{j_1,\dots,j_m}(s,\theta)\\
 &\quad\cdot \prod_{l=1}^mb_{j_l,n}'(\<x^{n}_{s_l}, e_{j_l}\>+\eps w_{j_l}B^{H_{j_l}}(s_l))\Big\}ds\Big|^2\Big]\bigg]
\end{align*}

Note that the term $I_2(\theta,\theta',t)$ has the same form as the term on the right-hand side of this equation, except that the integral is now over $\Delta^m_{\theta',t}$ instead of $\Delta^m_{\theta,t}$, and the function $\Hfct_{j_1,\dots,j_m}(s,\theta)$ must be replaced by the function $\tilde\Hfct_{j_1,\dots,j_m}(s,\theta,\theta')$. We then perform the exact same steps as in the prove of \eqref{3CompactnessMalliavin3} until we reach estimation
\begin{align*}
 |\Hfct_{j_1,\dots,j_m}^{\otimes8}(\rho,\tau,\sigma_1,\sigma_2;s_{1},\dots,s_{8m},\theta)|\leq (1+r)^{8m},
\end{align*}

which in our case is replaced by

\begin{align*}
 |\tilde\Hfct_{j_1,\dots,j_m}^{\otimes8}(\rho,\tau,\sigma_1,\sigma_2;s_{1},\dots,s_{8m},\theta)|\leq (1+r)^{8(m-1)}|\theta-\theta'|^4=(1+r)^{8m}\frac{|\theta-\theta'|^4}{(1+r)^8},
\end{align*}

as to be seen easily from \eqref{3HtildefunctionEstimation}. Continuing with the exact same steps as in the proof of \eqref{3CompactnessMalliavin3}, just keeping in mind the additional term $\frac{|\theta-\theta'|^4}{(1+r)^8}$, we eventually get the estimate

\begin{align}\label{3I2Estimate}
 I_2(\theta,\theta',t)\leq e^{\frac{1}{2}M^2T}\delta_T^{-2\delta_H}|t-\theta'|^{2\delta_H}\frac{|\theta-\theta'|}{(1+r)^2}.
\end{align}

For the term $I_3(\theta,\theta',t)$, similarly to the case of $I_2(\theta,\theta',t)$, we compare with \eqref{3IterativeRep} and note that it also has the same structure as the right-hand side of this equation, just that the sum over $m$ becomes a double sum over $m_1$ and $m_2$ and the integral over $\Delta^m_{\theta,t}$ becomes an integral over $\Delta^{m_1}_{\theta',t}\times\Delta^{m_2}_{\theta,\theta'}$. We therefore again perform the exact same steps as in the proof of \eqref{3CompactnessMalliavin3} until we reach \eqref{3ReGrouping}, just that whenever we multiply two integrals over $\Delta^{m_1}_{\theta',t}\times\Delta^{m_2}_{\theta,\theta'}$ we apply Corollary \ref{3DeltatimesDeltaShuffle}. Note that this means that all the shuffle permutations $\sigma_1,\sigma_2\in S(m,m)$, $\tau\in S(2m,2m)$ and $\rho\in S(4m,4m)$ are replaced by permutations $(\sigma^1_1,\sigma^2_1),(\sigma^1_2,\sigma^2_2)\in S(m_1,m_1)\times S(m_2,m_2)$, $(\tau^1,\tau^2)\in S(2m_1,2m_1)\times S(2m_2,2m_2)$ and $(\rho^1,\rho^2)\in S(4m_1,4m_1)\times S(4m_2,4m_2)$. We have to adjust the definitions for the index permutaions $\alpha$ and $\alpha^\rho$ to the new type of permutations but this works in the canonical way. In equation \eqref{3ReGrouping}, having in mind Corollary \ref{3DoubleShuffleInverse}, we re-group the terms now a little bit differently, namely

\begin{align*}
 -\sum_{l=1}^{8m_1+8m_2}\tilde u_lY_{j_{\alpha^{(\rho^1,\rho^2)}_l}}(s_l)&=\sum_{k_1\in\{j_1,\dots,j_{m_1}\}}\sum_{\substack{l\in\{1,\dots,8m_1\}:\\ j_{\alpha^{(\rho^1,\rho^2)}_l}=k_1}}(-\tilde u_l)Y_{k_1}(s_l)\\
 &\quad+\sum_{k_2\in\{j_{m_1+1},\dots,j_{m_1+m_2}\}}\sum_{\substack{l\in\{8m_1+1,\dots,8m_1+8m_2\}:\\ j_{\alpha^{(\rho^1,\rho^2)}_l}=k_2}}(-\tilde u_l)Y_{k_2}(s_l)\\
 &=\sum_{k_1\in\{j_1,\dots,j_{m_1}\}}\sum_{\substack{l\in\{1,\dots,8m_1\}:\\ j_{\alpha^{(\rho^1,\rho^2)}_l}=k_1}}(-\tilde u_l)F_{k_1}(s_l,W_{s_l})\\
 &\quad+\sum_{k_2\in\{j_{m_1+1},\dots,j_{m_1+m_2}\}}\sum_{\substack{l\in\{8m_1+1,\dots,8m_1+8m_2\}:\\ j_{\alpha^{(\rho^1,\rho^2)}_l}=k_2}}(-\tilde u_l)F_{k_2}(s_l,W_{s_l})\\
 &\quad+\sum_{k_1\in\{j_1,\dots,j_{m_1}\}}G^1_{k_1}(s)+\sum_{k_2\in\{j_{m_1+1},\dots,j_{m_1+m_2}\}}G^2_{k_2}(s),
\end{align*}

where, for $s\in\Delta^{8m_1}_{\theta',t}\times\Delta^{8m_2}_{\theta,\theta'}$,

\begin{align*}
 G^1_k(s)&:=\sum_{\substack{l\in\{1,\dots,8m_1\}:\\ j_{\alpha^{(\rho^1,\rho^2)}_l}=k_1}}(-\tilde u_l)\eps w_kB^{H_{k_1}}(s_l)\\
 G^2_k(s)&:=\sum_{\substack{l\in\{8m_1+1,\dots,8m_1+8m_2\}:\\ j_{\alpha^{(\rho^1,\rho^2)}_l}=k_2}}(-\tilde u_l)\eps w_kB^{H_{k_2}}(s_l).
\end{align*}

Next we define for each $k_1\in\{j_1,\dots,j_{m_1}\}$, $k_2\in\{j_{m_1+1},\dots,j_{m_1+m_2}\}$,

\begin{align*}
 a^1_{k_1}&:=8\card\{l\in\{1,\dots,m_1\}:\,j_{l}=k_1\},\\
 a^2_{k_2}&:=8\card\{l\in\{m_1+1,\dots,m_1+m_2\}:\,j_{l}=k_2\}
\end{align*}

and adjust the transformations from \eqref{3ReGrouping2} to our new setup by defining

\begin{align*}
 \R^{8m_1+8m_2}\ni\tilde u&=\begin{pmatrix}
                      \tilde u_1\\
                      \vdots\\
                      \tilde u_{8m_1+8m_2}
                     \end{pmatrix}\mapsto\tilde u^{k_i,i}=\begin{pmatrix}
                      \tilde u^{k_i,i}_1\\
                      \vdots\\
                      \tilde u^{k_i,i}_{a^i_{k_i}}
                     \end{pmatrix}\in\R^{a^i_{k_i}},\quad i=1,2\\
 \Delta^{8m_1}_{\theta',t}\times\Delta^{8m_2}_{\theta,\theta'}\ni s&=\begin{pmatrix}
                      s_1\\
                      \vdots\\
                      s_{8m_1+8m_2}
                     \end{pmatrix}\mapsto s^{k_1,1}=\begin{pmatrix}
                      s^{k_1,1}_1\\
                      \vdots\\
                      s^{k_1,1}_{a^1_{k_1}}
                     \end{pmatrix}\in\Delta^{a^1_{k_1}}_{\theta',t}\\
 \Delta^{8m_1}_{\theta',t}\times\Delta^{8m_2}_{\theta,\theta'}\ni s&=\begin{pmatrix}
                      s_1\\
                      \vdots\\
                      s_{8m_1+8m_2}
                     \end{pmatrix}\mapsto s^{k_2,2}=\begin{pmatrix}
                      s^{k_2,2}_1\\
                      \vdots\\
                      s^{k_2,2}_{a^2_{k_2}}
                     \end{pmatrix}\in\Delta^{a^2_{k_2}}_{\theta,\theta'}                    
\end{align*}

where $\tilde u^{k_1,1}$ consists of those entries $\tilde u_l$ of $\tilde u$ such that $l\in\{1,\dots,8m_1\}$ and $j_{\alpha^{(\rho^1,\rho^2)}_l}=k_1$ whereas $\tilde u^{k_2,2}$ consists of those entries $\tilde u_l$ such that $l\in\{8m_1+1,\dots,8m_1+8m_2\}$ and $j_{\alpha^{(\rho^1,\rho^2)}_l}=k_1$ ($s^{k_1,1}$ and $s^{k_2,2}$ are defined accordingly). Then we get

\begin{align*}
 \Var_{\bar\PP}\Big(-\sum_{l=1}^{8m_1+8m_2}\tilde u_lY_{j_{\alpha^{(\rho^1,\rho^2)}_l}}(s_l)\Big)&\geq \sum_{k_1\in\{j_1,\dots,j_{m_1}\}}\Var_{\bar\PP}\Big(\sum_{l=1}^{a^1_{k_1}}(-\tilde u^{k_1,1}_l)\eps w_{k_1}B^{H_{k_1}}(s^{k_1,1}_l)\Big)\\
 &\quad+\sum_{k_2\in\{j_{m_1+1},\dots,j_{m_1+m_2}\}}\Var_{\bar\PP}\Big(\sum_{l=1}^{a^2_{k_2}}(-\tilde u^{k_2,2}_l)\eps w_{k_2}B^{H_{k_2}}(s^{k_2,2}_l)\Big).
\end{align*}

With this in hands and using that

\begin{align*}
 \Delta^{8m_1}_{\theta',t}\times\Delta^{8m_2}_{\theta,\theta'}\subseteq \Big(\prod_{k_1\in\{j_1,\dots,j_{m_1}\}}\Delta^{a^1_{k_1}}_{\theta',t}\Big)\times\Big(\prod_{k_2\in\{j_{m_1+1},\dots,j_{m_1+m_2}\}}\Delta^{a^2_{k_2}}_{\theta,\theta'}\Big)
\end{align*}

we can proceed until estimation \eqref{3IntermediateEstimate3} in the same manner as performed in the proof of \eqref{3CompactnessMalliavin3}. In order to make the calculations more readable, let us introduce the following short-hand notation $$\Hfct_{j_1,\dots,j_{m_1+m_2}}^{\otimes4}(\tau^{1,2},\sigma^{1,2}_{1,2};s,\theta):=\Hfct_{j_1,\dots,j_{m_1+m_2}}^{\otimes4}((\tau^1,\tau^2),(\sigma^1_1,\sigma^2_1),(\sigma^1_2,\sigma^2_2);s_{1},\dots,s_{4(m_1+m_2)},\theta).$$ Estimation \eqref{3IntermediateEstimate3} now becomes

\begin{align*}
 &\bar E\bigg[\bigg|(2\pi)^{-4(m_1+m_2)}\int_{\R^{4(m_1+m_2)}}\int_{\Delta^{4m_1}_{\theta',t}\times\Delta^{4m_2}_{\theta,\theta'}}\Hfct_{j_1,\dots,j_{m_1+m_2}}^{\otimes4}(\tau^{1,2},\sigma^{1,2}_{1,2};s,\theta)\\
 &\quad\cdot\prod_{l=1}^{4(m_1+m_2)}e^{-iu_l(Y_{j_{\alpha_l}}(s_l)-z_l)}(-iu_l)dsdu\bigg|^2\bigg]\\
 &\leq\sum_{(\rho^1,\rho^2)\in S(4m_1,4m_1)\times S(4m_2,4m_2)}\bigg\{\bigg[\prod_{k_1\in\{j_1,\dots,j_{m_1}\}}\left(\frac{3\sqrt{2}(1+r)}{\sqrt{\pi}}C_{k_1}^{-\frac{3}{2}}\eps ^{-3}|w_{k_1}|^{-3}\right)^{a^1_{k_1}}\\
 &\quad\cdot\int_{\Delta^{a^1_{k_1}}_{\theta',t}}\prod_{l=1}^{a^1_{k_1}}|s^{k_1,1}_l-s^{k_1,1}_{l+1}|^{-3H_{k_1}}ds^{k_1,1}\bigg]\\
 &\quad\cdot\bigg[\prod_{k_2\in\{j_{m_1+1},\dots,j_{m_1+m_2}\}}\left(\frac{3\sqrt{2}(1+r)}{\sqrt{\pi}}C_{k_2}^{-\frac{3}{2}}\eps ^{-3}|w_{k_2}|^{-3}\right)^{a^2_{k_2}}\\
 &\quad\cdot\int_{\Delta^{a^2_{k_2}}_{\theta,\theta'}}\prod_{l=1}^{a^2_{k_2}}|s^{k_2,2}_l-s^{k_2,2}_{l+1}|^{-3H_{k_2}}ds^{k_2,2}\bigg]\bigg\}\\
 &\leq|t-\theta'|^{8m_1\delta_H}|\theta'-\theta|^{8m_2\delta_H}\prod_{l=1}^{m_1+m_2}\left(\frac{12\sqrt{2}(1+r)\Gamma(\delta_H)}{\sqrt{\pi}}C_{j_l}^{-\frac{3}{2}}\eps ^{-3}|w_{j_l}|^{-3}\right)^8,
\end{align*}

where we used the same estimates for the integrals over $\Delta^{a^1_{k_1}}_{\theta',t}$ and $\Delta^{a^2_{k_2}}_{\theta,\theta'}$ that we have shown before. Following now the exact same steps from the proof of \eqref{3CompactnessMalliavin3} until \eqref{3FinalEstimationStep1}, we finally achieve

\begin{align}\label{3I3Estimate}
 I_3(\theta,\theta',t)\leq e^{\frac{1}{2}M^2T}\delta_T^{-4\delta_H}|t-\theta'|^{2\delta_H}|\theta-\theta'|^{2\delta_H}.
\end{align}

We are now ready for a conclusion. Observe that if $\beta<\delta_H$, each of the right handsides in \eqref{3I1Estimate}, \eqref{3I2Estimate} and \eqref{3I3Estimate} is integrable w.r.t. $\theta'$ and $\theta$. Finally, we observe that none of the estimates \eqref{3I1Estimate}, \eqref{3I2Estimate} and \eqref{3I3Estimate} depends on $n$. We therefore have

\begin{align}\label{3FinalEstimationStep2}
\begin{split}
 &\sup_{n\geq1}E'\bigg[\int_0^t \int_0^t \frac{E[|\DM_\theta x^{n}(t) - \DM_{\theta'} x^{n}(t)|^2]}{|\theta' - \theta|^{1+2\beta}} d\theta' d\theta\bigg]\\
 &\quad\leq 3\int_0^t \int_0^t \big(I_1(\theta,\theta')+I_2(\theta,\theta',t)+I_3(\theta,\theta',t)\big)d\theta' d\theta\\
 &\quad<\infty.
\end{split}
\end{align}

\end{proof}

%%%%%%%%%%%%%%%%%%%%%%%%%%%%%%%%%%%%%%%%%%%%%%%%%%%%%%%%%%%%%%%%%%%%%%%%%%%%%%%%%%%%%%%%%%%%%%%%%%%%%%%%%%%%%%%%%%%%%%%%%%%%%%%%%%%%%%%%%%%%%%%%%%

\subsection{The convergence result}

Before we provide a proof for the $L^2$-convergence of the sequence $(x^n(t))_{n\in\NN}$ for every $t\in[0,T]$, we need to recall some definitions and facts.

\begin{defn}\label{3DefWienerTrafo}
 Let $X\in L^2(\Omega\times\Omega',\F_t,P\otimes P')$. Further, we denote by $\mathbb{S}([0,t])$ the space of simple functions on $[0,t]$, i.e. functions of the form
 \begin{align}\label{3simplefct}
  \alpha(t)=\sum_{j=1}^l\alpha_j\ind_{[t_{j-1},t_j)},\quad\text{where }0=t_0<t_1<t_2<\dots<t_l\leq t,\text{ and }l\in\NN.
 \end{align}
 Let now, $\alpha_i\in\mathbb{S}([0,t])$ for all $i=1,2\dots$, with the corresponding time points $0<t_{i,1}<\dots<t_{i,l_i}\leq t$ and let $\phi\in L^2([0,t])$. We define the \emph{Wiener transform of} $X$ by
 \begin{align}\label{3WienerTrafo}
  \Wi_{\Omega\times\Omega'}(X)(\phi,\alpha):=\bar E\bigg[X\exp\Big(\int_0^t\phi(s)dW(s)+\sum_{i\geq1}\sum_{j=1}^{l_i}\alpha_{i,j}w_i(B^{H_i}(t_{i,j})-B^{H_i}(t_{i,j-1}))\Big)\bigg].
 \end{align}
 Similarly, for $X\in L^2(\Omega,\F^W_t,P)$, the Wiener transform is defined by
 \begin{align*}
  \Wi_{\Omega}(X)(\phi):=E\bigg[X\exp\Big(\int_0^t\phi(s)dW(s)\Big)\bigg],
 \end{align*}
 and for $X\in L^2(\Omega',\F^\BB_t,P')$, by
 \begin{align*}
  \Wi_{\Omega'}(X)(\alpha):=E'\bigg[X\exp\Big(\sum_{i\geq1}\sum_{j=1}^{l_i}\alpha_{i,j}w_i(B^{H_i}(t_{i,j})-B^{H_i}(t_{i,j-1}))\Big)\bigg].
 \end{align*}
\end{defn}

\begin{rem}\label{3Facts1and2}
 We recall the following facts about the Wiener transform:
 \begin{itemize}
  \item[(Fact 1)] Every $X\in L^2(\Omega\times\Omega',\F_t,P\otimes P')$ (or $X\in L^2(\Omega,\F^W_t,P)$, $X\in L^2(\Omega',\F^\BB_t,P')$) is uniquely defined by its Wiener transform, up to a nullset.
  \item[(Fact 2)] $X^n\stackrel{n\rightarrow\infty}{\longrightarrow} X$ weakly in $L^2(\Omega\times\Omega',\F_t,P\otimes P')$ if and only if
  \begin{align*}
   \Wi_{\Omega\times\Omega'}(X^n)(\phi,\alpha)\rightarrow \Wi_{\Omega\times\Omega'}(X)(\phi,\alpha),
  \end{align*}
  for all $\phi$ and $\alpha$ given as in Definition \ref{3DefWienerTrafo}.
 \end{itemize}
\end{rem}

\begin{cor}\label{3UniquenessWeakLimit}
 Let $X^n\rightarrow X$, weakly in $L^2$, and $X^n\rightarrow Y$, weakly in $L^2$. Then, $X=Y$ a.e., i.e. the weak limit in $L^2$ is unique (up to a nullset).
\end{cor}

\begin{proof}
 This is a direct consequence of Fact 1 and Fact 2: Since $X^n\rightarrow X$ and $X^n\rightarrow Y$, weakly in $L^2$, it holds for all $(\phi,\alpha)$ as in Definition \ref{3DefWienerTrafo},
 \begin{align*}
  \R\ni\Wi_{\Omega\times\Omega'}(X)(\phi,\alpha)\leftarrow\Wi_{\Omega\times\Omega'}(X^n)(\phi,\alpha)\rightarrow\Wi_{\Omega\times\Omega'}(Y)(\phi,\alpha)\in\R,
 \end{align*}
 which implies $\Wi_{\Omega\times\Omega'}(X)(\phi,\alpha)=\Wi_{\Omega\times\Omega'}(Y)(\phi,\alpha)$ and thus $X=Y$ a.e.
\end{proof}

From now on, let $x^n$ denote the sequence that we constructed in the section before and fix $t\in[0,T]$. The following lemmata will help us proving the strong convergence of $(x^n)_{n\in\NN}$ in $L^2$.

\begin{lem}\label{3WeaklyConvergingSubsequence}\hspace{12cm}
\begin{itemize}
 \item[(a)] The sequence $(x^n(t))_{n\in\NN}\subseteq L^2(\Omega\times\Omega',\F_t,P\otimes P')$ contains a subsequence that converges weakly in $L^2(\Omega\times\Omega',\F_t,P\otimes P')$.
 \item[(b)] Moreover, for every $n\in\NN$, $|x^n(t)|^2\in L^2(\Omega\times\Omega',\F_t,P\otimes P')$ and the sequence $(|x^n(t)|^2)_{n\in\NN}\subseteq L^2(\Omega\times\Omega',\F_t,P\otimes P')$ contains a subsequence that converges weakly in $L^2(\Omega\times\Omega',\F_t,P\otimes P')$.
\end{itemize}
\end{lem}

\begin{proof}
 First, we prove \emph{(b)}. Since every bounded sequence in a reflexive space (which $L^2(\Omega\times\Omega',\F_t,P\otimes P')$ is) has a weakly convergent subsequence, we only need to show that $(\||x^n(t)|^2\|_{L^2})_{n\in\NN}$ is bounded. In order to do so, note that for all $a,b,c\in\R$,
 $$(a+b+c)^4=((a+b+c)^2)^2\leq(3a^2+3b^2+3c^2)^2\leq27a^4+27b^4+27c^4.$$ Now, by application of Jensen's inequality, Fubini's theorem and the fact that $\mathcal{L}_{P\otimes P'}(W(t))=\N(0,t)$, we have
%  \begin{align*}
%   \|x^n(t)\|_{L^2}^2&=\bar E[|x^n(t)|^2]=\bar E\Big[\big|\eta(0)+\int_0^tb^n(x^n_s+\eps\BB(s))ds+W(t)\big|^2\Big]\\
%   &\leq 3|\eta(0)|^2+3\bar E\Big[\big|\int_0^tb^n(x^n_s+\eps\BB(s))ds\big|^2\Big]+3t\\
%   &\leq 3|\eta(0)|^2+3t\int_0^t\bar E\big[|b^n(x^n_s+\eps\BB(s))|^2\big]ds+3t\\
%   &\leq 3|\eta(0)|^2+3t^2\|b^n\|_\infty^2+3t\\
%   &\leq 3|\eta(0)|^2+3t^2\|b\|_\infty^2+3t<\infty,\quad\text{independently of }n.
%  \end{align*}
%  and therefore
 \begin{align*}
  \big\||x^n(t)|^2\big\|_{L^2}^2&=\bar E[|x^n(t)|^4]=\bar E\Big[\big|\eta(0)+\int_0^tb^n(x^n_s+\eps\BB(s))ds+W(t)\big|^4\Big]\\
  &\leq 27|\eta(0)|^4+27\bar E\Big[\big|\int_0^tb^n(x^n_s+\eps\BB(s))ds\big|^4\Big]+27\bar E[|W(t)|^4]\\
  &\leq 27|\eta(0)|^2+27t^3\int_0^t\bar E\big[|b^n(x^n_s+\eps\BB(s))|^4\big]ds+81t^2\\
  &\leq 27|\eta(0)|^4+27t^4\|b^n\|_\infty^4+81t^2\\
  &\leq 27|\eta(0)|^4+27t^4\|b\|_\infty^2+81t^2<\infty,\quad\text{independently of }n.
 \end{align*}
 In order to prove \emph{(a)}, we apply H\"{o}lder's inequality and get
 \begin{align*}
  \|x^n(t)\|_{L^2}^2=\bar E[|x^n(t)|^2]\leq\Big(\bar E[|x^n(t)|^4]\Big)^{\frac{1}{2}}\leq\sqrt{27|\eta(0)|^4+27t^4\|b\|_\infty^2+81t^2}<\infty,
 \end{align*}
 independently of $n$.
\end{proof}

We know from Lemma \ref{3WeaklyConvergingSubsequence} \emph{(a)} that there exists an $x(t)\in L^2(\Omega\times\Omega',\F_t,P\otimes P')$ and a subsequence $(n_k)_{k\in\NN}$ s.t. $x^{n_k}(t)\stackrel{k\rightarrow\infty}{\longrightarrow} x(t)$, weakly in $L^2(\Omega\times\Omega',\F_t,P\otimes P')$. Moreover, by Lemma \ref{3WeaklyConvergingSubsequence} \emph{(b)}, applied to the sequence $(x^{n_k}(t))_{k\geq1}$, we know that there exists an $y(t)\in L^2(\Omega\times\Omega',\F_t,P\otimes P')$ and a subsequence $(n_{k_l})_{l\in\NN}$ s.t. still $x^{n_{k_l}}(t)\stackrel{l\rightarrow\infty}{\longrightarrow} x(t)$, weakly in $L^2(\Omega\times\Omega',\F_t,P\otimes P')$ \emph{and} $\big(x^{n_{k_l}}(t)\big)^2\stackrel{l\rightarrow\infty}{\longrightarrow} y(t)$, weakly in $L^2(\Omega\times\Omega',\F_t,P\otimes P')$. Let w.l.o.g. already $x^{n}(t)\stackrel{n\rightarrow\infty}{\longrightarrow} x(t)$, weakly in $L^2(\Omega\times\Omega',\F_t,P\otimes P')$ and $\big(x^{n}(t)\big)^2\stackrel{n\rightarrow\infty}{\longrightarrow} y(t)$, weakly in $L^2(\Omega\times\Omega',\F_t,P\otimes P')$ (otherwise we re-define the sequence $x^n(t)$ as to be this subsequence that we just found). Note that $x(t)$ is $\F_t$-measurable as an element of $L^2(\Omega\times\Omega',\F_t,P\otimes P')$.\\ %\textcolor{red}{Note that if we could proof that $\big(x(t)\big)^2=y(t)$, we would immediately have strong convergence in $L^2(\Omega\times\Omega',\F_t,P\otimes P')$ of $x^n(t)$ to $x(t)$.}\\
\par
In this section we will prove that $x^n(t)$ converges indeed \emph{strongly} in $L^2(\Omega\times\Omega',\F_t,P\otimes P')$ to $x(t)$. The next lemma contains some technical results that we need in order to do so. Let us first introduce another notation: recall that $x^n(t)$, $x(t)$ denote functions $\Omega\times\Omega'\rightarrow\R$. We denote by $x^n(t,\cdot,\omega')$ and $x(t,\cdot,\omega')$ these functions where we plug in $\omega'\in\Omega'$. In other words, $x^n(t,\cdot,\omega')$ and $x(t,\cdot,\omega')$ are functions $\Omega\rightarrow\R$.

\begin{lem}\label{3WienerRepresentations}
 The following representations hold true
 \begin{itemize}
  \item[(a)] For a.e. $\omega'\in\Omega'$ we have
  \begin{align}
  \begin{split}
   &\Wi_\Omega\big(|x^n(t,\cdot,\omega')|^2\big)(\phi)\\
   &=E\bigg[(\eta(0)+W(t))^2\mathcal{E}\Big(\int_0^\cdot b^n(F_1(W_\cdot)_s+\eps w_1B^{H_1}(s,\omega'),\dots)dW(s)\Big)_t\\
   &\quad\cdot\exp\Big(-\int_0^t\phi(s)b^n(F_1(W_\cdot)_s+\eps w_1B^{H_1}(s,\omega'),\dots)ds\Big)\exp\Big(\int_0^t\phi(s)dW(s)\Big)\bigg]
  \end{split}
  \end{align}
  for every $\phi$ in a given dense subset in $L^2([0,T])$.
  \item[(b)] For a.e. $\omega'\in\Omega'$ we have
  \begin{align}
  \begin{split}
   &\Wi_\Omega\big(x^n(t,\cdot,\omega')\big)(\phi)\\
   &=E\bigg[(\eta(0)+W(t))\mathcal{E}\Big(\int_0^\cdot b^n(F_1(W_\cdot)_s+\eps w_1B^{H_1}(s,\omega'),\dots)dW(s)\Big)_t\\
   &\quad\cdot\exp\Big(-\int_0^t\phi(s)b^n(F_1(W_\cdot)_s+\eps w_1B^{H_1}(s,\omega'),\dots)ds\Big)\exp\Big(\int_0^t\phi(s)dW(s)\Big)\bigg]
  \end{split}
  \end{align}
  for every $\phi$ in a given dense subset in $L^2([0,T])$.
  \item[(c)] For a.e. $\omega'\in\Omega'$ we have
  \begin{align}
  \begin{split}
   &\Wi_\Omega\big(x(t,\cdot,\omega')\big)(\phi)\\
   &=E\bigg[(\eta(0)+W(t))\mathcal{E}\Big(\int_0^\cdot b(F_1(W_\cdot)_s+\eps w_1B^{H_1}(s,\omega'),\dots)dW(s)\Big)_t\\
   &\quad\cdot\exp\Big(-\int_0^t\phi(s)b(F_1(W_\cdot)_s+\eps w_1B^{H_1}(s,\omega'),\dots)ds\Big)\exp\Big(\int_0^t\phi(s)dW(s)\Big)\bigg]
  \end{split}
  \end{align}
  for every $\phi$ in a given dense subset in $L^2([0,T])$.
  \item[(d)] For a.e. $\omega'\in\Omega'$ we have
  \begin{align}
  \begin{split}
   &\Wi_\Omega\big(y(t,\cdot,\omega')\big)(\phi)\\
   &=E\bigg[(\eta(0)+W(t))^2\mathcal{E}\Big(\int_0^\cdot b(F_1(W_\cdot)_s+\eps w_1B^{H_1}(s,\omega'),\dots)dW(s)\Big)_t\\
   &\quad\cdot\exp\Big(-\int_0^t\phi(s)b(F_1(W_\cdot)_s+\eps w_1B^{H_1}(s,\omega'),\dots)ds\Big)\exp\Big(\int_0^t\phi(s)dW(s)\Big)\bigg]
  \end{split}
  \end{align}
  for every $\phi$ in a given dense subset in $L^2([0,T])$.
 \end{itemize}
\end{lem}

\begin{proof}
 In order to show \emph{(a)}, we fix, for the moment $\phi\in L^2([0,t])$. Then, for all $\alpha$ as in Definition \ref{3DefWienerTrafo} with corresponding $0<t_1<\dots<t_{i,l_i}\leq t$, we have
%  \begin{align*}
%   &\Wi_{\Omega'}\Big(\omega'\mapsto\Wi_\Omega\big(|x^n(t)|^2\big)(\phi)\Big)(\alpha)\\
%   &=\int_{\Omega'}\Wi_\Omega\big(|x^n(t)|^2\big)(\phi)(\omega')\exp\Big(\sum_{i\geq1}\sum_{j=1}^{l_i}\alpha_{i,j}w_i(B^{H_i}(t_{i,j},\omega')-B^{H_i}(t_{i,j-1},\omega'))\Big)P'(d\omega')\\
%   &=\int_{\Omega'}\int_{\Omega}\big|x^n(t,\omega,\omega')\big|^2\exp\Big(\int_0^t\phi(s)dW(s)\Big)\\
%   &\qquad\cdot\exp\Big(\sum_{i\geq1}\sum_{j=1}^{l_i}\alpha_{i,j}w_i(B^{H_i}(t_{i,j},\omega')-B^{H_i}(t_{i,j-1},\omega'))\Big)P(d\omega)P'(d\omega')\\
%   &=\Wi_{\Omega'}\Big(\omega'\mapsto\Wi_\Omega\big(|x^n(t,\cdot,\omega')|^2\big)(\phi)\Big)(\alpha)
%  \end{align*}
\begin{align*}
  &\Wi_{\Omega'}\Big(\omega'\mapsto\Wi_\Omega\big(|x^n(t,\cdot,\omega')|^2\big)(\phi)\Big)(\alpha)\\
  &=\int_{\Omega'}\Wi_\Omega\big(|x^n(t,\cdot,\omega')|^2\big)(\phi)\exp\Big(\sum_{i\geq1}\sum_{j=1}^{l_i}\alpha_{i,j}w_i(B^{H_i}(t_{i,j},\omega')-B^{H_i}(t_{i,j-1},\omega'))\Big)P'(d\omega')\\
  &=\int_{\Omega'}\int_{\Omega}\big|x^n(t,\omega,\omega')\big|^2\exp\Big(\int_0^t\phi(s)dW(s,\omega)\Big)\\
  &\quad\cdot\exp\Big(\sum_{i\geq1}\sum_{j=1}^{l_i}\alpha_{i,j}w_i(B^{H_i}(t_{i,j},\omega')-B^{H_i}(t_{i,j-1},\omega'))\Big)P(d\omega)P'(d\omega')\\
  &=\bar E\Big[|x^n(t)|^2\exp\Big(\int_0^t\phi(s)dW(s)+\sum_{i\geq1}\sum_{j=1}^{l_i}\alpha_{i,j}w_i(B^{H_i}(t_{i,j})-B^{H_i}(t_{i,j-1}))\Big)\Big].
 \end{align*}
 
Defining

\begin{align*}
 \widetilde W^n(t):=W(t)+\int_0^tb^n(x^n_s+\eps\BB(s))ds,
\end{align*}

we have that $x^n(t)=\eta(0)+\widetilde W^n(t)$ and, by Girsanov's theorem, $\widetilde W^n$ is a Brownian motion under the probability measure $\widetilde P^n$ given by

\begin{align*}
 \frac{d\widetilde P^n}{dP\otimes P'}\Big|_{\F_t}&=\mathcal{E}\Big(-\int_0^\cdot b^n(x^n_s+\eps\BB(s))dW(s)\Big)_t,
\end{align*}

which is equivalent to

\begin{align*}
 \frac{dP\otimes P'}{d\widetilde P^n}\Big|_{\F_t}&=\mathcal{E}\Big(\int_0^\cdot b^n(x^n_s+\eps\BB(s))d\widetilde W^n(s)\Big)_t\\
 &=\mathcal{E}\Big(\int_0^\cdot b^n(F_1(\widetilde W^n_\cdot)_s+\eps w_1B^{H_1}(s),\dots)d\widetilde W^n(s)\Big)_t.
\end{align*}

Plugging this in and exploiting the fact that the process $(\widetilde W^n,\BB)$ has under $\widetilde P^n$ the same law as the process $(W,\BB)$ has under $P\otimes P'$, we achieve

\begin{align*}
  &\Wi_{\Omega'}\Big(\omega'\mapsto\Wi_\Omega\big(|x^n(t,\cdot,\omega')|^2\big)(\phi)\Big)(\alpha)\\
  &=E_{\widetilde P^n}\Big[\big|\eta(0)+\widetilde W^n(t)\big|^2\mathcal{E}\Big(\int_0^\cdot b^n(F_1(\widetilde W^n_\cdot)_s+\eps w_1B^{H_1}(s),\dots)d\widetilde W^n(s)\Big)_t\\
  &\quad\cdot\exp\Big(\int_0^t\phi(s)d\widetilde W^n(s)-\int_0^t\phi(s)b^n(F_1(\widetilde W^n_\cdot)_s+\eps w_1B^{H_1}(s),\dots)ds\\
  &\quad+\sum_{i\geq1}\sum_{j=1}^{l_i}\alpha_{i,j}w_i(B^{H_i}(t_{i,j})-B^{H_i}(t_{i,j-1}))\Big)\Big]\\
  &=\bar E\Big[\big|\eta(0)+W(t)\big|^2\mathcal{E}\Big(\int_0^\cdot b^n(F_1(W_\cdot)_s+\eps w_1B^{H_1}(s),\dots)dW(s)\Big)_t\\
  &\quad\cdot\exp\Big(\int_0^t\phi(s)dW(s)-\int_0^t\phi(s)b^n(F_1(W_\cdot)_s+\eps w_1B^{H_1}(s),\dots)ds\\
  &\quad+\sum_{i\geq1}\sum_{j=1}^{l_i}\alpha_{i,j}w_i(B^{H_i}(t_{i,j})-B^{H_i}(t_{i,j-1}))\Big)\Big]\\
  &=E'\bigg[E\Big[\big|\eta(0)+W(t)\big|^2\mathcal{E}\Big(\int_0^\cdot b^n(F_1(W_\cdot)_s+\eps w_1B^{H_1}(s),\dots)dW(s)\Big)_t\\
  &\quad\cdot\exp\Big(\int_0^t\phi(s)dW(s)-\int_0^t\phi(s)b^n(F_1(W_\cdot)_s+\eps w_1B^{H_1}(s),\dots)ds\Big)\Big]\\
  &\quad\cdot\exp\Big(\sum_{i\geq1}\sum_{j=1}^{l_i}\alpha_{i,j}w_i(B^{H_i}(t_{i,j})-B^{H_i}(t_{i,j-1}))\Big)\bigg]\\
  &=\Wi_{\Omega'}\Big(\omega'\mapsto E\Big[\big|\eta(0)+W(t)\big|^2\mathcal{E}\Big(\int_0^\cdot b^n(F_1(W_\cdot)_s+\eps w_1B^{H_1}(s,\omega'),\dots)dW(s)\Big)_t\\
  &\quad\cdot\exp\Big(\int_0^t\phi(s)dW(s)-\int_0^t\phi(s)b^n(F_1(W_\cdot)_s+\eps w_1B^{H_1}(s,\omega'),\dots)ds\Big)\Big]\Big).
 \end{align*}

 It follows now from Corollary \ref{3UniquenessWeakLimit} that for our particular choice of $\phi$,
 
 \begin{align*}
  &\Wi_\Omega\big(|x^n(t,\cdot,\omega')|^2\big)(\phi)\\
   &=E\bigg[(\eta(0)+W(t))^2\mathcal{E}\Big(\int_0^\cdot b^n(F_1(W_\cdot)_s+\eps w_1B^{H_1}(s),\dots)dW(s)\Big)_t\\
   &\quad\cdot\exp\Big(-\int_0^t\phi(s)b^n(F_1(W_\cdot)_s+\eps w_1B^{H_1}(s),\dots)ds\Big)\exp\Big(\int_0^t\phi(s)dW(s)\Big)\bigg],
 \end{align*}

 for $P'$-a.e. $\omega'\in\Omega'$. Or in other words: there exists an $\Omega'_1(\phi)\subseteq\Omega'$ s.t. $P'(\Omega'_1(\phi))=1$ and for all $\omega'\in\Omega'_1(\phi)$ the above equality holds. Now, since $L^2([0,t])$ is \emph{separable}, there is a countable dense set $\Phi\subseteq L^2([0,t])$. Setting
 
 \begin{align*}
  \Omega'_1:=\bigcap_{\phi\in\Phi}\Omega'_1(\phi),
 \end{align*}

 we still have $P'(\Omega'_1)=1$ and the equation holds for all $\omega'\in\Omega'_1$.\\
 \par
 The proof of \emph{(b)} works exactly the same way as the proof of \emph{(a)}. Now, with the same dense set $\Phi\subseteq L^2([0,t])$, we define
 \begin{align*}
  \Omega'_2:=\bigcap_{\phi\in\Phi}\Omega'_2(\phi),
 \end{align*}
 and have $P'(\Omega'_2)=1$.\\
 \par
 In order to show \emph{(c)}, fix again $\phi\in L^2([0,t])$. Note that, by Lemma \ref{3WeaklyConvergingSubsequence}, $X^n\stackrel{n\rightarrow\infty}{\longrightarrow}X(t)$ weakly in $L^2(\Omega\times\Omega',\F_t,P\otimes P')$. It follows from Remark \ref{3Facts1and2} that, for all $\alpha$ as in Definition \ref{3DefWienerTrafo},
 
 \begin{align*}
  &\Wi_{\Omega'}\Big(\omega'\mapsto\Wi_\Omega\big(x^n(t,\cdot,\omega')\big)(\phi)\Big)(\alpha)\\
  %&=\int_{\Omega'}\Wi_\Omega\big(|x^n(t,\cdot,\omega')|^2\big)(\phi)\exp\Big(\sum_{i\geq1}\sum_{j=1}^{l_i}\alpha_{i,j}w_i(B^{H_i}(t_{i,j},\omega')-B^{H_i}(t_{i,j-1},\omega'))\Big)P'(d\omega')\\
  %&=\int_{\Omega'}\int_{\Omega}\big|x^n(t,\omega,\omega')\big|^2\exp\Big(\int_0^t\phi(s)dW(s,\omega)\Big)\\
  %&\qquad\cdot\exp\Big(\sum_{i\geq1}\sum_{j=1}^{l_i}\alpha_{i,j}w_i(B^{H_i}(t_{i,j},\omega')-B^{H_i}(t_{i,j-1},\omega'))\Big)P(d\omega)P'(d\omega')\\
  &=\bar E\Big[x^n(t)\exp\Big(\int_0^t\phi(s)dW(s)+\sum_{i\geq1}\sum_{j=1}^{l_i}\alpha_{i,j}w_i(B^{H_i}(t_{i,j})-B^{H_i}(t_{i,j-1}))\Big)\Big]\\
  &=\Wi_{\Omega\times\Omega'}\big(x^n(t)\big)(\phi,\alpha)\\
  &\rightarrow\Wi_{\Omega\times\Omega'}\big(x(t)\big)(\phi,\alpha),\quad\text{as }n\rightarrow\infty.
 \end{align*}
 
 On the other hand, by \emph{(b)}, we know that
 \begin{align*}
   &\Wi_\Omega\big(x^n(t,\cdot,\omega')\big)(\phi)\\
   &=E\bigg[(\eta(0)+W(t))\mathcal{E}\Big(\int_0^\cdot b^n(F_1(W_\cdot)_s+\eps w_1B^{H_1}(s),\dots)dW(s)\Big)_t\\
   &\quad\cdot\exp\Big(-\int_0^t\phi(s)b^n(F_1(W_\cdot)_s+\eps w_1B^{H_1}(s),\dots)ds\Big)\exp\Big(\int_0^t\phi(s)dW(s)\Big)\bigg],
 \end{align*}
 and thus, by the pointwise convergence of $b^n$ to $b$ and the boundedness of $b^n$, it follows from dominated convergence that
 \begin{align*}
  &\Wi_{\Omega'}\Big(\omega'\mapsto\Wi_\Omega\big(x^n(t,\cdot,\omega')\big)(\phi)\Big)(\alpha)\\
  &=\bar E\Big[(\eta(0)+W(t))\mathcal{E}\Big(\int_0^\cdot b^n(F_1(W_\cdot)_s+\eps w_1B^{H_1}(s),\dots)dW(s)\Big)_t\\
   &\quad\cdot\exp\Big(-\int_0^t\phi(s)b^n(F_1(W_\cdot)_s+\eps w_1B^{H_1}(s),\dots)ds\Big)\exp\Big(\int_0^t\phi(s)dW(s)\Big)\\
   &\quad\cdot\exp\Big(\sum_{i\geq1}\sum_{j=1}^{l_i}\alpha_{i,j}w_i(B^{H_i}(t_{i,j})-B^{H_i}(t_{i,j-1}))\Big)\Big]\\
  &\stackrel{n\rightarrow\infty}{\longrightarrow}\bar E\Big[(\eta(0)+W(t))\mathcal{E}\Big(\int_0^\cdot b(F_1(W_\cdot)_s+\eps w_1B^{H_1}(s),\dots)dW(s)\Big)_t\\
   &\quad\cdot\exp\Big(-\int_0^t\phi(s)b(F_1(W_\cdot)_s+\eps w_1B^{H_1}(s),\dots)ds\Big)\exp\Big(\int_0^t\phi(s)dW(s)\Big)\\
   &\quad\cdot\exp\Big(\sum_{i\geq1}\sum_{j=1}^{l_i}\alpha_{i,j}w_i(B^{H_i}(t_{i,j})-B^{H_i}(t_{i,j-1}))\Big)\Big]\\
   &=\Wi_{\Omega'}\Big(\omega'\mapsto E\Big[(\eta(0)+W(t))\mathcal{E}\Big(\int_0^\cdot b(F_1(W_\cdot)_s+\eps w_1B^{H_1}(s),\dots)dW(s)\Big)_t\\
   &\quad\cdot\exp\Big(-\int_0^t\phi(s)b(F_1(W_\cdot)_s+\eps w_1B^{H_1}(s),\dots)ds\Big)\exp\Big(\int_0^t\phi(s)dW(s)\Big)\Big]\Big)(\alpha).
 \end{align*}
 This proves that there exists an $\Omega'_3(\phi)\subseteq\Omega'$ with $P'(\Omega'_3(\phi))=1$ s.t. for all $\omega'\in\Omega'_3(\phi)$
 \begin{align*}
  &\Wi_\Omega\big(x^n(t,\cdot,\omega')\big)(\phi)\\
  &=E\Big[(\eta(0)+W(t))\mathcal{E}\Big(\int_0^\cdot b(F_1(W_\cdot)_s+\eps w_1B^{H_1}(s,\omega'),\dots)dW(s)\Big)_t\\
   &\quad\cdot\exp\Big(-\int_0^t\phi(s)b(F_1(W_\cdot)_s+\eps w_1B^{H_1}(s,\omega'),\dots)ds\Big)\exp\Big(\int_0^t\phi(s)dW(s)\Big)\Big]
 \end{align*}
 Now, with the same dense set $\Phi\subseteq L^2([0,t])$, we define
 \begin{align*}
  \Omega'_3:=\bigcap_{\phi\in\Phi}\Omega'_3(\phi).
 \end{align*}
 The proof of \emph{(d)} follows the same lines as the proof of \emph{(c)}: Since $\big(x^n(t)\big)^2\rightarrow y(t)$, weakly in $L^2(\Omega\times\Omega',\F_t,P\otimes P')$, we have
 \begin{align*}
  \Wi_{\Omega'}\Big(\omega'\mapsto\Wi_\Omega\big(|x^n(t,\cdot,\omega')|^2\big)(\phi)\Big)(\alpha)&=\Wi_{\Omega\times\Omega'}\big(|x^n(t)|^2\big)(\phi,\alpha)\\
  &\rightarrow\Wi_{\Omega\times\Omega'}\big(y(t)\big)(\phi,\alpha),\quad\text{as }n\rightarrow\infty.
 \end{align*}
 On the other hand, \emph{(a)} and the dominated convergence theorem yield
 \begin{align*}
  &\Wi_{\Omega'}\Big(\omega'\mapsto\Wi_\Omega\big(|x^n(t,\cdot,\omega')|^2\big)(\phi)\Big)(\alpha)\\
  &\rightarrow \Wi_{\Omega'}\Big(\omega'\mapsto E\Big[\big|\eta(0)+W(t)\big|^2\mathcal{E}\Big(\int_0^\cdot b(F_1(W_\cdot)_s+\eps w_1B^{H_1}(s,\omega'),\dots)dW(s)\Big)_t\\
  &\quad\cdot\exp\Big(\int_0^t\phi(s)dW(s)-\int_0^t\phi(s)b(F_1(W_\cdot)_s+\eps w_1B^{H_1}(s,\omega'),\dots)ds\Big)\Big]\Big).
 \end{align*}
 Therefore, there exists $\Omega'_4(\phi)\subseteq\Omega'$ with $P'(\Omega'_4(\phi))=1$ s.t. for all $\omega'\in\Omega'_4(\phi)$
 \begin{align*}
  &\Wi_\Omega\big(y(t,\cdot,\omega')\big)(\phi)\\
  &=E\Big[|\eta(0)+W(t)|^2\mathcal{E}\Big(\int_0^\cdot b(F_1(W_\cdot)_s+\eps w_1B^{H_1}(s,\omega'),\dots)dW(s)\Big)_t\\
   &\quad\cdot\exp\Big(-\int_0^t\phi(s)b(F_1(W_\cdot)_s+\eps w_1B^{H_1}(s,\omega'),\dots)ds\Big)\exp\Big(\int_0^t\phi(s)dW(s)\Big)\Big].
 \end{align*}
 Finally, with the same dense set $\Phi\subseteq L^2([0,t])$, we define
 \begin{align*}
  \Omega'_4:=\bigcap_{\phi\in\Phi}\Omega'_4(\phi).
 \end{align*} 
\end{proof}

\begin{cor}\label{3CorollaryFromWienerRepresentations}
 It holds for every $\omega'\in\Omega'_1\cap\Omega'_2\cap\Omega'_3\cap\Omega'_4$:
 \begin{align*}
  E\big[\big(x^n(t,\cdot,\omega')\big)^2\big]&=E\Big[(\eta(0)+W(t))^2\mathcal{E}\Big(\int_0^\cdot b^n(F_1(W_\cdot)_s+\eps w_1B^{H_1}(s,\omega'),\dots)dW(s)\Big)_t\Big]\\
  &\rightarrow E\Big[(\eta(0)+W(t))^2\mathcal{E}\Big(\int_0^\cdot b(F_1(W_\cdot)_s+\eps w_1B^{H_1}(s,\omega'),\dots)dW(s)\Big)_t\Big]\\
  &=E\big[y(t,\cdot,\omega')\big],\quad\text{as }n\rightarrow\infty.
 \end{align*}
\end{cor}

\begin{proof}
 Lemma \ref{3WienerRepresentations} \emph{(a)} and \emph{(d)} with $\phi=0$.
\end{proof}

We are now ready to prove the main theorem of this section.

\begin{thm}\label{3StrongConvergenceL2}
 $x^n(t)\stackrel{n\rightarrow\infty}{\longrightarrow}x(t)$ strongly in $L^2(\Omega\times\Omega',\F_t,P\otimes P')$.
\end{thm}

\begin{proof}
 \textit{Step 1: Apply compactness results from previous section}\\
 Recall from Lemma \ref{3VI_relcomp} and Corollary \ref{3VI_compactcrit} that for a.e. fixed $\omega'\in\Omega'$ (or in other words for every $\omega'\in\Omega'_0$ for some $\Omega'_0\subseteq\Omega'$ with $P'(\Omega'_0)=1$) there exists a subsequence $n_k(\omega')$ such that $n_k(\omega')\stackrel{k\rightarrow\infty}{\longrightarrow}\infty$ and $(x^{n_k(\omega')}(t,\cdot,\omega'))_{k\geq1}$ is a relatively compact set. This implies that for every $\omega'\in\Omega'_0$ it exists a sub-subsequence $(n_{k_q}(\omega'))_{q\geq1}\subseteq\NN$ s.t. $n_{k_q}(\omega')\stackrel{q\rightarrow\infty}{\longrightarrow}\infty$ and it exists an $\check x(t,\cdot,\omega')\in L^2(\Omega,\F^W_t,P)$ such that
 \begin{align}
  x^{n_{k_q}(\omega')}(t,\cdot,\omega')\stackrel{q\rightarrow\infty}{\longrightarrow}\check x(t,\cdot,\omega'),\quad\text{ strongly in }L^2(\Omega,\F^W_t,P).
 \end{align}
 From now on, we assume w.l.o.g. that $(n_k(\omega'))_{k\geq1}$ denotes already the sub-subsequence that converges. Furthermore, we define $\Omega^*:=\bigcap_{i=0}^4\Omega'_i$. Then $P'(\Omega^*)=1$.\\
 \par
 \textit{Step 2: Prove that $x^n(t,\cdot,\omega')\stackrel{n\rightarrow\infty}{\longrightarrow}x(t,\cdot,\omega')$ strongly in $L^2(\Omega,\F^W_t,P)$}\\
 Recall from Corollary \ref{3CorollaryFromWienerRepresentations} that for every $\omega'\in\Omega^*$
 \begin{align}\label{3Star}
  E\big[\big(x^n(t,\cdot,\omega')\big)^2\big]\rightarrow E\big[y(t,\cdot,\omega')\big],\quad\text{as }n\rightarrow\infty.
 \end{align}
 Since at this point, $\omega'\in\Omega^*$ is \emph{fixed}, and convergence of a sequence implies also convergence of any subsequence, we can replace $x^n(t,\cdot,\omega')$ by $x^{n_k(\omega')}(t,\cdot,\omega')$ and get
 \begin{align*}
  E\big[\big(x^{n_k(\omega')}(t,\cdot,\omega')\big)^2\big]\rightarrow E\big[y(t,\cdot,\omega')\big],\quad\text{as }n\rightarrow\infty.
 \end{align*}
 On the other hand, the strong convergence of $x^{n_k(\omega')}(t,\cdot,\omega')$ to $\check x(t,\cdot,\omega')$ in $L^2(\Omega,\F^W_t,P)$ which was shown in the previous step implies that
 \begin{align*}
  E\big[\big(x^{n_k(\omega')}(t,\cdot,\omega')\big)^2\big]\rightarrow E\big[\big(\check x(t,\cdot,\omega')\big)^2\big],\quad\text{as }n\rightarrow\infty,
 \end{align*}
 and therefore we have
 \begin{align*}
  E\big[y(t,\cdot,\omega')\big]= E\big[\big(\check x(t,\cdot,\omega')\big)^2\big].
 \end{align*}
 Plugging this into \eqref{3Star}, we see that, for all $\omega'\in\Omega^*$,
 \begin{align}\label{3A}
  E\big[\big(x^n(t,\cdot,\omega')\big)^2\big]\rightarrow E\big[\big(\check x(t,\cdot,\omega')\big)^2\big],\quad\text{as }n\rightarrow\infty.
 \end{align}
 Moreover, by Lemma \ref{3WienerRepresentations} \emph{(b)} and \emph{(c)}, we have for all $\phi\in\Phi$ and all $\omega'\in\Omega^*$
 \begin{align}\label{32Star}
  \begin{split}
   &\Wi_\Omega\big(x^n(t,\cdot,\omega')\big)(\phi)\\
   &=E\bigg[(\eta(0)+W(t))\mathcal{E}\Big(\int_0^\cdot b^n(F_1(W_\cdot)_s+\eps w_1B^{H_1}(s,\omega'),\dots)dW(s)\Big)_t\\
   &\quad\cdot\exp\Big(-\int_0^t\phi(s)b^n(F_1(W_\cdot)_s+\eps w_1B^{H_1}(s,\omega'),\dots)ds\Big)\exp\Big(\int_0^t\phi(s)dW(s)\Big)\bigg]\\
   &\rightarrow E\bigg[(\eta(0)+W(t))\mathcal{E}\Big(\int_0^\cdot b(F_1(W_\cdot)_s+\eps w_1B^{H_1}(s,\omega'),\dots)dW(s)\Big)_t\\
   &\quad\cdot\exp\Big(-\int_0^t\phi(s)b(F_1(W_\cdot)_s+\eps w_1B^{H_1}(s,\omega'),\dots)ds\Big)\exp\Big(\int_0^t\phi(s)dW(s)\Big)\bigg]\\
   &=\Wi_\Omega\big(x(t,\cdot,\omega')\big)(\phi)
  \end{split}
 \end{align}
 Here we applied dominated convergence (recall the pointwise convergence of $b^n$ to $b$ and the boundness of $b^n$ by $\|b\|_\infty$). %\textcolor{red}{This does not trivially follow from Fact 2 since $x^n(t)\rightarrow x(t)$ weakly in $L^2(\Omega\times\Omega')$ not necessarily implies $x^n(t,\cdot,\omega')\rightarrow x(t,\cdot,\omega')$ in $L^2(\Omega)$.}
 Again, since $\omega'\in\Omega^*$ is \emph{fixed}, we can replace $x^n(t,\cdot,\omega')$ by $x^{n_k(\omega')}(t,\cdot,\omega')$ and get
 \begin{align*}
  \Wi_\Omega\big(x^{n_k(\omega')}(t,\cdot,\omega')\big)(\phi)\rightarrow\Wi_\Omega\big(x(t,\cdot,\omega')\big)(\phi).
 \end{align*}
 On the other hand, we know, by the strong convergence of $x^{n_k(\omega')}(t,\cdot,\omega')$ to $\check x(t,\cdot,\omega')$ in $L^2(\Omega,\F^W_t,P)$ that
 \begin{align*}
  \Wi_\Omega\big(x^{n_k(\omega')}(t,\cdot,\omega')\big)(\phi)\rightarrow\Wi_\Omega\big(\check x(t,\cdot,\omega')\big)(\phi),
 \end{align*}
 which implies that $\Wi_\Omega\big(x(t,\cdot,\omega')\big)(\phi)=\Wi_\Omega\big(\check x(t,\cdot,\omega')\big)(\phi)$. By Remark \ref{3Facts1and2}, This implies that for every $\omega'\in\Omega^*$,
 \begin{align}\label{3XequalsXtilde}
  x(t,\cdot,\omega')=\check x(t,\cdot,\omega')\quad\text{ as elements in }L^2(\Omega,\F^W_t,P).
 \end{align}
 Plugging this into \eqref{32Star}, we have
 \begin{align}\label{3B}
 \begin{split}
  \Wi_\Omega\big(x^n(t,\cdot,\omega')\big)(\phi)&\stackrel{n\rightarrow\infty}{\longrightarrow} \Wi_\Omega\big(\check x(t,\cdot,\omega')\big)(\phi),\quad\text{or, in other words,}\\
   x^n(t,\cdot,\omega')&\stackrel{n\rightarrow\infty}{\longrightarrow}\check x(t,\cdot,\omega'),\quad\text{weakly in }L^2(\Omega,\F^W_t,P).
 \end{split}
 \end{align}
 Recall that in any Hilbert space,
 \begin{align*}
  \text{convergence of the norms }+\text{ weak convergence }\Rightarrow\text{ strong convergence.}
 \end{align*}
 Therefore, \eqref{3A} and \eqref{3B} together with \eqref{3XequalsXtilde} imply
 \begin{align*}
  x^n(t,\cdot,\omega')\stackrel{n\rightarrow\infty}{\longrightarrow}\check x(t,\cdot,\omega')=x(t,\cdot,\omega')\quad\text{strongly in }L^2(\Omega,\F^W_t,P),
 \end{align*}
 for each $\omega'\in\Omega^*$ separately.\\
 \par
 \textit{Step 3: Prove that $x^n(t)\stackrel{n\rightarrow\infty}{\longrightarrow}x(t)$ strongly in $L^2(\Omega\times\Omega',\F_t,P\otimes P')$}\\
 We know from the previous step that
 \begin{align}\label{3C}
  E\big[\big(x^n(t,\cdot,\omega')-x(t,\cdot,\omega')\big)^2\big]\stackrel{n\rightarrow\infty}{\longrightarrow}0,\quad\text{for }P'-\text{a.e. }\omega'\in\Omega'.
 \end{align}
 Furthermore, $\omega'\mapsto E[(x(t,\cdot,\omega'))^2]\in L^1(\Omega',\F^\BB_t,P')$ since $x(t)\in L^2(\Omega\times\Omega',\F_t,P\otimes P')$. Moreover, we have
 \begin{align*}
  E\big[\big(x^n(t,\cdot,\omega')\big)^2\big]&=\int_\Omega\big(x^n(t,\omega,\omega')\big)^2P(d\omega)\\
  &=\int_\Omega\Big[\Big(\eta(0)+\int_0^tb^n(x^n_s+\eps\BB(s))ds+W(t)\Big)(\omega,\omega')\Big]^2P(d\omega)\\
  &\leq\int_\Omega\Big(3\eta(0)^2+3\big(\int_0^tb^n(x^n_s+\eps\BB(s))ds\big)^2+3(W(t))^2\Big)(\omega,\omega')P(d\omega)\\
  &\leq\int_\Omega\Big(3\eta(0)^2+3t^2\|b\|_\infty^2+3(W(t,\omega))^2\Big)P(d\omega)\\
  &=3\eta(0)^2+3t^2\|b\|_\infty^2+3t<\infty.
 \end{align*}
 Using that, $(a-b)^2\leq 2a^2+2b^2$, we therefore have
 \begin{align}\label{3D}
 \begin{split}
  E\big[\big(x^n(t,\cdot,\omega')-x(t,\cdot,\omega')\big)^2\big]&\leq 6\eta(0)^2+6t^2\|b\|_\infty^2+6t+2E[(x(t,\cdot,\omega'))^2]\\
  &=:g(\omega')\in L^1(\Omega',\F^\BB_t,P').
 \end{split}
 \end{align}
 It follows now from \eqref{3C} and \eqref{3D} by dominated convergence that
 \begin{align*}
  \bar E\big[\big(x^n(t)-x(t)\big)^2\big]=\int_{\Omega'}E\big[\big(x^n(t,\cdot,\omega')-x(t,\cdot,\omega')\big)^2\big]P'(d\omega')\stackrel{n\rightarrow\infty}{\longrightarrow}0.
 \end{align*}
 This finishes the proof.
\end{proof}

\subsection{The strong solution}

In order to prove the existence of a strong solution to \eqref{3SFDE}, we first define the process $\tilde x:[0,T]\times\Omega\times\Omega'\rightarrow\R$ by

\begin{align}\label{3Xtildeprocess}
 \tilde x(t,\omega,\omega'):=\eta(0)+W(t,\omega),\quad x_0=\eta.
\end{align}

Moreover, we set

\begin{align}\label{3Wtilde}
\begin{split}
 \widetilde W(t,\omega,\omega')&:=W(t,\omega)-\int_0^tb(F_1(W_\cdot(\omega))_s+\eps w_1B^{H_1}(s,\omega'),\dots)ds\\
 &=W(t,\omega)-\int_0^tb(\tilde x_s(\omega,\omega')+\eps\BB(s,\omega'))ds
\end{split}
\end{align}

Then, we have

\begin{align*}
 \tilde x(t)=\eta(0)+\int_0^tb(\tilde x_s+\eps\BB(s))ds+\widetilde W(t),
\end{align*}

and it follows by Girsanov's theorem that, under the new measure $\widetilde P$ given by

\begin{align*}
 \frac{d\widetilde P}{dP\otimes P'}\Big|_{\F_t}=\mathcal{E}\Big(\int_0^\cdot b(F_1(W_\cdot(\omega))_s+\eps w_1B^{H_1}(s,\omega'),\dots)dW(s)\Big)_t
\end{align*}

$\widetilde W$ is a Brownian motion which is independent of $\BB$. In other words, we have constructed a weak solution $\tilde x$ to \eqref{3SFDE} under the measure $\widetilde P$.

\begin{rem}\label{3VI_stochbasisrmk}
As outlined in the scheme above, the main challenge to establish existence of a strong solution is now to show that $\tilde x$ is $(\widetilde\F_t)_{t\in[0,T]}$-adapted. Indeed, in that case, there exists a family of measurable functionals $\psi(t,\cdot,\cdot):C([0,T];\R)\times C([0,T];M_2)\rightarrow M_2$, $t\in [0,T]$ such that $\tilde x_t=\psi(t,(\widetilde W_\cdot)_{[0,t]},(\BB_\cdot)_{[0,t]}\big)$ (see e.g. \cite{MMNPZ10} or \cite{MBP06} for an explicit form of $\psi(t,\cdot,\cdot)$ in the case of SDEs without delay driven by Brownian noise or L\'{e}vy noise, respectively), and for any other stochastic basis $(\hat{\Omega}, \hat{\mathfrak{A}}, \hat{P},\widehat W,\hat{\BB})$ one gets that $\hat x_t=\psi(t,(\widehat W_\cdot)_{[0,t]},(\hat\BB_\cdot)_{[0,t]}\big)$, $t\in [0,T]$, is a $(\widehat\F_t)_{t\in[0,T]}$-adapted solution to SDE \eqref{3SFDE}, where $\widehat\F_t:=\sigma((\widehat W(s),\hat{\BB}(s)), s\leq t)$.  But this means exactly the existence of a strong solution to SDE \eqref{3SFDE}.
\end{rem}

Recall that the approximative solutions $x^n$ are $(\F_t)_{t\in[0,T]}$-adapted and therefore, for every $t\in[0,T]$, the segment $x^n_t$ is $\F_t$-measurable. Since $\F_t=\sigma((W(s),\BB(s)),s\leq t)$, there exists a progressively measurable functional $\psi_n$ such that

\begin{align}\label{3xnpsi}
 x^n_t=\psi_n\big(t,(W_\cdot)_{[0,t]},(\BB_\cdot)_{[0,t]}\big).
\end{align}

Now we \textit{define} $\tilde x^n_t$ by replacing $W_\cdot$ by $\widetilde W_\cdot$ in this equation:

\begin{align}\label{3xntildepsi}
 \tilde x^n_t=\psi_n\big(t,(\widetilde W_\cdot)_{[0,t]},(\BB_\cdot)_{[0,t]}\big).
\end{align}

\begin{rem}
 Note that $\tilde x^n_t$ is \textit{by construction} measurable w.r.t. $\widetilde\F_t=\sigma((\widetilde W(s),\BB(s)),s\leq t)$
\end{rem}

We are going to prove that $\tilde x^n(t)$ converges strongly in $L^2(\Omega\times\Omega',\widetilde\F_t,\widetilde P)$ to the weak solution $\tilde x(t)$ that we constructed before, and therefore that $\tilde x(t)$ indeed is $\widetilde\F_t$-measurable. This proves that $\tilde x$ is a strong solution to the equation on the probability space $(\Omega\times\Omega',\widetilde\F_t,\widetilde P)$. In order to do so, we need the following lemmata.

\begin{lem}\label{3strongConvergenceXntilde}
 For every $t\in[0,T]$, $\tilde x^n(t)$ converges (strongly) in $L^2(\Omega\times\Omega',\widetilde\F_t,\widetilde P)$.
\end{lem}

\begin{proof}
 Since $L^2(\Omega\times\Omega',\widetilde\F_t,\widetilde P)$ is a complete space, it suffices to show that $(\tilde x^n(t))_{n\geq1}$ is a Cauchy sequence. To show that, let $n,m\in\NN$ and recall that $(\widetilde W_\cdot,\BB_\cdot)$ has under $\widetilde P$ the same law as $(W_\cdot,\BB_\cdot)$ has under $P\otimes P'$. Then,
 \begin{align*}
  E_{\widetilde P}\big[\big(\tilde x^n(t)-\tilde x^m(t)\big)^2\big]&=E_{\widetilde P}\big[\big(\psi_n\big(t,(\widetilde W_\cdot)_{[0,t]},(\BB_\cdot)_{[0,t]}\big)(0)-\psi_m\big(t,(\widetilde W_\cdot)_{[0,t]},(\BB_\cdot)_{[0,t]}\big)(0)\big)^2\big]\\
  &=\bar E\big[\big(\psi_n\big(t,(W_\cdot)_{[0,t]},(\BB_\cdot)_{[0,t]}\big)(0)-\psi_m\big(t,(W_\cdot)_{[0,t]},(\BB_\cdot)_{[0,t]}\big)(0)\big)^2\big]\\
  &=\bar E\big[\big(x^n(t)-x^m(t)\big)^2\big].%\\
  %&\rightarrow0,\quad\text{as }m,n\rightarrow\infty.
 \end{align*}
 We know from Theorem \ref{3StrongConvergenceL2} that $(x^n(t))_{n\geq1}$ converges strongly in $L^2(\Omega\times\Omega',\F_t,P\otimes P')$ and is therefore a Cauchy sequence. This implies that
 \begin{align*}
  E_{\widetilde P}\big[\big(\tilde x^n(t)-\tilde x^m(t)\big)^2\big]=\bar E\big[\big(x^n(t)-x^m(t)\big)^2\big]\rightarrow0,\quad\text{as }m,n\rightarrow\infty.
 \end{align*}
\end{proof}

\begin{lem}
 $\tilde x^n$ is a solution to the SFDE
 \begin{align}\label{3ApproxSFDE2}
  \tilde x^n(t)=\eta(0)+\int_0^tb^n(\tilde x_s+\eps\BB(s))ds+\widetilde W(t).
 \end{align}
\end{lem}

\begin{proof}
 Note that $\tilde x^n(t)=\tilde x^n_t(0)=\psi_n\big(t,(\widetilde W_\cdot)_{[0,t]},(\BB_\cdot)_{[0,t]}\big)(0)$ and that the process $(\widetilde W_\cdot,\BB_\cdot)$ has under $\widetilde P$ the same law as the process $(W_\cdot,\BB_\cdot)$ has under $P\otimes P'$. With that in mind, we have
 \begin{align*}
  &E_{\widetilde P}\Big[\Big(\tilde x^n(t)-\eta(0)-\int_0^tb^n(\tilde x_s+\eps\BB(s))ds-\widetilde W(t)\Big)^2\Big]\\
  &=E_{\widetilde P}\Big[\Big(\psi_n\big(t,(\widetilde W_\cdot)_{[0,t]},(\BB_\cdot)_{[0,t]}\big)(0)\\
  &\quad-\eta(0)-\int_0^tb^n\big(\psi_n\big(s,(\widetilde W_\cdot)_{[0,s]},(\BB_\cdot)_{[0,s]}\big)(0)+\eps\BB(s)\big)ds-\widetilde W(t)\Big)^2\Big]\\
  &=\bar E\Big[\Big(\psi_n\big(t,(W_\cdot)_{[0,t]},(\BB_\cdot)_{[0,t]}\big)(0)\\
  &\quad-\eta(0)-\int_0^tb^n\big(\psi_n\big(s,(W_\cdot)_{[0,s]},(\BB_\cdot)_{[0,s]}\big)(0)+\eps\BB(s)\big)ds-W(t)\Big)^2\Big]\\
  &=\bar E\Big[\Big(x^n(t)-\eta(0)-\int_0^tb^n(x_s+\eps\BB(s))ds-W(t)\Big)^2\Big]\\
  &=0,
 \end{align*}
 since $x^n$ solves \eqref{3ApproxSFDE}.
\end{proof}

\begin{lem}\label{3weakConvergencefXn}
 For each $t\in[0,T]$ and every bounded, continuous function $f:\R\rightarrow\R$,
 \begin{align*}
  f(\tilde x^n(t))\stackrel{n\rightarrow\infty}{\longrightarrow} E_{\widetilde P}[f(\tilde x(t))|\widetilde\F_t],\quad\text{weakly in }L^2(\Omega\times\Omega',\widetilde\F_t,\widetilde P).
 \end{align*}
\end{lem}

\begin{proof}
 In order to prove this lemma, we define the Wiener transform on the space $L^2(\Omega\times\Omega',\widetilde\F_t,\widetilde P)$ similarly to the Wiener transform we have defined on $L^2(\Omega\times\Omega',\F_t,P\otimes P')$, namely, for $X\in L^2(\Omega\times\Omega',\widetilde\F_t,\widetilde P)$ and $(\phi,\alpha)$ as in Definition \ref{3WienerTrafo},
 \begin{align*}
  &\Wi_{(\Omega\times\Omega',\widetilde\F_t,\widetilde P)}(X)(\phi,\alpha)\\
  &:=E_{\widetilde P}\Big[X\exp\Big(\int_0^t\phi(s)d\widetilde W(s)+\sum_{i\geq1}\sum_{j=1}^{l_i}\alpha_{i,j}w_i(B^{H_i}(t_j)-B^{H_i}(t_{j-1}))\Big)\Big].
 \end{align*}
 It has the same properties as the Wiener transform on $L^2(\Omega\times\Omega',\F_t,P\otimes P')$. Furthermore, we define the process
 \begin{align*}
  \widehat W^n(t):=\widetilde W(t)+\int_0^tb^n(\tilde x^n_s+\eps\BB(s))ds,
 \end{align*}
 which is a Brownian motion under the measure $\widehat P^n$ given by the Radon-Nikodym derivative
 \begin{align*}
  \frac{d\widehat P^n}{d\widetilde P}\Big|_{\widetilde\F_t}=\mathcal{E}\Big(-\int_0^\cdot b^n(\tilde x^n_s+\eps\BB(s))d\widetilde W(s)\Big)_t.
 \end{align*}
 We can rewrite the SFDE for $\tilde x^n$ as
 \begin{align*}
  \tilde x^n(t)=\eta(0)+\widehat W^n(t),\quad\tilde x^n_0=\eta,
 \end{align*}
 and get $\<\tilde x^n_t,e_i\>=F_i(\widehat W^n_\cdot)_t$. This enables us to write the Radon-Nikodym derivative of $\widetilde P$ w.r.t. $\widehat P^n$ by
 \begin{align*}
  \frac{d\widetilde P}{d\widehat P^n}\Big|_{\widetilde\F_t}=\mathcal{E}\Big(\int_0^\cdot b^n(F_1(\widehat W^n_\cdot)_s+\eps w_1B^{H_1}(s),\dots)d\widehat W^n(s)\Big)_t.
 \end{align*}
 We therefore have for every bounded, continuous $f:\R\rightarrow\R$:
 \begin{align*}
  &\Wi_{(\Omega\times\Omega',\widetilde\F_t,\widetilde P)}\big(f(\tilde x^n(t))\big)(\phi,\alpha)\\
  &=E_{\widetilde P}\Big[f(\tilde x^n(t))\exp\Big(\int_0^t\phi(s)d\widetilde W(s)+\sum_{i\geq1}\sum_{j=1}^{l_i}\alpha_{i,j}w_i(B^{H_i}(t_j)-B^{H_i}(t_{j-1}))\Big)\Big]\\
  &=E_{\widetilde P}\Big[f(\eta(0)+\widehat W^n(t))\exp\Big(\int_0^t\phi(s)d\widehat W^n(s)-\int_0^t\phi(s)b^n(F_1(\widehat W^n_\cdot)_s+\eps w_1B^{H_1}(s),\dots)ds\\
  &\quad+\sum_{i\geq1}\sum_{j=1}^{l_i}\alpha_{i,j}w_i(B^{H_i}(t_j)-B^{H_i}(t_{j-1}))\Big)\Big]\\
%   &=E_{\widehat P^n}\Big[\mathcal{E}\Big(\int_0^\cdot b^n(F_1(\widehat W^n_\cdot)_s+\eps w_1B^{H_1}(s),\dots)d\widehat W^n(s)\Big)_tf(\eta(0)+\widehat W^n(t))\\
%   &\quad\cdot\exp\Big(\int_0^t\phi(s)d\widehat W^n(s)-\int_0^t\phi(s)b^n(F_1(\widehat W^n_\cdot)_s+\eps w_1B^{H_1}(s),\dots)ds\\
%   &\quad+\sum_{i\geq1}\sum_{j=1}^{l_i}\alpha_{i,j}w_i(B^{H_i}(t_j)-B^{H_i}(t_{j-1}))\Big)\Big]\\
%   &=\bar E\Big[\mathcal{E}\Big(\int_0^\cdot b^n(F_1(W_\cdot)_s+\eps w_1B^{H_1}(s),\dots)dW(s)\Big)_tf(\eta(0)+W(t))\\
%   &\quad\cdot\exp\Big(\int_0^t\phi(s)dW(s)-\int_0^t\phi(s)b^n(F_1(W_\cdot)_s+\eps w_1B^{H_1}(s),\dots)ds\\
%   &\quad+\sum_{i\geq1}\sum_{j=1}^{l_i}\alpha_{i,j}w_i(B^{H_i}(t_j)-B^{H_i}(t_{j-1}))\Big)\Big],
 \end{align*}
 
 \begin{align*}
%   &\Wi_{(\Omega\times\Omega',\widetilde\F_t,\widetilde P)}\big(f(\tilde x^n(t))\big)(\phi,\alpha)\\
%   &=E_{\widetilde P}\Big[f(\tilde x^n(t))\exp\Big(\int_0^t\phi(s)d\widetilde W(s)+\sum_{i\geq1}\sum_{j=1}^{l_i}\alpha_{i,j}w_i(B^{H_i}(t_j)-B^{H_i}(t_{j-1}))\Big)\Big]\\
%   &=E_{\widetilde P}\Big[f(\eta(0)+\widehat W^n(t))\exp\Big(\int_0^t\phi(s)d\widehat W^n(s)-\int_0^t\phi(s)b^n(F_1(\widehat W^n_\cdot)_s+\eps w_1B^{H_1}(s),\dots)ds\\
%   &\quad+\sum_{i\geq1}\sum_{j=1}^{l_i}\alpha_{i,j}w_i(B^{H_i}(t_j)-B^{H_i}(t_{j-1}))\Big)\Big]\\
  &\hspace{-2cm}=E_{\widehat P^n}\Big[\mathcal{E}\Big(\int_0^\cdot b^n(F_1(\widehat W^n_\cdot)_s+\eps w_1B^{H_1}(s),\dots)d\widehat W^n(s)\Big)_tf(\eta(0)+\widehat W^n(t))\\
  &\hspace{-2cm}\quad\cdot\exp\Big(\int_0^t\phi(s)d\widehat W^n(s)-\int_0^t\phi(s)b^n(F_1(\widehat W^n_\cdot)_s+\eps w_1B^{H_1}(s),\dots)ds\\
  &\hspace{-2cm}\quad+\sum_{i\geq1}\sum_{j=1}^{l_i}\alpha_{i,j}w_i(B^{H_i}(t_j)-B^{H_i}(t_{j-1}))\Big)\Big]\\
  &\hspace{-2cm}=\bar E\Big[\mathcal{E}\Big(\int_0^\cdot b^n(F_1(W_\cdot)_s+\eps w_1B^{H_1}(s),\dots)dW(s)\Big)_tf(\eta(0)+W(t))\\
  &\hspace{-2cm}\quad\cdot\exp\Big(\int_0^t\phi(s)dW(s)-\int_0^t\phi(s)b^n(F_1(W_\cdot)_s+\eps w_1B^{H_1}(s),\dots)ds\\
  &\hspace{-2cm}\quad+\sum_{i\geq1}\sum_{j=1}^{l_i}\alpha_{i,j}w_i(B^{H_i}(t_j)-B^{H_i}(t_{j-1}))\Big)\Big],
 \end{align*}
 
 since $(\widehat W^n,\BB)$ has the same distribution under $\widehat P^n$ as $(W,\BB)$ has under $P$. By dominated convergence, we now have
 \begin{align*}
  &\Wi_{(\Omega\times\Omega',\widetilde\F_t,\widetilde P)}\big(f(\tilde x^n(t))\big)(\phi,\alpha)\\
  &=\bar E\Big[\mathcal{E}\Big(\int_0^\cdot b^n(F_1(W_\cdot)_s+\eps w_1B^{H_1}(s),\dots)dW(s)\Big)_tf(\eta(0)+W(t))\\
  &\quad\cdot\exp\Big(\int_0^t\phi(s)dW(s)-\int_0^t\phi(s)b^n(F_1(W_\cdot)_s+\eps w_1B^{H_1}(s),\dots)ds\\
  &\quad+\sum_{i\geq1}\sum_{j=1}^{l_i}\alpha_{i,j}w_i(B^{H_i}(t_j)-B^{H_i}(t_{j-1}))\Big)\Big]\\
  &\rightarrow\bar E\Big[\mathcal{E}\Big(\int_0^\cdot b(F_1(W_\cdot)_s+\eps w_1B^{H_1}(s),\dots)dW(s)\Big)_tf(\eta(0)+W(t))\\
  &\quad\cdot\exp\Big(\int_0^t\phi(s)dW(s)-\int_0^t\phi(s)b(F_1(W_\cdot)_s+\eps w_1B^{H_1}(s),\dots)ds\\
  &\quad+\sum_{i\geq1}\sum_{j=1}^{l_i}\alpha_{i,j}w_i(B^{H_i}(t_j)-B^{H_i}(t_{j-1}))\Big)\Big]\\
  &=\bar E\bigg[\Big(\frac{d\widetilde P}{dP\otimes P'}\Big|_{\widetilde\F_t}\Big)f(\tilde x(t))\exp\Big(\int_0^t\phi(s)d\widetilde W(s)+\sum_{i\geq1}\sum_{j=1}^{l_i}\alpha_{i,j}w_i(B^{H_i}(t_j)-B^{H_i}(t_{j-1}))\Big)\bigg].
 \end{align*}
 So far, we do not know whether $f(\tilde x(t))$ is $\widetilde\F_t$-measurable, therefore, this is \textit{not} necessarily the Wiener transform of $f(\tilde x(t))$ (it might even not exist). But, we can apply the tower property and get
 \begin{align*}
  &\Wi_{(\Omega\times\Omega',\widetilde\F_t,\widetilde P)}\big(f(\tilde x^n(t))\big)(\phi,\alpha)\\
  &\rightarrow E_{\widetilde P}\bigg[E\Big[f(\tilde x(t))\exp\Big(\int_0^t\phi(s)d\widetilde W(s)+\sum_{i\geq1}\sum_{j=1}^{l_i}\alpha_{i,j}w_i(B^{H_i}(t_j)-B^{H_i}(t_{j-1}))\Big)\Big|\widetilde\F_t\Big]\bigg]\\
  &=E_{\widetilde P}\Big[E[f(\tilde x(t))|\widetilde\F_t]\exp\Big(\int_0^t\phi(s)d\widetilde W(s)+\sum_{i\geq1}\sum_{j=1}^{l_i}\alpha_{i,j}w_i(B^{H_i}(t_j)-B^{H_i}(t_{j-1}))\Big)\Big]\\
  &=\Wi_{(\Omega\times\Omega',\widetilde\F_t,\widetilde P)}\big(E[f(\tilde x(t))|\widetilde\F_t]\big)(\phi,\alpha),
 \end{align*}
 which proves the statement.
\end{proof}

We are now able to prove the main theorem of this paper.

\begin{thm}\label{3mainTheorem}
 The weak solution $\tilde x$ w.r.t. the probability measure $\widetilde P$ is $(\widetilde\F_t)_{t\in[0,T]}$-adapted, i.e. it is indeed a strong solution.
\end{thm}

\begin{proof}
 By Lemma \ref{3strongConvergenceXntilde}, we know that $\tilde x^n$ converges strongly in $L^2(\Omega\times\Omega',\widetilde\F_t,\widetilde P)$, i.e. it exists $\hat x(t)\in L^2(\Omega\times\Omega',\widetilde\F_t,\widetilde P)$ s.t.
 \begin{align*}
  \tilde x^n(t)\stackrel{n\rightarrow\infty}{\longrightarrow}\hat x(t),\quad\text{strongly in }L^2(\Omega\times\Omega',\widetilde\F_t,\widetilde P).
 \end{align*}
 This implies that $\tilde x^n(t)$ converges against $\hat x(t)$ in probability and therefore, by the continuous mapping theorem, we have for every bounded continuous $f:\R\rightarrow\R$:
 \begin{align*}
  f(\tilde x^n(t))\stackrel{n\rightarrow\infty}{\longrightarrow}f(\hat x(t)),\quad\text{in probability.}
 \end{align*}
 It follows from the boundedness of $f$ and Vitali's Theorem that
 \begin{align*}
  f(\tilde x^n(t))\stackrel{n\rightarrow\infty}{\longrightarrow}f(\hat x(t)),\quad\text{strongly in }L^2(\Omega\times\Omega',\widetilde\F_t,\widetilde P),
 \end{align*}
 and since strong convergence implies weak convergence, we also have that $f(\tilde x^n(t))$ converges to $f(\hat x(t))$ weakly in $L^2(\Omega\times\Omega',\widetilde\F_t,\widetilde P)$. On the other hand, we know from Lemma \ref{3weakConvergencefXn} that
 \begin{align*}
  f(\tilde x^n(t))\stackrel{n\rightarrow\infty}{\longrightarrow}E_{\widetilde P}\big[f(\tilde x(t))\big|\widetilde\F_t\big],\quad\text{weakly in }L^2(\Omega\times\Omega',\widetilde\F_t,\widetilde P).
 \end{align*}
 Therefore, by the uniqueness of the weak limit, we have for every bounded, continuous $f$
 \begin{align}\label{3fXhatEqualsCondfXtilde}
  f(\hat x(t))=E_{\widetilde P}\big[f(\tilde x(t))\big|\widetilde\F_t\big].
 \end{align}
 Now consider the sequence $(f_k)_{k\in\NN}$ of bounded, continuous functions given by
 \begin{align*}
  f_k(x)=x\ind_{[-k,k]}(x).
 \end{align*}
 Then, for every $x\in\R$, we have $f_k(x)\rightarrow x$ pointwise, as $k\rightarrow\infty$, and $|f_k(\tilde x(t))|\leq|\tilde x(t)|$ is square integrable. Therefore, by the dominated convergence theorem for conditional expectations, we have
 \begin{align*}
  \hat x(t)=\lim_{k\rightarrow\infty}f_k(\hat x(t))=\lim_{k\rightarrow\infty}E_{\widetilde P}\big[f_k(\tilde x(t))\big|\widetilde\F_t\big]=E_{\widetilde P}\big[\lim_{k\rightarrow\infty}f_k(\tilde x(t))\big|\widetilde\F_t\big]=E_{\widetilde P}\big[\tilde x(t)\big|\widetilde\F_t\big].
 \end{align*}
 Plugging this into \eqref{3fXhatEqualsCondfXtilde}, we have for every bounded, continuous $f:\R\rightarrow\R$:
 \begin{align*}
  f\big(E_{\widetilde P}\big[\tilde x(t)\big|\widetilde\F_t\big]\big)=E_{\widetilde P}\big[f(\tilde x(t))\big|\widetilde\F_t\big].
 \end{align*}
 It follows from Lemma \ref{3measurabilityLemma} that $\tilde x(t)$ is $\widetilde\F_t$-measurable. It follows from Remark \ref{3VI_stochbasisrmk} that the constructed solution is indeed a strong solution.
\end{proof}

\begin{rem}\label{3StrongSolutionOriginalSetup}
 In particular, by back-lifting to the original stochastic Basis $(\Omega\times\Omega', \F\otimes\F', P\otimes P',W,\BB)$ via
 \begin{align*}
  x^n_t&=\psi_n(t,(W_\cdot)_{[0,t]},(\BB_\cdot)_{[0,t]}\big)\\
  x_t&=\psi(t,(W_\cdot)_{[0,t]},(\BB_\cdot)_{[0,t]}\big),
 \end{align*}
 for $t\in [0,T]$, one observes that the process $x$ from Theorem \ref{3StrongConvergenceL2} actually is the strong solution to SFDE \eqref{3SFDE} and for $x^n$ denoting the strong solutions of the approximative SFDE, $x^n(t)$ converges strongly in $L^2(\Omega\times\Omega',\F_t,P\otimes P')$ to $x(t)$.
\end{rem}

\subsection{The infinite dimensional case}

We are finally in the position to treat the infinite dimensional case. Let $b:M_2\rightarrow\R$ be given by
\begin{align*}
 b(x)=\sum_{i=1}^\infty b_i(\<x,e_i\>),
\end{align*}
which we identify with the function $\tilde b:\ell^2\rightarrow\R$
\begin{align*}
 \tilde b(z)=\sum_{i=1}^\infty b_i(z_i)
\end{align*}
via $b(x)=\tilde b\big((\<x,e_i\>)_{i\geq1}\big)$. From now on we write $b$ instead of $\tilde b$.\\
\par
Now, for every $d\in\NN$, we define the approximation $b^d:\R^d\rightarrow\R$

\begin{align}\label{3FiniteDimensionalApprox}
 b^d(z):=\sum_{i=1}^d b^d_i(z_i)=\sum_{i=1}^d b_i(z_i)\ind_{[-d,d]}(z_i).
\end{align}

Then $b^d$ converges pointwise against $b$, as the next lemma shows.

\begin{lem}\label{3PointwiseConvergenceFiniteDimensionalApprox}
 Let $z=(z_i)_{i\geq1}\in\ell^2$. Then
 \begin{align}
  |b(z)-b^d(z_1,\dots,z_d)|\rightarrow0.
 \end{align}
\end{lem}

\begin{proof}
 Fix $\eps>0$. Recall that $\big(\|b_i\|_\infty\big)_{i\geq1}\in\ell^1$. We can therefore choose a $d_0=d_0(\eps)\in\NN$ such that
 \begin{align*}
  \sum_{i=d_0+1}^\infty\|b_i\|_\infty<\eps.
 \end{align*}
 Then, for $d=d(z_1,\dots,z_{d_0})>d_0$ large enough, we have $\ind_{(-\infty,-d)\cup(d,\infty)}(z_i)=0$ for $i=1,\dots,d_0$ and thus
 \begin{align*}
  &|b(z)-b^d(z_1,\dots,z_d)|\\
  &\quad=\Big|\sum_{i=1}^\infty b_i(z_i)-\sum_{i=1}^d b_i(z_i)\ind_{[-d,d]}(z_i)\Big|\\
  &\quad=\Big|\sum_{i=1}^d b_i(z_i)\ind_{(-\infty,-d)\cup(d,\infty)}(z_i)+\sum_{i=d+1}^\infty b_i(z_i)\Big|\\
  &\quad=\Big|\sum_{i=1}^{d_0} b_i(z_i)\ind_{(-\infty,-d)\cup(d,\infty)}(z_i)+\sum_{i=d_0+1}^d b_i(z_i)\ind_{(-\infty,-d)\cup(d,\infty)}(z_i)+\sum_{i=d+1}^\infty b_i(z_i)\Big|\\
  &\quad\leq\sum_{i=1}^{d_0} |b_i(z_i)|\cdot0+\sum_{i=d_0+1}^d |b_i(z_i)|\ind_{(-\infty,-d)\cup(d,\infty)}(z_i)+\sum_{i=d+1}^\infty |b_i(z_i)|\\
  &\quad\leq0+\sum_{i=d_0+1}^\infty \|b_i\|_\infty<\eps.
 \end{align*}
\end{proof}

Note that the functions $b^d_i$ are all bounded, measurable with bounded support, i.e. from the previous results, we know that the corresponding SFDE
\begin{align}\label{3ApproxSFDEinfiniteDim}
\begin{cases}
 x^d(t)&=\eta(0)+\int_0^tb^d(\<x^d_s,e_1\>+\eps w_1 B^{H_1}(s),\dots,\<x^d_s,e_d\>+\eps w_d B^{H_d}(s))ds+W(t)\\
 x^d_0&=\eta,
\end{cases}
\end{align}
has a unique solution. Furthermore, by defining $b^{d,n}_i:\R\rightarrow\R$ for $i=1,\dots,d$ and $n\in\NN$ by
\begin{align*}
 b^{d,n}_i(z)=b^d_i\ast\varphi_n(z),
\end{align*}
and $x^{d,n}$ the strong solution of the corresponding SFDE
\begin{align}\label{3ApproxAproxSFDEinfiniteDim}
\begin{cases}
 x^{d,n}(t)&=\eta(0)+\int_0^tb^{d,n}(\<x^{d,n}_s,e_1\>+\eps w_1 B^{H_1}(s),\dots,\<x^{d,n}_s,e_d\>+\eps w_d B^{H_d}(s))ds+W(t)\\
 x^{d,n}_0&=\eta,
\end{cases}
\end{align}
we get an approximation just as in the sections before. In particular, by Theorem \ref{3StrongConvergenceL2} and Remark \ref{3StrongSolutionOriginalSetup}, we know that, for $t\in[0,T]$, $x^{d,n}(t)\stackrel{n\rightarrow\infty}{\longrightarrow}x^d(t)$ strongly in $L^2(\Omega\times\Omega',\F_t,P\otimes P')$.\\
\par
Before we can prove our main result, we need to modify Assumption \ref{3StandingAssumptions} $(H)$ and $(A)$ a little bit:

\begin{ass}\label{3FinalAssumption}
 Suppose, it exists a $\delta_H\in(0,1)$ such that the following conditions are satisfied:
\begin{itemize}
 \item[(H')] The sequence of Hurst parameters $(H_k)_{k\geq1}$ corresponding to the definition of the pertubation $\BB$ (see Lemma \ref{3CylfBM}) satisfies
 \begin{align}\label{3Hcondition2}
  H_k<\frac{1-\delta_H}{3}\quad\text{for all }k\geq1.
 \end{align}
 \item[(A')] It holds
 \begin{align}\label{3Acondition2}
  \sum_{j=1}^\infty A_j<1,
 \end{align}
 where $A_j$, $j=1,\dots,d$ are defined by
 \begin{align}\label{3ADefinition2}
  A_j=\frac{48\sqrt{2}(1+r)\Gamma(\delta_H)}{\sqrt{\pi}}C_{j}^{-\frac{3}{2}}|w_{j}|^{-3}\|b_{j}\|_{L^1},
 \end{align}
%  The functions $A_j:[0,T]\rightarrow\R$, $j\geq1$ defined by
%  \begin{align}\label{3ADefinition2}
%   A_j(t)=\frac{96\sqrt{2}(1+r)\Gamma(\delta_H)\left(t\vee t^{\delta_H}\right)}{\sqrt{\pi}}C_{j}^{-\frac{3}{2}}|w_{j}|^{-3}\|b_{j}\|_{L^1},
%  \end{align}
%  satisfies, for every $t\in[0,T]$,
%  \begin{align}\label{3Acondition}
%   \sum_{j=1}^\infty A_j(t)<t\vee t^{\delta_H}.
%  \end{align}
\end{itemize}
\end{ass}

The next lemma generalizes Lemma \ref{3VI_relcomp}.

\begin{lem}\label{3VI_relcomp2}
Fix $t\in[0,T]$. Let Assumptions \ref{3StandingAssumptions} $(T)$ and \ref{3FinalAssumption} $(H')$ and $(A')$ be satisfied. Then, for almost every $\omega'\in\Omega'$, there exists a subsequence $d_k(\omega')$ such that

\begin{enumerate}
 \item we have
 \begin{align}\label{3CompactnessInfiniteDim1}
 \sup_{k\geq 1}E\bigg[\int_0^t|\DM_{\theta} x^{d_k(\omega')}(t,\cdot,\omega')|^2d\theta\bigg] <\infty.
 \end{align}
 \item there exists a $\beta\in (0,1/2)$ such that
 \begin{align}\label{3CompactnessInfiniteDim2}
  \sup_{k\geq 1}\int_0^t \int_0^t \frac{E[|\DM_\theta x^{d_k(\omega')}(t,\cdot,\omega') - \DM_{\theta'} x^{d_k(\omega')}(t,\cdot,\omega')|^2]}{|\theta' - \theta|^{1+2\beta}} d\theta' d\theta <\infty.
 \end{align}
\end{enumerate}
\end{lem}

\begin{proof}
 As in the proof of Lemma \ref{3VI_relcomp}, it suffices to show that
 \begin{align}\label{3CompactnessFinal1}
 \sup_{d\geq 1}E_{P\otimes P'}\bigg[\int_0^t|\DM_{\theta} x^d(t)|^2d\theta\bigg] <\infty,
 \end{align}
 and
 \begin{align}\label{3CompactnessFinal2}
  \sup_{d\geq 1}\int_0^t \int_0^t \frac{E_{P\otimes P'}[|\DM_\theta x^{d}(t) - \DM_{\theta'} x^{d}(t)|^2]}{|\theta' - \theta|^{1+2\beta}} d\theta' d\theta <\infty.
 \end{align}
 First, we show \eqref{3CompactnessFinal1}. For this, note that for $t\in[0,T]$, $x^{d,n}(t)\stackrel{n\rightarrow\infty}{\longrightarrow}x^d(t)$ strongly in $L^2(\Omega\times\Omega',\F_t,P\otimes P')$. Moreover, as SFDE \eqref{3ApproxAproxSFDEinfiniteDim} has Lipschitz coefficients (recall that $b^{d,n}_i\in C^\infty_c(\R)$ by construction), $x^{d,n}(t)\in\mathbb{D}^{1,2}$. Finally, we have, by \eqref{3FinalEstimationStep1} that $$\sup_{n\geq1}E_{P\otimes P'}\Big[\int_0^T|\DM_\theta x^{d,n}(t)|^2d\theta\Big]<\infty.$$ Therefore, by Lemma \ref{3NualartLemma}, we have for all $t\in[0,T]$,
 \begin{align*}
  x^{d}(t)\in\mathbb{D}^{1,2}\quad\text{and}\quad \DM_\cdot x^{d,n}(t)\stackrel{n\rightarrow\infty}{\longrightarrow}\DM_\cdot x^{d}(t)\quad\text{weakly in }L^2(\Omega\times\Omega',L^2([0,T])).
 \end{align*}
 Therefore, by Fatou's Lemma and the weak lower semi-continuity of norms, we have
 \begin{align*}
  \sup_{d\geq 1}E'\bigg[E\Big[\int_0^t|\DM_{\theta} x^d(t)|^2d\theta\Big]\bigg]=\sup_{d\geq 1}\liminf_{n\rightarrow\infty}E'\bigg[E\Big[\int_0^t|\DM_{\theta} x^{d,n}(t)|^2d\theta\Big]\bigg].
 \end{align*}
 By \eqref{3UsefulEstimationForLater1}, we have 
 \begin{align*}
 E'\bigg[E\Big[\int_0^t|\DM_{\theta} x^{d,n}(t)|^2d\theta\Big]\bigg]\leq e^{\frac{1}{2}M^2T}\bigg(\sum_{m=1}^\infty\Big((t-\theta)^{\delta_H}\sum_{j=1}^dA_{j}\Big)^m\bigg)^2,
 \end{align*}
 which does not depend on $n$ any more. Now, Fubini-Tonelli's theorem yields
 \begin{align*}
  \sup_{d\geq 1}E'\bigg[E\Big[\int_0^t|\DM_{\theta} x^d(t)|^2d\theta\Big]\bigg]&\leq \sup_{d\geq 1}\liminf_{n\rightarrow\infty}E'\bigg[E\Big[\int_0^t|\DM_{\theta} x^{d,n}(t)|^2d\theta\Big]\bigg]\\
  &\leq e^{\frac{1}{2}M^2T}\lim_{d\geq 1}\bigg(\sum_{m=1}^\infty\Big((t-\theta)^{\delta_H}\sum_{j=1}^dA_{j}\Big)^m\bigg)^2\\
  &=e^{\frac{1}{2}M^2T}\bigg(\sum_{m=1}^\infty\Big((t-\theta)^{\delta_H}\sum_{j=1}^\infty A_{j}\Big)^m\bigg)^2\\
  &<\infty.
 \end{align*}
 This proves \eqref{3CompactnessFinal1}.\\
 \par
 In order to prove \eqref{3CompactnessFinal2}, let $\rho\in L^2(\R^2,\R)$, $\rho\geq0$. Then, by Lemma \eqref{3NualartLemma}, H\"{o}lder's inequality and monotone convergence
 \begin{align*}
  &\int_0^t \int_0^t \rho(\theta,\theta')E_{P\otimes P'}[|\DM_\theta x^{d}(t) - \DM_{\theta'} x^{d}(t)|^2] d\theta' d\theta\\
  &=\lim_{n,m\rightarrow\infty}\int_0^t \int_0^t E_{P\otimes P'}[\rho(\theta,\theta')(\DM_\theta x^{d,n}(t) - \DM_{\theta'} x^{d,n}(t))(\DM_\theta x^{d,m}(t) - \DM_{\theta'} x^{d,m}(t))] d\theta' d\theta\\
  &\leq \lim_{n\rightarrow\infty}\bigg(\int_0^t \int_0^t \rho(\theta,\theta')E_{P\otimes P'}[|\DM_\theta x^{d,n}(t) - \DM_{\theta'} x^{d,n}(t))|^2] d\theta' d\theta\bigg)^{\frac{1}{2}}\\
  &\quad\cdot\lim_{m\rightarrow\infty}\bigg(\int_0^t \int_0^t \rho(\theta,\theta')E_{P\otimes P'}[|\DM_\theta x^{d,m}(t) - \DM_{\theta'} x^{d,m}(t))|^2] d\theta' d\theta\bigg)^{\frac{1}{2}}\\
  &=\lim_{n\rightarrow\infty}\int_0^t \int_0^t \rho(\theta,\theta')E_{P\otimes P'}[|\DM_\theta x^{d,n}(t) - \DM_{\theta'} x^{d,n}(t))|^2] d\theta' d\theta\\
  &=\int_0^t \int_0^t \rho(\theta,\theta')\sup_{n\geq1}E_{P\otimes P'}[|\DM_\theta x^{d,n}(t) - \DM_{\theta'} x^{d,n}(t))|^2] d\theta' d\theta.
 \end{align*}
 Since this holds for every $\rho\in L^2(\R^2,\R)$, $\rho\geq0$, we obtain that for Lebesgue-a.e. $(\theta,\theta')\in[0,t]^2$,
 \begin{align*}
  E_{P\otimes P'}[|\DM_\theta x^{d}(t) - \DM_{\theta'} x^{d}(t)|^2]\leq \sup_{n\geq1}E_{P\otimes P'}[|\DM_\theta x^{d,n}(t) - \DM_{\theta'} x^{d,n}(t))|^2].
 \end{align*}
 On the other hand, we know, by \eqref{3I1Estimate}, \eqref{3I2Estimate} and \eqref{3I3Estimate} and Assumption \ref{3FinalAssumption} that
 \begin{align*}
  \sup_{n\geq1}E_{P\otimes P'}[|\DM_\theta x^{d,n}(t) - \DM_{\theta'} x^{d,n}(t))|^2]&\leq 3e^{\frac{1}{2}M^2T}\delta_T^{-2\delta_H}|\theta'-\theta|^{2\delta_H}\\
  &\quad+3e^{\frac{1}{2}M^2T}\delta_T^{-2\delta_H}|t-\theta'|^{2\delta_H}\frac{|\theta-\theta'|}{(1+r)^2}\\
  &\quad+3e^{\frac{1}{2}M^2T}\delta_T^{-4\delta_H}|t-\theta'|^{2\delta_H}|\theta-\theta'|^{2\delta_H}\\
  &=C(\theta,\theta',t)|\theta-\theta'|^{2(\delta_H\wedge\frac{1}{2})},
 \end{align*}
 
 for some bounded function $C:[0,T]\times[0,T]\times[0,T]\rightarrow\R$. Note that the right handside of this equation \textit{does not depend on $d$}. Therefore,
 
 \begin{align*}
  \sup_{d\geq1}E_{P\otimes P'}[|\DM_\theta x^{d}(t) - \DM_{\theta'} x^{d}(t)|^2]\leq C(\theta,\theta',t)|\theta-\theta'|^{2(\delta_H\wedge\frac{1}{2})}.
 \end{align*}
 
 Choosing $\beta<\delta_H\wedge\frac{1}{2}$, we therefore have that
 
 \begin{align*}
  &\sup_{d\geq 1}\int_0^t \int_0^t \frac{E_{P\otimes P'}[|\DM_\theta x^{d}(t) - \DM_{\theta'} x^{d}(t)|^2]}{|\theta' - \theta|^{1+2\beta}} d\theta' d\theta\\
  &\quad\leq\int_0^t \int_0^t \frac{\sup_{d\geq 1}E_{P\otimes P'}[|\DM_\theta x^{d}(t) - \DM_{\theta'} x^{d}(t)|^2]}{|\theta' - \theta|^{1+2\beta}} d\theta' d\theta\\
  &\leq\int_0^t \int_0^t C(\theta,\theta',t)|\theta-\theta'|^{2(\delta_H\wedge\frac{1}{2}-\beta)-1} d\theta' d\theta<\infty,
 \end{align*}
 
 which proves \eqref{3CompactnessFinal2}.
\end{proof}

\begin{thm}
 Under assumptions \ref{3StandingAssumptions} $(T)$ and \ref{3FinalAssumption} $(H')$ and $(A')$, SFDE \ref{3SFDE} has a unique strong solution.
\end{thm}

\begin{proof}
 The proof follows the exact same lines as in the finite dimensional case (see sections before) just that instead of Lemma \ref{3VI_relcomp}, we apply Lemma \ref{3VI_relcomp2} in order to show relative compactness of $(x^{d_k(\omega')}(t))_{k\geq1}$ for almost every $\omega'\in\Omega'$ and every $t\in[0,T]$. Finally, strong uniqueness follows by using the representation of the Wiener transform of solutions as before.
\end{proof}

\appendix
 
\section{Technical results}

\begin{lem}\label{3measurabilityLemma}
 Let $(\Omega,\mathcal{A},P)$ be a probability space and $X\in L^2(\Omega,\mathcal{A},P)$. Furthermore, let $\F\subseteq\mathcal{A}$ be a sigma-algebra. If for all bounded, continuous $\phi:\R\rightarrow\R$,
 \begin{align*}
  \phi(E[X|\F])=E[\phi(X)|\F],
 \end{align*}
 then $X$ is $\F$-measurable.
\end{lem}

\begin{proof}
 We prove that $X=E[X|\F]$. To see that, we define, for every $n\in\NN$, the bounded, continuous function $\phi_n:\R\rightarrow\R$ by $\phi_n(x):=x^2\vee n$ and consider
 \begin{align*}
  E\big[\big(X-E[X|\F]\big)^2\big]%&=E\big[X^2\big]+E\big[\big(E[X|\F]\big)^2\big]-2E\big[XE[X|\F]\big]\\
  %&=E\big[X^2\big]+E\big[\big(E[X|\F]\big)^2\big]-2E\big[E\big[XE[X|\F]\big|\F\big]\big]\\
  &=E\big[X^2-\big(E[X|\F]\big)^2\big]\\
  &=E\Big[\lim_{n\rightarrow\infty}\Big(X^2\vee n-\Big(\big(E[X|\F]\big)^2\vee n\Big)\Big)\Big].%\\
  %&=E\big[\lim_{n\rightarrow\infty}\big(\phi_n(X)-\phi_n(E[X|\F])\big)\big]
 \end{align*}
 Since $X,E[X|\F]\in L^2(\Omega,\mathcal{A},P)$ and $\phi_n(x)\leq x^2$, it follows from dominated convergence and the tower property that
 \begin{align*}
  E\big[\big(X-E[X|\F]\big)^2\big]&=E\big[\lim_{n\rightarrow\infty}\big(\phi_n(X)-\phi_n(E[X|\F])\big)\big]\\
  &=\lim_{n\rightarrow\infty}E\big[\phi_n(X)-\phi_n(E[X|\F])\big]\\
  &=\lim_{n\rightarrow\infty}E\big[E\big[\phi_n(X)-\phi_n(E[X|\F])\big|\F\big]\big]\\
  &=\lim_{n\rightarrow\infty}E\big[E\big[\phi_n(X)\big|\F\big]-\phi_n(E[X|\F])\big]\\
  &=\lim_{n\rightarrow\infty}0=0.
 \end{align*}
\end{proof}

The next lemma is a simple adaptation of \cite[Lemma 1.2.3]{Nualart}

\begin{lem}\label{3NualartLemma}
 Let $(\Omega,\mathfrak{A},P)$ be a probability space, $0<T<\infty$ and let
 \begin{itemize}
  \item[(i)] $Y^n\stackrel{n\rightarrow\infty}{\longrightarrow}Y$ in $L^2(\Omega)$,
  \item[(ii)] $Y^n\in\mathbb{D}^{1,2}$ for all $n\in\NN$, and
  \item[(iii)] $\sup_{n\geq1}E\Big[\int_0^T|D_sY^n|^2ds\Big]<\infty$.
 \end{itemize}
 Then $Y\in\mathbb{D}^{1,2}$ and $D_\cdot Y^n\stackrel{n\rightarrow\infty}{\longrightarrow}D_\cdot Y$ weakly in $L^2(\Omega,H)$, where $H=L^2([0,T])$.
\end{lem}

\begin{proof}
 %\textcolor{red}{To Frank: Can you please polish the proof?}\\
 We first show that $Y\in\mathbb{D}^{1,2}$. For that, let $J_n(Y)$ denote the projection to the $n^{\text{th}}$ Wiener chaos. Then, what we need to show is that
 \begin{align*}
  \sum_{n=1}^\infty n\|J_n(Y)\|_2^2<\infty.
 \end{align*}
 Since the $n^{\text{th}}$ projection $J_n$ is continuous w.r.t. the $L^2$ topology, we have by $(i)$ and $(ii)$ and 
 \begin{align*}
  \sum_{n=1}^\infty n\|J_n(Y)\|_2^2=\sum_{n=1}^\infty n\lim_{m\rightarrow\infty}\|J_n(Y^m)\|_2^2.%\leq\liminf_{m\rightarrow\infty}\sum_{n=1}^\infty n\|J_n(Y^m)\|_2^2<\infty.
 \end{align*}
 Now Fatou's Lemma applied to the sum $\sum_{n=1}^\infty\dots$ and $(iii)$ yield
 \begin{align*}
  \sum_{n=1}^\infty n\|J_n(Y)\|_2^2&\leq\liminf_{m\rightarrow\infty}\sum_{n=1}^\infty n\|J_n(Y^m)\|_2^2=\liminf_{m\rightarrow\infty}E\Big[\int_0^T|D_sY^m|^2ds\Big]\\
  &\leq\sup_{m\geq1}E\Big[\int_0^T|D_sY^m|^2ds\Big]<\infty.
 \end{align*}
 Now that we know that $Y\in\mathbb{D}^{1,2}$, we can prove the convergence. For that, let $\delta(a)$ denote the Skorohod integral $\int_0^Ta(s)\delta(ds)$ and let $a\in\dom(\delta)$. By the duality formula and $(i)$, we have
 \begin{align*}
  E\Big[\int_0^TD_sY^na(s)ds\Big]=E\big[Y^n\delta(a)\big]\stackrel{n\rightarrow\infty}{\longrightarrow} E\big[Y\delta(a)\big]=E\Big[\int_0^TD_sYa(s)ds\Big].
 \end{align*}
 The weak convergence in $L^2(\Omega,H)$ follows now from the fact that $\dom(\delta)$ is dense in $L^2(\Omega\times[0,T])$.
\end{proof}

\begin{lem}\label{3SimplexIntegralLemma}
 Let $m\in\NN$, $s_{m+1=\theta}$ and, for $j=1,\dots,m$, $a_j\in(-1,\infty)$. Then
 \begin{align*}
  \int_{\Delta^m_{\theta,t}}\prod_{j=1}^m(s_j-s_{j+1})^{a_j}ds=\frac{\prod_{l=1}^m\Gamma(a_l+1)}{\Gamma(\sum_{l=1}^ma_l+m+1)}(t-\theta)^{\sum_{l=1}^ma_l+m}.
 \end{align*}
\end{lem}

\begin{proof}
 Before we start with the actual proof, we recall the well-known fact that for all $a,b>-1$
 \begin{align}\label{3IntegralHilfsFormel}
  \int_\theta^s(s-s_1)^a(s_1-\theta)^bds_1=\frac{\Gamma(a+1)\Gamma(b+1)}{\Gamma(a+b+2)}(s-\theta)^{a+b+1}
 \end{align}
 We will now proof the lemma by induction. For that, note that
 \begin{align*}
  \int_\theta^t(s_1-\theta)^{a_1}ds_1=\frac{1}{a_1+1}(t-\theta)^{a_1+1}=\frac{\Gamma(a_1+1)}{\Gamma(a_1+2)}(t-\theta)^{a_1+1},
 \end{align*}
 since for all $z>0$, $\Gamma(z+1)=z\Gamma(z)$. This proves the lemma for $m=1$. Suppose now, we have proven the result up to some fixed $m$. We need to show that it also holds for $m+1$. Therefore, we consider
 \begin{align*}
  &\int_{\Delta^{m+1}_{\theta,t}}\prod_{j=1}^{m+1}(s_j-s_{j+1})^{a_j}ds\\
  &=\int\limits_\theta^t\int\limits_\theta^{s_1}\dots\int\limits_\theta^{s_{m-1}}\int\limits_\theta^{s_m}(s_{m+1}-\theta)^{a_{m+1}}(s_m-s_{m+1})^{a_m}\prod_{j=1}^{m-1}(s_j-s_{j+1})^{a_j}ds_{m+1}ds_m\dots ds_1ds_1\\
  &=\int\limits_\theta^t\int\limits_\theta^{s_1}\dots\int\limits_\theta^{s_{m-1}}\prod_{j=1}^{m-1}(s_j-s_{j+1})^{a_j}\int\limits_\theta^{s_m}(s_{m+1}-\theta)^{a_{m+1}}(s_m-s_{m+1})^{a_m}ds_{m+1}ds_m\dots ds_1ds_1\\
  &=\int\limits_\theta^t\int\limits_\theta^{s_1}\dots\int\limits_\theta^{s_{m-1}}\prod_{j=1}^{m-1}(s_j-s_{j+1})^{a_j}\frac{\Gamma(a_m+1)\Gamma(a_{m+1}+1)}{\Gamma(a_m+a_{m+1}+2)}(s_m-\theta)^{a_m+a_{m+1}+1}ds_m\dots ds_1ds_1\\
  &=\frac{\Gamma(a_m+1)\Gamma(a_{m+1}+1)}{\Gamma(a_m+a_{m+1}+2)}\int_{\Delta^m_{\theta,t}}(s_m-\theta)^{a_m+a_{m+1}+1}\prod_{j=1}^{m-1}(s_j-s_{j+1})^{a_j}ds,
 \end{align*}
 where we applied \eqref{3IntegralHilfsFormel}. Defining $\tilde a_j:=a_j$ for $j=1,\dots,m-1$ and $\tilde a_m:=a_m+a_{m+1}+1$ and applying our assumption that the formula holds for $m$, we get
 \begin{align*}
 &\int_{\Delta^m_{\theta,t}}(s_m-\theta)^{a_m+a_{m+1}+1}\prod_{j=1}^{m-1}(s_j-s_{j+1})^{a_j}ds\\
 &\quad=\int_{\Delta^m_{\theta,t}}\prod_{j=1}^m(s_j-s_{j+1})^{\tilde a_j}ds\\
 &\quad=\frac{\prod_{l=1}^m\Gamma(\tilde a_l+1)}{\Gamma(\sum_{l=1}^m\tilde a_l+m+1)}(t-\theta)^{\sum_{l=1}^m\tilde a_l+m}\\
 &\quad=\frac{\Gamma(a_m+a_{m+1}+1+1)\prod_{l=1}^{m-1}\Gamma(a_l+1)}{\Gamma(\sum_{l=1}^{m-1}a_l+a_m+a_{m+1}+1+m+1)}(t-\theta)^{\sum_{l=1}^{m-1} a_l+a_m+a_{m+1}+1+m}\\
 &\quad=\frac{\Gamma(a_m+a_{m+1}+2)\prod_{l=1}^{m-1}\Gamma(a_l+1)}{\Gamma(\sum_{l=1}^{m+1}a_l+m+2)}(t-\theta)^{\sum_{l=1}^{m+1} a_l+m+1}.
 \end{align*}
 Plugging this in, we finally achieve
 \begin{align*}
  &\int_{\Delta^{m+1}_{\theta,t}}\prod_{j=1}^{m+1}(s_j-s_{j+1})^{a_j}ds\\
  &=\frac{\Gamma(a_m+1)\Gamma(a_{m+1}+1)}{\Gamma(a_m+a_{m+1}+2)}\frac{\Gamma(a_m+a_{m+1}+2)\prod_{l=1}^{m-1}\Gamma(a_l+1)}{\Gamma(\sum_{l=1}^{m+1}a_l+m+2)}(t-\theta)^{\sum_{l=1}^{m+1} a_l+m+1}\\
  &=\frac{\prod_{l=1}^{m+1}\Gamma(a_l+1)}{\Gamma(\sum_{l=1}^{m+1}a_l+(m+1)+1)}(t-\theta)^{\sum_{l=1}^{m+1}a_l+(m+1)},
 \end{align*}
 which proves the formula for $m+1$. By the principle of induction, this finishes the proof.
\end{proof}

The following result which is due to \cite[Theorem 1] {DPMN92} provides a compactness criterion for subsets of $L^{2}(\Omega)$ using Malliavin calculus.

\begin{thm}
\label{3VI_MCompactness}Let $\left\{ \left( \Omega ,\mathcal{A},P\right)
;H\right\} $ be a Gaussian probability space, that is $\left( \Omega ,%
\mathcal{A},P\right) $ is a probability space and $H$ a separable closed
subspace of Gaussian random variables of $L^{2}(\Omega )$, which generate
the $\sigma $-field $\mathcal{A}$. Denote by $\mathbf{D}$ the derivative
operator acting on elementary smooth random variables in the sense that%
\begin{equation*}
\mathbf{D}(f(h_{1},\ldots,h_{n}))=\sum_{i=1}^{n}\partial
_{i}f(h_{1},\ldots,h_{n})h_{i},\text{ }h_{i}\in H,f\in C_{b}^{\infty }(\mathbb{R%
}^{n}).
\end{equation*}%
Further let $\mathbb{D}^{1,2}$ be the closure of the family of elementary
smooth random variables with respect to the norm%
\begin{align*}
\left\Vert F\right\Vert _{1,2}:=\left\Vert F\right\Vert _{L^{2}(\Omega
)}+\left\Vert \mathbf{D}F\right\Vert _{L^{2}(\Omega ;H)}.
\end{align*}%
Assume that $C$ is a self-adjoint compact operator on $H$ with dense image.
Then for any $c>0$ the set
\begin{equation*}
\mathcal{G}=\left\{ G\in \mathbb{D}^{1,2}:\left\Vert G\right\Vert
_{L^{2}(\Omega )}+\left\Vert C^{-1} \mathbf{D} \,G\right\Vert _{L^{2}(\Omega ;H)}\leq
c\right\}
\end{equation*}%
is relatively compact in $L^{2}(\Omega )$.
\end{thm}

In order to formulate compactness criteria useful for our purposes, we need the following technical result which also can be found in \cite{DPMN92}.

\begin{lem}
\label{3VI_DaPMN} Let $v_{s},s\geq 0$ be the Haar basis of $L^{2}([0,T])$. For
any $0<\alpha <1/2$ define the operator $A_{\alpha }$ on $L^{2}([0,T])$ by%
\begin{equation*}
A_{\alpha }v_{s}=2^{k\alpha }v_{s}\text{, if }s=2^{k}+j\text{ }
\end{equation*}%
for $k\geq 0,0\leq j\leq 2^{k}$ and%
\begin{equation*}
A_{\alpha }1=1.
\end{equation*}%
Then for all $\beta $ with $\alpha <\beta <(1/2),$ there exists a constant $%
c_{1}$ such that%
\begin{equation*}
\left\Vert A_{\alpha }f\right\Vert \leq c_{1}\left\{ \left\Vert f\right\Vert
_{L^{2}([0,T])}+\left(\int_{0}^{T}\int_{0}^{T}\frac{\left|
f(t)-f(t^{\prime })\right|^2}{\left\vert t-t^{\prime }\right\vert
^{1+2\beta }}dt\,dt^{\prime }\right)^{1/2}\right\} .
\end{equation*}
\end{lem}

A direct consequence of Theorem \ref{3VI_MCompactness} and Lemma \ref{3VI_DaPMN} is now the following compactness criteria.

\begin{cor} \label{3VI_compactcrit}
Let a sequence of $\mathcal{F}_T$-measurable random variables $X_n\in\mathbb{D}^{1,2}$, $n=1,2,\dots$, be such that there exists a constant $C>0$ with
$$
\sup_n E[|X_n|^2] \leq C ,
$$
$$
\sup_n E \left[ \| \DM_t X_n \|_{L^2([0,T])}^2 \right] \leq C \,
$$
and there exists a $\beta \in (0,1/2)$ such that
$$
\sup_n \int_0^T \int_0^T \frac{E \left[ \| \DM_t X_n - \DM_{t'} X_n \|^2 \right]}{|t-t'|^{1+2\beta}} dtdt' <\infty
$$
where $\|\cdot\|$ denotes any matrix norm.

Then the sequence $X_n$, $n=1,2,\dots$, is relatively compact in $L^{2}(\Omega )$.
\end{cor}

\textbf{Acknowledgements} The financial support from the Norwegian Research Council within the ISP project 239019 ``FINance, INsurance, Energy, Weather and STOCHastics'' (FINEWSTOCH) and the project 250768/F20 ``Challenges in STOchastic CONtrol, INFormation and Applications'' (STOCONINF) is greatly acknowledged.

\end{document}